\documentclass{amsart}
\usepackage[english]{babel}
\usepackage{enumitem}

\usepackage{mathtools}
\usepackage{amssymb} 

\usepackage{caption}
\usepackage[labelfont=rm]{subcaption}
\usepackage{subfig}

\usepackage{tikz}
\usetikzlibrary{calc}

\newcommand{\coveringDiagram}{
\begin{scope}[yscale=1.2, xscale=1.1, every node/.style={scale=.95}]
\node (X) at (0,4) {$\X = (S,\omega)$};
\node (Xh) at (-4,3) {$\Xh=\rquo{\X}{\langle \tau_h\rangle}=(S_h,\omega_h)$};
\node (Xv) at (4,3) {$\Xv=\rquo{\X}{\langle \tau_v\rangle}=(S_v,\omega_v)$};
\node (Xhv) at (0,2) {$\Xhv=\rquo{\X}{\langle \tau_h,\tau_v\rangle}=(\tilde S,\tilde\omega)$};
\node (Wh) at (-4,1) {$\Wh=\rquo{\X}{\langle \tau_h,\iota\circ\tau_v\rangle}=(\mathbb{T}^2,q_h)$};
\node (Wv) at (4,1) {$\Wv=\rquo{\X}{\langle \tau_v,\iota\circ\tau_h\rangle}=(\mathbb{T}^2,q_v)$};
\node (W) at (0,0) {$\W=\rquo{\X}{\langle \tau_h,\tau_v,\iota\rangle}=(\mathbb{CP}^1,q)$};

\begin{scope}[every node/.style={scale=.75}]
\draw[->] (X) -- node[midway, above] {$\Ph$} (Xh);
\draw[->] (Xh) -- node[midway, right] {$\Pph$} (Wh);
\draw[->] (Wh) -- node[midway, below] {$\ph$} (W);
\draw[->] (X) -- node[midway, above] {$\Pv$} (Xv);
\draw[->] (Xv) -- node[midway, right] {$\Ppv$} (Wv);
\draw[->] (Wv) -- node[midway, below] {$\pv$} (W);
\draw[->] (X) -- node[midway, right] {$\Phv$} (Xhv);
\draw[->] (Xhv) -- node[midway, right] {$\phv$} (W);

\draw[->] (X) ..controls (-2.5,2.5) and (-2,1).. node[near end, right] {$\PP$} (W);

\end{scope}
\end{scope}

}

\newcommand{\rightPocket}[2][1]{
\begin{scope}[line width=.5, shift={(#2)}, scale=#1]
\draw (0, .5) arc (90:-90:.2 and .5);
\draw[densely dashed] (0, .5) arc (90:300:.2 and .5);
\draw (0, .5) -- (2, .5);
\draw (2, .5) -- (2,-.5);
\draw (2,-.5) -- (0,-.5);
\draw[line width=.25] (1, .5) arc (90:-90:.2 and .5);
\draw[line width=.25, densely dashed] (1, .5) arc (90:300:.2 and .5);
\end{scope}
}

\newcommand{\leftSaddleConnectionRightPocket}[2][1]{
\begin{scope}[line width=.5, shift={(#2)}, scale=#1]
\draw[line width=1] (0, .5) arc (90:-90:.2 and .5);
\draw[line width=1, densely dashed] (0, .5) arc (90:300:.2 and .5);
\end{scope}
}
\newcommand{\rightSaddleConnectionRightPocket}[2][1]{
\begin{scope}[line width=.5, shift={(#2)}, scale=#1]
\draw[line width=1] (2, .5) -- (2,-.5);
\end{scope}
}

\newcommand{\upperLeftNodeRightPocket}[3][1]{
\begin{scope}[line width=.5, shift={(#2)}, scale=#1, every node/.style={scale=.7}]
\draw (0, .5) node[circle,fill,inner sep=2, label={[label distance=-3]45:#3}]{};
\end{scope}
}
\newcommand{\lowerLeftNodeRightPocket}[3][1]{
\begin{scope}[line width=.5, shift={(#2)}, scale=#1, every node/.style={scale=.7}]
\draw (0,-.5) node[circle,fill,inner sep=2, label={[label distance=-3]-45:#3}]{};
\end{scope}
}
\newcommand{\upperRightNodeRightPocket}[3][1]{
\begin{scope}[line width=.5, shift={(#2)}, scale=#1, every node/.style={scale=.7}]
\draw (2, .5) node[circle,fill,inner sep=2, label={[label distance=-3]above:#3}]{};
\end{scope}
}
\newcommand{\lowerRightNodeRightPocket}[3][1]{
\begin{scope}[line width=.5, shift={(#2)}, scale=#1, every node/.style={scale=.7}]
\draw (2,-.5) node[circle,fill,inner sep=2, label={[label distance=-3]below:#3}]{};
\end{scope}
}

\newcommand{\leftSphere}[2][1]{
\begin{scope}[line width=.5, shift={(#2)}, scale=#1]
\draw ( 0,.5) ..controls (-.25,.5) and (-.75,1.5).. (-1.5,1.5);
\draw (-1.5,1.5) arc (90:180:1.5 and 1.75);
\draw (-3  ,-.25) arc (180:270:1.5);
\draw ( 0,-.5) ..controls (0,-.75) and (-.75,-1.75).. (-1.5,-1.75);
\draw[line width=.25] (-1.5,1.5) arc (90:-90:.6 and 3.25/2);
\draw[line width=.25, densely dashed] (-1.5,1.5) arc (90:300:.6 and 3.25/2);
\end{scope}
}

\newcommand{\pocketConfiguration}[2][1]{
\rightPocket[#1]{#2}
\lowerLeftNodeRightPocket[#1]{#2}{$P_{i}$}
\leftSaddleConnectionRightPocket[#1]{#2}
\upperRightNodeRightPocket[#1]{#2}{$P_{j_1}$}
\lowerRightNodeRightPocket[#1]{#2}{$P_{j_2}$}
\rightSaddleConnectionRightPocket[#1]{#2}
\leftSphere[#1]{#2}
}

\newcommand{\cylinder}[2][1]{
\begin{scope}[line width=.5, shift={(#2)}, scale=#1]
\draw (-1, .5) arc (90:-90:.2 and .5);
\draw[densely dashed] (-1, .5) arc (90:300:.2 and .5);
\draw (-1,-.5) -- ( 1,-.5);

\draw ( 1, .5) arc (90:-90:.2 and .5);
\draw[densely dashed] (1, .5) arc (90:300:.2 and .5);

\draw ( 1, .5) -- (-1, .5);

\draw[line width=.25] (0, .5) arc (90:-90:.2 and .5);
\draw[line width=.25, densely dashed] (0, .5) arc (90:300:.2 and .5);
\end{scope}
}

\newcommand{\leftSaddleConnectionCylinder}[2][1]{
\begin{scope}[shift={(-1*#1,0)}]
\leftSaddleConnectionRightPocket[#1]{#2}
\end{scope}
}
\newcommand{\rightSaddleConnectionCylinder}[2][1]{
\begin{scope}[shift={( 1*#1,0)}]
\leftSaddleConnectionRightPocket[#1]{#2}
\end{scope}
}

\newcommand{\upperLeftNodeCylinder}[3][1]{
\begin{scope}[line width=.5, shift={(#2)}, scale=#1, every node/.style={scale=.7}]
\draw (-1, .5) node[circle,fill,inner sep=2, label={[label distance=-3]45:#3}]{};
\end{scope}
}
\newcommand{\lowerLeftNodeCylinder}[3][1]{
\begin{scope}[line width=.5, shift={(#2)}, scale=#1, every node/.style={scale=.7}]
\draw (-1,-.5) node[circle,fill,inner sep=2, label={[label distance=-3]-45:#3}]{};
\end{scope}
}
\newcommand{\upperRightNodeCylinder}[3][1]{
\begin{scope}[line width=.5, shift={(#2)}, scale=#1, every node/.style={scale=.7}]
\draw (1, .5) node[circle,fill,inner sep=2, label={[label distance=-3]135:#3}]{};
\end{scope}
}
\newcommand{\lowerRightNodeCylinder}[3][1]{
\begin{scope}[line width=.5, shift={(#2)}, scale=#1, every node/.style={scale=.7}]
\draw (1,-.5) node[circle,fill,inner sep=2, label={[label distance=-3]-135:#3}]{};
\end{scope}
}

\newcommand{\rightSphere}[2][1]{
\begin{scope}[line width=.5, shift={(#2)}, scale=#1]
\draw ( 0  ,-.5) ..controls (.25,-.5) and (.5,-1.75).. (1.5,-1.75);
\draw ( 1.5,-1.75) arc (-90:0:1.5 and 1.75);
\draw ( 3  , 0    ) arc (0:90:1.5);
\draw ( 0  , .5) ..controls (0,.75) and (.75,1.5).. (1.5,1.5);
\draw[line width=.25] (1.5,-1.75) arc (-90:90:.6 and 3.25/2);
\draw[line width=.25, densely dashed] (1.5,-1.75) arc (270:70:.6 and 3.25/2);
\end{scope}
}

\newcommand{\dumbbellConfiguration}[2][1]{
\cylinder[#1]{#2}
\lowerLeftNodeCylinder[#1]{#2}{$P_{i_1}$}
\leftSaddleConnectionCylinder[#1]{#2}
\upperRightNodeCylinder[#1]{#2}{$P_{i_2}$}
\rightSaddleConnectionCylinder[#1]{#2}

\begin{scope}[shift={(-1*#1,0)}]
\leftSphere[#1]{#2}
\end{scope}

\begin{scope}[shift={( 1*#1,0)}]
\rightSphere[#1]{#2}
\end{scope}

}

\newcommand{\hole}[2][1]{
\begin{scope}[line width=.5, shift={(#2)}, scale=#1]
\draw[line width=.5, rounded corners=30*#1] (-.88,.08) -- (0,-.48) -- (.88,.08);
\draw[line width=.5, rounded corners=22*#1] (-.72,0) -- (0,.36) -- (.72,0);
\end{scope}
}

\newcommand{\leftPocket}[2][1]{
\begin{scope}[line width=.5, shift={(#2)}, scale=#1]
\draw (0, .5) arc (90:-90:.2 and .5);
\draw[densely dashed] (0, .5) arc (90:300:.2 and .5);

\draw (0, .5) -- (-2,.5) -- (-2,-.5) -- (0,-.5);

\draw[line width=.25] (-1, .5) arc (90:-90:.2 and .5);
\draw[line width=.25, densely dashed] (-1, .5) arc (90:300:.2 and .5);
\end{scope}
}

\newcommand{\rightSaddleConnectionLeftPocket}[2][1]{
\begin{scope}[line width=.5, shift={(#2)}, scale=#1]
\draw[line width=1] (0, .5) arc (90:-90:.2 and .5);
\draw[line width=1, densely dashed] (0, .5) arc (90:300:.2 and .5);
\end{scope}
}
\newcommand{\leftSaddleConnectionLeftPocket}[2][1]{
\begin{scope}[line width=.5, shift={(#2)}, scale=#1]
\draw[line width=1] (-2, .5) -- (-2,-.5);
\end{scope}
}

\newcommand{\upperRightNodeLeftPocket}[3][1]{
\begin{scope}[line width=.5, shift={(#2)}, scale=#1, every node/.style={scale=.7}]
\draw (0, .5) node[circle,fill,inner sep=2, label={[label distance=-3]135:#3}]{};
\end{scope}
}

\newcommand{\upperLeftNodeLeftPocket}[3][1]{
\begin{scope}[line width=.5, shift={(#2)}, scale=#1, every node/.style={scale=.7}]
\draw (-2, .5) node[circle,fill,inner sep=2, label={[label distance=-3]above:#3}]{};
\end{scope}
}
\newcommand{\lowerLeftNodeLeftPocket}[3][1]{
\begin{scope}[line width=.5, shift={(#2)}, scale=#1, every node/.style={scale=.7}]
\draw (-2,-.5) node[circle,fill,inner sep=2, label={[label distance=-3]below:#3}]{};
\end{scope}
}

\newcommand{\centralTorus}[2][1]{
\begin{scope}[line width=.5, shift={(#2)}, scale=#1]
\draw (1.5,.5) ..controls (1.25,.5) and (1,1.25).. (0,1.25);
\draw (1.5,-.5) ..controls (1.5,-.75) and (.75,-1.25).. (0,-1.25);
\draw (-1.5,-.5) ..controls (-1.25,-.5) and (-1,-1.25).. (0,-1.25);
\draw (-1.5,.5) ..controls (-1.5,.75) and (-.75,1.25).. (0,1.25);
\end{scope}
\begin{scope}[shift={(#2)}, xscale=.85, yscale=1.5]
\hole[#1]{(0,0)}
\end{scope}
}

\newcommand{\liftPocketRZero}[2][1]{
\begin{scope}[shift={(1.5*#1,0)}]
\rightPocket[#1]{#2}
\lowerLeftNodeRightPocket[#1]{#2}{$P_{i}^0$}
\leftSaddleConnectionRightPocket[#1]{#2}
\upperRightNodeRightPocket[#1]{#2}{$P_{j_1}^0$}
\lowerRightNodeRightPocket[#1]{#2}{$P_{j_2}^0$}
\rightSaddleConnectionRightPocket[#1]{#2}
\end{scope}
\begin{scope}[shift={(-1.5*#1,0)}]
\leftPocket[#1]{#2}
\upperRightNodeLeftPocket[#1]{#2}{$P_{i}^1$}
\rightSaddleConnectionLeftPocket[#1]{#2}
\lowerLeftNodeLeftPocket[#1]{#2}{$P_{j_1}^1$}
\upperLeftNodeLeftPocket[#1]{#2}{$P_{j_2}^1$}
\leftSaddleConnectionLeftPocket[#1]{#2}
\end{scope}
\centralTorus[#1]{#2}
}

\newcommand{\centralRightNodeRightPocket}[3][1]{
\begin{scope}[line width=.5, shift={(#2)}, scale=#1, every node/.style={scale=.7}]
\draw (2, 0) node[circle,fill,inner sep=1, label={[label distance=-3]right:#3}]{};
\end{scope}
}

\newcommand{\leftTorus}[2][1]{
\begin{scope}[line width=.5, shift={(#2)}, scale=#1]
\draw (0,.5) ..controls (0,.6) and (-.75,.9).. (-1.5,.9);
\draw (-1.5,.9) arc (90:180:1.5 and 1);
\draw (-3  ,-.1) arc (180:270:1.5 and .8);
\draw (0,-.5) ..controls (0,-.6) and (-.75,-.9).. (-1.5,-.9);
\end{scope}
\begin{scope}[shift={(#2)}, scale=.75]
\hole[#1]{(-2*#1,0)}
\end{scope}
}

\newcommand{\liftPocketROne}[2][1]{
\begin{scope}[yscale=2]
\rightPocket[#1]{#2}
\lowerLeftNodeRightPocket[#1]{#2}{$P_{i}^0$}
\upperLeftNodeRightPocket[#1]{#2}{$P_{i}^1$}
\leftSaddleConnectionRightPocket[#1]{#2}
\upperRightNodeRightPocket[#1]{#2}{$P_{j_1}^0$}
\centralRightNodeRightPocket[#1]{#2}{$P_{j_2}^\times$}
\lowerRightNodeRightPocket[#1]{#2}{$P_{j_1}^1$}
\rightSaddleConnectionRightPocket[#1]{#2}
\leftTorus[#1]{#2}
\end{scope}
}

\newcommand{\wrappingTorus}[2][1]{
\begin{scope}[line width=.5, shift={(#2)}, scale=#1]
\draw (-1,.5) ..controls (-1.125,.5) and (-1.25,.75).. (-1.25,1);
\draw (-1,-.5) ..controls (-1,-1) and (-1.75,-1).. (-1.75,.25);
\draw (1   , .5 ) ..controls (1,.75) and (1.125,.75).. (1.25,1);
\draw (1  ,-.5 ) ..controls (1.1,-.5) and (1.15,-.75).. (1.3,-.75);
\draw (1.3,-.75) ..controls (1.4,-.75) and (1.75,-.75).. (1.75,.25);
\draw (1.75,.25) arc (0:180:1.75 and 2);
\draw (0,1.25) arc (90:46:1.5 and 2);
\draw (0,1.25) arc (90:138:1.5 and 2);
\end{scope}
}

\newcommand{\upperCentralNodeCylinder}[3][1]{
\begin{scope}[line width=.5, shift={(#2)}, scale=#1, every node/.style={scale=.7}]
\draw (0, .5) node[circle,fill,inner sep=1, label={[label distance=-3]above:#3}]{};
\end{scope}
}
\newcommand{\lowerCentralNodeCylinder}[3][1]{
\begin{scope}[line width=.5, shift={(#2)}, scale=#1, every node/.style={scale=.7}]
\draw (0,-.5) node[circle,fill,inner sep=1, label={[label distance=-3]below:#3}]{};
\end{scope}
}

\newcommand{\liftPocketRTwo}[2][1]{
\begin{scope}[shift={(-1*#1,0)}]
\cylinder[#1]{#2}
\leftSaddleConnectionCylinder[#1]{#2}
\end{scope}
\begin{scope}[shift={(1*#1,0)}]
\cylinder[#1]{#2}
\rightSaddleConnectionCylinder[#1]{#2}
\end{scope}
\begin{scope}[xscale=2]
\upperCentralNodeCylinder[#1]{#2}{$P_{j_1}^\times$}
\lowerCentralNodeCylinder[#1]{#2}{$P_{j_2}^\times$}
\wrappingTorus[#1]{#2}
\end{scope}
\begin{scope}[line width=.5, shift={(#2)}, scale=#1, every node/.style={scale=.7}]
\draw (-2,-.5) node[circle,fill,inner sep=2, label={[label distance=-6*#1*#1]left:$P_{i}^0$}]{};
\draw (2,.5) node[circle,fill,inner sep=2, label={[label distance=-6*#1*#1]right:$P_{i}^1$}]{};
\end{scope}
}

\newcommand{\smallLeftSphere}[2][1]{
\begin{scope}[line width=.5, shift={(#2)}, scale=#1]
\draw (0,.5) ..controls (-.25,.5) and (-.25,1).. (-.75,1);
\draw (-.75,1) arc (90:270:1);
\draw (0,-.5) ..controls (0,-.75) and (-.5,-1).. (-.75,-1);

\draw[line width=.1*#1] (-.75,1) arc (90:-90:.3 and 1);
\draw[line width=.1*#1, densely dashed] (-.75,1) arc (90:300:.3 and 1);
\end{scope}
}

\newcommand{\smallRightSphere}[2][1]{
\begin{scope}[line width=.5, shift={(#2)}, scale=#1]
\draw (0,-.5) ..controls (.25,-.5) and (.25,-1).. (.75,-1);
\draw (.75,1) arc (90:-90:1);
\draw (0,.5) ..controls (0,.75) and (.5,1).. (.75,1);

\draw[line width=.1*#1] (.75,1) arc (90:-90:.3 and 1);
\draw[line width=.1*#1, densely dashed] (.75,1) arc (90:300:.3 and 1);
\end{scope}
}

\newcommand{\liftDumbbellRZero}[2][1]{
\begin{scope}[shift={(2.5*#1,0)}]
\cylinder[#1]{#2}
\lowerLeftNodeCylinder[#1]{#2}{$P_{i_1}^0$}
\leftSaddleConnectionCylinder[#1]{#2}
\upperRightNodeCylinder[#1]{#2}{$P_{i_2}^0$}
\rightSaddleConnectionCylinder[#1]{#2}
  \begin{scope}[shift={(1*#1,0)}]
  \smallRightSphere[#1]{#2}
  \end{scope}
\end{scope}
\begin{scope}[shift={(-2.5*#1,0)}]
\cylinder[#1]{#2}
\upperRightNodeCylinder[#1]{#2}{$P_{i_1}^1$}
\rightSaddleConnectionCylinder[#1]{#2}
\lowerLeftNodeCylinder[#1]{#2}{$P_{i_2}^1$}
\leftSaddleConnectionCylinder[#1]{#2}
  \begin{scope}[shift={(-1*#1,0)}]
  \smallLeftSphere[#1]{#2}
  \end{scope}
\end{scope}
\centralTorus[#1]{#2}
}

\newcommand{\smallRRightSphere}[2][1]{
\begin{scope}[line width=.5, shift={(#2)}, scale=#1]
\draw (0,-.5) ..controls (0,-.6) and (.75,-.9).. (1.25,-.9);
\draw (1.25,.9) arc (90:-90:1.75 and .9);
\draw (0, .5) ..controls (0, .6) and (.75, .9).. (1.25, .9);

\draw[line width=.1*#1] (1.25,.9) arc (90:-90:.6 and .9);
\draw[line width=.1*#1, densely dashed] (1.25,.9) arc (90:300:.6 and .9);
\end{scope}
}

\newcommand{\liftDumbbellROne}[2][1]{
\begin{scope}[yscale=2]
\cylinder[#1]{#2}
\lowerLeftNodeCylinder[#1]{#2}{$P_{i_1}^0$}
\upperLeftNodeCylinder[#1]{#2}{$P_{i_1}^1$}
\leftSaddleConnectionCylinder[#1]{#2}
\upperRightNodeCylinder[#1]{#2}{$P_{i_2}^0$}
\lowerRightNodeCylinder[#1]{#2}{$P_{i_2}^1$}
\rightSaddleConnectionCylinder[#1]{#2}
\begin{scope}[shift={(-1*#1,0)}]
\leftTorus[#1]{#2}
\end{scope}
\begin{scope}[shift={( 1*#1,0)}]
\smallRRightSphere[#1]{#2}
\end{scope}
\end{scope}
}

\newcommand{\leftSSphere}[2][1]{
\begin{scope}[line width=.5, shift={(#2)}, scale=#1]
\draw (-1,1.5) ..controls (-1,1.75) and (-1.75,2).. (-2.5,2);
\draw (-2.5,2) arc (90:270:1.75 and 2);
\draw (-1,-1.5) ..controls (-1,-1.75) and (-1.75,-2).. (-2.5,-2);

\draw[line width=.25] (-2.5,2) arc (90:-90:.6 and 2);
\draw[line width=.25, densely dashed] (-2.5,2) arc (90:300:.6 and 2);

\draw (-1, .5) ..controls (-1.2, .5) and (-1.4, .25).. (-1.5,0);
\draw (-1,-.5) ..controls (-1.2,-.5) and (-1.4,-.25).. (-1.55,.1);
\end{scope}
}

\newcommand{\rightSSphere}[2][1]{
\begin{scope}[line width=.5, shift={(#2)}, scale=#1]
\draw (1,1.5) ..controls (1.25,1.5) and (1.5,2).. (2.5,2);
\draw (2.5,2) arc (90:-90:1.75 and 2);
\draw (1,-1.5) ..controls (1.25,-1.5) and (1.5,-2).. (2.5,-2);

\draw[line width=.25] (2.5,2) arc (90:-90:.6 and 2);
\draw[line width=.25, densely dashed] (2.5,2) arc (90:300:.6 and 2);

\draw (1, .5) ..controls (1, .3) and (1.4, .2).. (1.47,0);
\draw (1,-.5) ..controls (1,-.3) and (1.4,-.2).. (1.54,.1);
\end{scope}
}

\newcommand{\liftDumbbellRTwo}[2][1]{
\begin{scope}[shift={(0,-1*#1)}]
\cylinder[#1]{#2}
\lowerLeftNodeCylinder[#1]{#2}{$P_{i_1}^0$}
\leftSaddleConnectionCylinder[#1]{#2}
\rightSaddleConnectionCylinder[#1]{#2}
\end{scope}
\begin{scope}[shift={(0, 1*#1)}]
\cylinder[#1]{#2}
\upperLeftNodeCylinder[#1]{#2}{$P_{i_1}^1$}
\leftSaddleConnectionCylinder[#1]{#2}
\rightSaddleConnectionCylinder[#1]{#2}
\end{scope}
\begin{scope}[line width=.5, shift={(#2)}, scale=#1, every node/.style={scale=.7}]
\draw (1, -.5) node[circle,fill,inner sep=2, label={[label distance=-3]right:$P_{i_2}^0$}]{};
\draw (1, .5) node[circle,fill,inner sep=2, label={[label distance=-3]right:$P_{i_2}^1$}]{};
\end{scope}
\leftSSphere[#1]{#2}
\rightSSphere[#1]{#2}
}

\newcommand{\leftHandle}[2][1]{
\begin{scope}[line width=.5, shift={(#2)}, scale=#1]
\draw (0,0) node[circle,fill,inner sep=2]{};
\draw[line width=1] (0,0) arc (90:-90:.2 and .5);
\draw[densely dashed, line width=1] (0, 0) arc (90:300:.2 and .5);
\draw[line width=1] (0,0) arc (-90:90:.2 and .5);
\draw[densely dashed, line width=1] (0, 0) arc (270:60:.2 and .5);

\draw (0,1) ..controls (-.3 ,1  ) and (-.3,1.1).. (-.6,1.1)
            ..controls (-1.2,1.1) and (-1.5,.5  ).. (-1.5,0  );
\draw (0,-1) ..controls (-.3 ,-1  ) and (-.3,-1.1).. (-.6,-1.1)
             ..controls (-1.2,-1.1) and (-1.5,-.5  ).. (-1.5, 0  );

\draw (-.9,.2) arc (90:270:.2);
\draw (-.9,.2) ..controls (-.8, .2).. (0,0);
\draw (-.9,-.2) ..controls (-.8,-.2).. (0,0);

\draw (-1.5,0) node[circle,fill,inner sep=1]{};
\draw (-1.1,0) node[circle,fill,inner sep=1]{};
\draw (-1.5,0) arc (180:360:.2 and .1);
\draw[densely dashed] (-1.5,0) arc (180:-30:.2 and .1);

\draw[line width=.25] (-.75,1.08) arc (90:-90:.18 and .45);
\draw[line width=.25, densely dashed] (-.75,1.08) arc (90:300:.18 and .45);
\draw[line width=.25] (-.75,-1.08) arc (-90:90:.18 and .45);
\draw[line width=.25, densely dashed] (-.75,-1.08) arc (-90:-300:.18 and .45);
\end{scope}
}

\newcommand{\rightHandle}[2][1]{
\begin{scope}[line width=.5, shift={(#2)}, scale=#1]
\draw (0,0) node[circle,fill,inner sep=2]{};
\draw[line width=1] (0,0) arc (90:-90:.2 and .5);
\draw[densely dashed, line width=1] (0, 0) arc (90:300:.2 and .5);
\draw[line width=1] (0,0) arc (-90:90:.2 and .5);
\draw[densely dashed, line width=1] (0, 0) arc (270:60:.2 and .5);

\begin{scope}[xscale=-1]
\draw (0,1) ..controls (-.3 ,1  ) and (-.3,1.1).. (-.6,1.1)
            ..controls (-1.2,1.1) and (-1.5,.5  ).. (-1.5,0  );
\draw (0,-1) ..controls (-.3 ,-1  ) and (-.3,-1.1).. (-.6,-1.1)
             ..controls (-1.2,-1.1) and (-1.5,-.5  ).. (-1.5, 0  );

\draw (-.9,.2) arc (90:270:.2);
\draw (-.9,.2) ..controls (-.8, .2).. (0,0);
\draw (-.9,-.2) ..controls (-.8,-.2).. (0,0);

\draw (-1.5,0) node[circle,fill,inner sep=1]{};
\draw (-1.1,0) node[circle,fill,inner sep=1]{};
\draw (-1.5,0) arc (180:360:.2 and .1);
\draw[densely dashed] (-1.5,0) arc (180:-30:.2 and .1);
\end{scope}

\draw[line width=.25] (.75,1.08) arc (90:-90:.18 and .45);
\draw[line width=.25, densely dashed] (.75,1.08) arc (90:300:.18 and .45);
\draw[line width=.25] (.75,-1.08) arc (-90:90:.18 and .45);
\draw[line width=.25, densely dashed] (.75,-1.08) arc (-90:-300:.18 and .45);
\end{scope}
}

\newcommand{\fourTimes}[1]
{ #1
\begin{scope}[scale=-1] #1 \end{scope}
\begin{scope}[yscale=-1] #1 \end{scope}
\begin{scope}[xscale=-1] #1 \end{scope}
}

\newcommand{\centralSurface}[2][1]{
\begin{scope}[shift={(#2)}, yscale=1.25]
\hole[.7*#1]{0,0}
\hole[.7*#1]{1.25*#1,.6*#1}
\hole[.7*#1]{1.25*#1,-.6*#1}
\hole[.7*#1]{-1.25*#1,.6*#1}
\hole[.7*#1]{-1.25*#1,-.6*#1}
\end{scope}

\begin{scope}[line width=.5, shift={(#2)}, scale=#1]
\fourTimes{
  \draw (0,1.2) ..controls (.5,1.2) and (.75,1.5).. (1.25,1.5)
              ..controls (2,1.5) and (2.2,1).. (2.5,1);
}
\end{scope}
}

\newcommand{\liftliftPocketRZero}[2][1]{
\begin{scope}[shift={(#2)}]
\leftHandle[#1]{-2.5*#1,0}
\rightHandle[#1]{2.5*#1,0}
\end{scope}

\centralSurface[#1]{#2}
}

\newcommand{\rightDoubleHandle}[2][1]{
\begin{scope}[line width=.5, shift={(#2)}, scale=#1]
\draw (0,1) node[circle,fill,inner sep=2]{};
\draw (0,-1) node[circle,fill,inner sep=2]{};

\draw[line width=1] (0,1) arc (90:-90:.2 and 1);
\draw[densely dashed, line width=1] (0,1) arc (90:300:.2 and 1);
\draw[densely dashed, line width=1] (0,1) arc (90:-90:.1 and 1);
\draw[densely dashed, line width=1] (0,1) arc (90:300:.1 and 1);

\draw (0,1) ..controls (.01,1) and (.5,1.2).. (.9,1.2)
            ..controls (1.25,1.2) and (1.25,.5).. (1.25,0) node[circle,fill,inner sep=1]{};
\draw (0,-1) ..controls (.01,-1) and (.5,-1.2).. (.9,-1.2)
            ..controls (1.25,-1.2) and (1.25,-.5).. (1.25,0);
\draw[densely dashed, line width=.5] (1.01,1.15) node[circle,fill,inner sep=1]{} ..controls (.95,1.2) and (.9,1.1).. (.9,1) -- (.9,0) node[circle,fill,inner sep=1]{} -- (.9,-1) ..controls (.9,-1.1) and (.95,-1.2).. (1.01,-1.15) node[circle,fill,inner sep=1]{};

\draw[line width=.5, rounded corners=7*#1] (0,1) -- (.9,.9) -- (1,1.1);
\draw[line width=.5, rounded corners=4*#1] (0,1) -- (.8,1.1) -- (.9,1);

\draw[densely dashed, line width=.5, rounded corners=7*#1] (0,-1) -- (.9,-1.1) -- (1,-.9);
\draw[densely dashed, line width=.5, rounded corners=4*#1] (0,-1) -- (.8,-.9) -- (.9,-1);

\end{scope}
}

\newcommand{\leftSurface}[2][1]{
\begin{scope}[shift={(#2)}, yscale=1.25]
\hole[.7*#1]{-1.25*#1,.6*#1}
\hole[.7*#1]{-1.25*#1,-.6*#1}
\hole[.7*#1]{-2.5*#1,0}
\hole[.7*#1]{-3.75*#1,-.6*#1}
\hole[.7*#1]{-3.75*#1,.6*#1}
\hole[.7*#1]{-5*#1,0}
\end{scope}

\begin{scope}[shift={(-2.5*#1,0)}]
\begin{scope}[line width=.5, shift={(#2)}, scale=#1]
\fourTimes{
  \draw (0,1.2) ..controls (.5,1.2) and (.75,1.5).. (1.25,1.5);
}
\draw (1.25,1.5) ..controls (1.75,1.5) and (2.5,1.25).. (2.5,1);
\draw (1.25,-1.5) ..controls (1.75,-1.5) and (2.5,-1.25).. (2.5,-1);
\draw (-1.25,1.5) ..controls (-2,1.5) and (-2.25,.75).. (-2.75,.75);
\draw (-2.75,.75) arc (90:270:1 and .75);
\draw (-1.25,-1.5) ..controls (-2,-1.5) and (-2.25,-.75).. (-2.75,-.75);
\end{scope}
\end{scope}
}

\newcommand{\liftliftPocketROne}[2][1]{
\rightDoubleHandle[#1]{#2}
\leftSurface[#1]{#2}
}

\newcommand{\twoCylinders}[2][1]{
\begin{scope}[line width=.5, shift={(#2)}, scale=#1]
\draw (1,0) node[circle,fill,inner sep=2]{};
\draw (-1,0) node[circle,fill,inner sep=2]{};

\draw[line width=1] (1,0) arc (90:-90:.2 and .5);
\draw[densely dashed, line width=1] (1, 0) arc (90:300:.2 and .5);
\draw[line width=1] (1,0) arc (-90:90:.2 and .5);
\draw[densely dashed, line width=1] (1, 0) arc (270:60:.2 and .5);

\draw[line width=1] (-1,0) arc (90:-90:.2 and .5);
\draw[densely dashed, line width=1] (-1, 0) arc (90:300:.2 and .5);
\draw[line width=1] (-1,0) arc (-90:90:.2 and .5);
\draw[densely dashed, line width=1] (-1, 0) arc (270:60:.2 and .5);

\fourTimes{
\draw (-1,1) ..controls (-.5,1) and (-.5,1.1).. (0,1.1);
\draw (-1,0) ..controls (-.5,0) and (-.5,.1).. (0,.1);
}

\draw[line width=.25] (0,1.1) arc (90:-90:.2 and .5);
\draw[line width=.25, densely dashed] (0,1.1) arc (90:300:.2 and .5);
\draw[line width=.25] (0,-1.1) arc (-90:90:.2 and .5);
\draw[line width=.25, densely dashed] (0,-1.1) arc (-90:-300:.2 and .5);
\end{scope}
}

\newcommand{\leftSmallSurface}[2][1]{
\begin{scope}[yscale=1.25, shift={(#2)}]
\hole[.7*#1]{-1.25*#1,.45*#1}
\hole[.7*#1]{-1.25*#1,-.45*#1}
\end{scope}

\begin{scope}[line width=.5, shift={(#2)}, scale=#1]
\draw (0,1) ..controls (-.5,1) and (-.75,1.25).. (-1.25,1.25);
\draw (0,-1) ..controls (-.5,-1) and (-.75,-1.25).. (-1.25,-1.25);

\draw (-1.25,1.25) arc (90:180:1.2 and .75);
\draw (-1.25,-1.25) arc (-90:-180:1.2 and .75);

\draw (-2.45,.5) ..controls (-2.45,.2) and (-2.25,.2).. (-2.25,0);
\draw (-2.45,-.5) ..controls (-2.45,-.2) and (-2.25,-.2).. (-2.25,0);
\end{scope}
}

\newcommand{\rightSmallSurface}[2][1]{
\begin{scope}[shift={(#2)}, xscale=-1]
\leftSmallSurface[#1]{0,0}
\end{scope}
}

\newcommand{\liftliftDumbbellRZero}[2][1]{
\begin{scope}[shift={(#2)}]
\leftSmallSurface[#1]{(-4.5*#1,0)}
\twoCylinders[#1]{(-3.5*#1,0)}
\centralSurface[#1]{0,0}
\twoCylinders[#1]{(3.5*#1,0)}
\rightSmallSurface[#1]{(4.5*#1,0)}
\end{scope}
}

\newcommand{\twoDoubleCylinders}[2][1]{
\begin{scope}[line width=.5, shift={(#2)}, scale=#1]
\draw (1,1) node[circle,fill,inner sep=2]{};
\draw (1,-1) node[circle,fill,inner sep=2]{};

\draw[line width=1] (1,1) arc (90:-90:.25 and 1);
\draw[densely dashed, line width=1] (1,1) arc (90:300:.25 and 1);
\draw[densely dashed, line width=1] (1,1) arc (90:-90:.1 and 1);
\draw[densely dashed, line width=1] (1,1) arc (90:300:.1 and 1);

\draw (-1,1) node[circle,fill,inner sep=2]{};
\draw (-1,-1) node[circle,fill,inner sep=2]{};

\draw[line width=1] (-1,1) arc (90:-90:.25 and 1);
\draw[densely dashed, line width=1] (-1,1) arc (90:300:.25 and 1);
\draw[densely dashed, line width=1] (-1,1) arc (90:-90:.1 and 1);
\draw[densely dashed, line width=1] (-1,1) arc (90:300:.1 and 1);

\draw (-1,1) ..controls (-.75,1) and (-.45,.9).. (.1,.9)
             ..controls (.55,.9) and (.75,1).. (1,1);
\draw (-1,1) ..controls (-.75,1) and (-.55,1.1).. (-.1,1.1)
             ..controls (.45,1.1) and (.75,1).. (1,1);
             
\draw (-1,-1) ..controls (-.75,-1) and (-.45,-1.1).. (.1,-1.1)
             ..controls (.55,-1.1) and (.75,-1).. (1,-1);
\draw[densely dashed] (-1,-1) ..controls (-.75,-1) and (-.55,-.9).. (-.1,-.9)
             ..controls (.45,-.9) and (.75,-1).. (1,-1);

\draw[line width=.25] (.1,.9) arc (90:-90:.1 and 1);
\draw[densely dashed, line width=.25] (.1,.9) arc (90:270:.05 and 1);
\draw[densely dashed, line width=.25] (-.1,1.1) arc (90:270:.1 and 1);
\draw[densely dashed, line width=.25] (-.1,1.1) arc (90:-90:.05 and 1);
\end{scope}
}

\newcommand{\anotherRightSmallSurface}[2][1]{
\begin{scope}[shift={(#2)}, yscale=1.25]
\hole[.7*#1]{2.5*#1,0}
\hole[.7*#1]{1.25*#1,.6*#1}
\hole[.7*#1]{1.25*#1,-.6*#1}
\end{scope}

\begin{scope}[line width=.5, shift={(#2)}, scale=#1]
\draw (1.25,1.5) ..controls (.55,1.5) and (0,1.1).. (0,1);
\draw (1.25,-1.5) ..controls (.55,-1.5) and (0,-1.1).. (0,-1);
\draw (1.25,1.5) ..controls (2,1.5) and (2.25,.75).. (2.75,.75);
\draw (2.75,.75) arc (90:-90:1 and .75);
\draw (1.25,-1.5) ..controls (2,-1.5) and (2.25,-.75).. (2.75,-.75);
\end{scope}
}

\newcommand{\liftliftDumbbellROne}[2][1]{
\begin{scope}[shift={(#2)}]
\twoDoubleCylinders[#1]{0,0}
\leftSurface[#1]{-1*#1,0}
\anotherRightSmallSurface[#1]{1*#1,0}
\end{scope}
}

\newcommand{\liftliftliftPocketRZero}[2][1]{
\begin{scope}[shift={(#2)}]
\begin{scope}[yscale=1.25]
\hole[.7*#1]{1*#1,1.5*#1}
\hole[.7*#1]{2*#1,.5*#1}
\hole[.7*#1]{-1*#1,1.5*#1}
\hole[.7*#1]{-2*#1,.5*#1}
\hole[.7*#1]{0,0}
\hole[.7*#1]{1*#1,-1.5*#1}
\hole[.7*#1]{2*#1,-.5*#1}
\hole[.7*#1]{-1*#1,-1.5*#1}
\hole[.7*#1]{-2*#1,-.5*#1}
\end{scope}
\begin{scope}[shift={(3*#1,1.5*#1)}, rotate=30]
  \rightHandle[#1]{0,0}
  \node[inner sep=0] (A) at (0,1*#1) {};
  \node[inner sep=0] (a) at (0,-1*#1) {};
\end{scope}
\begin{scope}[shift={(3*#1,-1.5*#1)}, rotate=-30]
  \rightHandle[#1]{0,0}
  \node[inner sep=0] (B) at (0,-1*#1) {};
  \node[inner sep=0] (b) at (0,1*#1) {};
\end{scope}
\begin{scope}[shift={(-3*#1,1.5*#1)}, rotate=-30]
  \leftHandle[#1]{0,0}
  \node[inner sep=0] (C) at (0,1*#1) {};
  \node[inner sep=0] (c) at (0,-1*#1) {};
\end{scope}
\begin{scope}[shift={(-3*#1,-1.5*#1)}, rotate=30]
  \leftHandle[#1]{0,0}
  \node[inner sep=0] (D) at (0,-1*#1) {};
  \node[inner sep=0] (d) at (0,1*#1) {};
\end{scope}\end{scope}

\begin{scope}[shift={(#2)}, scale=#1]
\fourTimes{
\draw (0,2.25) ..controls (.5,2.25) and (.5,2.5).. (1,2.5);
}
\draw(1,2.5) to[out=0, in=210, looseness=.75] (A);
\draw(1,-2.5) to[out=0, in=-210, looseness=.75] (B);
\draw(-1,2.5) to[out=180, in=-30, looseness=.75] (C);
\draw(-1,-2.5) to[out=180, in=30, looseness=.75] (D);

\draw (3,0) to[out=90, in=210, looseness=.75] (a);
\draw (3,0) to[out=-90, in=-210, looseness=.75] (b);
\draw (-3,0) to[out=90, in=-30, looseness=.75] (c);
\draw (-3,0) to[out=-90, in=30, looseness=.75] (d);
\end{scope}
}

\newcommand{\leftDoubleHandle}[2][1]{
\begin{scope}[line width=.5, shift={(#2)}, scale=#1]
\draw (0,1) node[circle,fill,inner sep=2]{};
\draw (0,-1) node[circle,fill,inner sep=2]{};

\draw[line width=1] (0,1) arc (90:-90:.2 and 1);
\draw[densely dashed, line width=1] (0,1) arc (90:300:.2 and 1);
\draw[densely dashed, line width=1] (0,1) arc (90:-90:.1 and 1);
\draw[densely dashed, line width=1] (0,1) arc (90:300:.1 and 1);

\begin{scope}[xscale=-1]
\draw (0,1) ..controls (.01,1) and (.5,1.2).. (.9,1.2)
            ..controls (1.25,1.2) and (1.25,.5).. (1.25,0) node[circle,fill,inner sep=1]{};
\draw (0,-1) ..controls (.01,-1) and (.5,-1.2).. (.9,-1.2)
            ..controls (1.25,-1.2) and (1.25,-.5).. (1.25,0);
\draw[densely dashed, line width=.5] (1.01,1.15) node[circle,fill,inner sep=1]{} ..controls (.95,1.2) and (.9,1.1).. (.9,1) -- (.9,0) node[circle,fill,inner sep=1]{} -- (.9,-1) ..controls (.9,-1.1) and (.95,-1.2).. (1.01,-1.15) node[circle,fill,inner sep=1]{};

\draw[line width=.5, rounded corners=7*#1] (0,1) -- (.9,.9) -- (1,1.1);
\draw[line width=.5, rounded corners=4*#1] (0,1) -- (.8,1.1) -- (.9,1);

\draw[densely dashed, line width=.5, rounded corners=7*#1] (0,-1) -- (.9,-1.1) -- (1,-.9);
\draw[densely dashed, line width=.5, rounded corners=4*#1] (0,-1) -- (.8,-.9) -- (.9,-1);
\end{scope}

\end{scope}
}

\newcommand{\liftliftliftPocketROne}[2][1]{
\begin{scope}[shift={(#2)}]
\begin{scope}[yscale=1.25]
\hole[.7*#1]{0,0}
\hole[.7*#1]{2*#1,0}
\hole[.7*#1]{-2*#1,0}
\hole[.7*#1]{3*#1,.6*#1}
\hole[.7*#1]{-3*#1,.6*#1}
\hole[.7*#1]{1*#1,1*#1}
\hole[.7*#1]{-1*#1,1*#1}
\hole[.7*#1]{3*#1,-.6*#1}
\hole[.7*#1]{-3*#1,-.6*#1}
\hole[.7*#1]{1*#1,-1*#1}
\hole[.7*#1]{-1*#1,-1*#1}
\end{scope}

\leftDoubleHandle[#1]{-4*#1,0}
\rightDoubleHandle[#1]{4*#1,0}

\begin{scope}[line width=.5, shift={(#2)}, scale=#1]
\fourTimes{
  \draw (0,1.6) ..controls (.45,1.6) and (.55,1.8).. (1,1.8)
                ..controls (2,1.8) and (2,1.25).. (3,1.25)
                ..controls (3.5,1.25) and (4,1.25).. (4,1);
}
\end{scope}

\end{scope}
}

\newcommand{\liftliftliftDumbbellRZero}[2][1]{
\begin{scope}[shift={(#2)}]
\begin{scope}[yscale=1.25]
\hole[.7*#1]{1*#1,1.5*#1}
\hole[.7*#1]{2*#1,.5*#1}
\hole[.7*#1]{-1*#1,1.5*#1}
\hole[.7*#1]{-2*#1,.5*#1}
\hole[.7*#1]{0,0}
\hole[.7*#1]{1*#1,-1.5*#1}
\hole[.7*#1]{2*#1,-.5*#1}
\hole[.7*#1]{-1*#1,-1.5*#1}
\hole[.7*#1]{-2*#1,-.5*#1}
\end{scope}
\begin{scope}[shift={(3*#1,1.5*#1)}, rotate=30]
  \twoCylinders[#1]{(1*#1,0)}
  \rightSmallSurface[#1]{(2*#1,0)}
  \node[inner sep=0] (A) at (0,1*#1) {};
  \node[inner sep=0] (a) at (0,-1*#1) {};
\end{scope}
\begin{scope}[shift={(3*#1,-1.5*#1)}, rotate=-30]
  \twoCylinders[#1]{(1*#1,0)}
  \rightSmallSurface[#1]{(2*#1,0)}
  \node[inner sep=0] (B) at (0,-1*#1) {};
  \node[inner sep=0] (b) at (0,1*#1) {};
\end{scope}
\begin{scope}[shift={(-3*#1,1.5*#1)}, rotate=-30]
  \twoCylinders[#1]{(-1*#1,0)}
  \leftSmallSurface[#1]{(-2*#1,0)}
  \node[inner sep=0] (C) at (0,1*#1) {};
  \node[inner sep=0] (c) at (0,-1*#1) {};
\end{scope}
\begin{scope}[shift={(-3*#1,-1.5*#1)}, rotate=30]
  \twoCylinders[#1]{(-1*#1,0)}
  \leftSmallSurface[#1]{(-2*#1,0)}
  \node[inner sep=0] (D) at (0,-1*#1) {};
  \node[inner sep=0] (d) at (0,1*#1) {};
\end{scope}\end{scope}

\begin{scope}[shift={(#2)}, scale=#1]
\fourTimes{
\draw (0,2.25) ..controls (.5,2.25) and (.5,2.5).. (1,2.5);
}
\draw(1,2.5) to[out=0, in=210, looseness=.75] (A);
\draw(1,-2.5) to[out=0, in=-210, looseness=.75] (B);
\draw(-1,2.5) to[out=180, in=-30, looseness=.75] (C);
\draw(-1,-2.5) to[out=180, in=30, looseness=.75] (D);

\draw (3,0) to[out=90, in=210, looseness=.75] (a);
\draw (3,0) to[out=-90, in=-210, looseness=.75] (b);
\draw (-3,0) to[out=90, in=-30, looseness=.75] (c);
\draw (-3,0) to[out=-90, in=30, looseness=.75] (d);
\end{scope}
}

\newcommand{\anotherLeftSmallSurface}[2][1]{
\begin{scope}[shift={(#2)}, xscale=-1]
\anotherRightSmallSurface[#1]{0,0}
\end{scope}
}

\newcommand{\liftliftliftDumbbellROne}[2][1]{
\begin{scope}[shift={(#2)}]
\begin{scope}[yscale=1.25]
\hole[.7*#1]{0,0}
\hole[.7*#1]{2*#1,0}
\hole[.7*#1]{-2*#1,0}
\hole[.7*#1]{3*#1,.6*#1}
\hole[.7*#1]{-3*#1,.6*#1}
\hole[.7*#1]{1*#1,1*#1}
\hole[.7*#1]{-1*#1,1*#1}
\hole[.7*#1]{3*#1,-.6*#1}
\hole[.7*#1]{-3*#1,-.6*#1}
\hole[.7*#1]{1*#1,-1*#1}
\hole[.7*#1]{-1*#1,-1*#1}
\end{scope}

\twoDoubleCylinders[#1]{-5*#1,0}
\twoDoubleCylinders[#1]{5*#1,0}

\anotherLeftSmallSurface[#1]{-6*#1,0}
\anotherRightSmallSurface[#1]{6*#1,0}

\begin{scope}[line width=.5, shift={(#2)}, scale=#1]
\fourTimes{
  \draw (0,1.6) ..controls (.45,1.6) and (.55,1.8).. (1,1.8)
                ..controls (2,1.8) and (2,1.25).. (3,1.25)
                ..controls (3.5,1.25) and (4,1.25).. (4,1);
}
\end{scope}

\end{scope}
}

\definecolor{lightgray}{gray}{.8}

\newcommand{\boundaryA}{(0,13.5) -- ++(2,0) -- ++(0,1.5) -- ++(1.5,0) -- ++(0,3) -- ++(7.5,0) -- ++(0,-6) -- ++(-4.5,0) -- ++(0,2) -- ++(2,0) -- ++(0,2) -- ++(-3.5,0) -- ++(0,-6) -- ++(9,0) -- ++(0,3) -- ++(5,0) -- ++(0,-9.5) -- ++(-4.5,0) -- ++(0,-2.5) -- ++(-3,0) -- ++(0,5) -- ++(6,0) -- ++(0,5) -- ++(-1.5,0) -- ++(0,-2.5) -- ++(-13,0) -- ++(0,2.5) -- ++(-2,0) -- ++(0,-4.5) -- ++(3.5,0) -- ++(0,-3) -- ++(2,0) -- ++(0,-2) -- ++(3,0) -- ++(0,-1.5)}

\newcommand{\goodCylinderOneZero}[2][1]{
\begin{scope}[line width=.25, shift={(#2)}, scale=#1]
\begin{scope}[scale=.1]
\draw (-21,-21) rectangle (21,21);
\fourTimes{
\fill[green!50!white, opacity=.2] (0,13.5) rectangle ++(2,1.5);
\draw[green!66!black] (0,13.5+.75) -- ++(2,0);
}
\fourTimes{\filldraw[lightgray] \boundaryA -- (0,0) -- cycle;}
\fourTimes{\draw \boundaryA;}
\end{scope}
\end{scope}
}

\newcommand{\goodCylinderZeroOne}[2][1]{
\begin{scope}[line width=.25, shift={(#2)}, scale=#1]
\begin{scope}[scale=.1, rotate=90]
\draw (-21,-21) rectangle (21,21);
\begin{scope}[xscale=18/19, yscale=19/18]
\fourTimes{
\fill[red!50!white, opacity=.2] (0,13.5) rectangle ++(2,1.5);
\draw[red!66!black] (0,13.5+.75) -- ++(2,0);
}
\fourTimes{\filldraw[lightgray] \boundaryA -- (0,0) -- cycle;}
\fourTimes{\draw \boundaryA;}
\end{scope}
\end{scope}
\end{scope}
}

\newcommand{\goodCylinderZeroZeroH}[2][1]{
\begin{scope}[line width=.25, shift={(#2)}, scale=#1]
\begin{scope}[scale=.1]
\draw (-21,-21) rectangle (21,21);
\fourTimes{
\fill[blue!50!white, opacity=.2] (5,14) rectangle ++(3.5,2);
\draw[blue!66!black] (5,14+1) -- ++(3.5,0);

\fill[yellow!50!white, opacity=.2] (5,10) rectangle ++(9,2);
\draw[yellow!66!black] (5,10+1) -- ++(9,0);

\fill[green!50!white, opacity=.2] (17.5,8.5) rectangle ++(-1.5,2.5);
\draw[green!66!black] (17.5,8.5+1.25) -- ++(-1.5,0);

\fill[red!50!white, opacity=.2] (3,8.5) rectangle ++(-2,2.5);
\draw[red!66!black] (3,8.5+1.25) -- ++(-2,0);
}
\fourTimes{\filldraw[lightgray] \boundaryA -- (0,0) -- cycle;}
\fourTimes{\draw \boundaryA;}
\end{scope}
\end{scope}
}

\newcommand{\goodCylinderZeroZeroV}[2][1]{
\begin{scope}[line width=.25, shift={(#2)}, scale=#1]
\begin{scope}[scale=.1]
\draw (-21,-21) rectangle (21,21);
\fourTimes{
\fill[blue!50!white, opacity=.2] (8.5,14) rectangle ++(-2,2);
\draw[blue!66!black] (8.5-1,14) -- ++(0,2);

\fill[yellow!50!white, opacity=.2] (5,10) rectangle ++(1.5,6);
\draw[yellow!66!black] (5+.75,10) -- ++(0,6);

\fill[green!50!white, opacity=.2] (17.5,6) rectangle ++(-1.5,5);
\draw[green!66!black] (17.5-.75,6) -- ++(0,5);

\fill[red!50!white, opacity=.2] (3,6.5) rectangle ++(-2,4.5);
\draw[red!66!black] (3-1,6.5) -- ++(0,4.5);
}
\fourTimes{\filldraw[lightgray] \boundaryA -- (0,0) -- cycle;}
\fourTimes{\draw \boundaryA;}
\end{scope}
\end{scope}
}

\newcommand{\twoTimes}[1]
{ #1
\begin{scope}[scale=-1] #1 \end{scope}
}

\newcommand{\boundaryB}{(0,11) -- ++(6,0) -- ++(0,-4) -- ++(6,0) -- ++(0,-4) -- ++(6,0) -- ++(0,-3)}

\newcommand{\goodCylinderOneOne}[2][1]{
\begin{scope}[line width=.25, shift={(#2)}, scale=#1]
\begin{scope}[scale=.1]
\twoTimes{
\draw[red, thick] (6,7) -- (21,21);
\draw[green, thick] (12,3) -- (21,21);
\draw[blue, thick] (6,-7) -- (21,-21);
\draw[yellow, thick] (12,-3) -- (21,-21);
}
\fourTimes{\filldraw[lightgray] \boundaryB -- (0,0) -- cycle;}
\fourTimes{\draw \boundaryB;}
\draw (-21,-21) rectangle (21,21);
\end{scope}
\end{scope}
}

\newtheorem{theo}{Theorem}[section]
\newtheorem{prop}[theo]{Proposition}
\newtheorem{coro}[theo]{Corollary}
\newtheorem{lemm}[theo]{Lemma}
\theoremstyle{remark}
\newtheorem{rema}[theo]{Remark}
\renewcommand{\d}{\mathrm{d}}
\newcommand{\N}{\mathbb{N}}
\newcommand{\Z}{\mathbb{Z}}
\newcommand{\R}{\mathbb{R}}
\newcommand{\ind}{\mathbf{1}}
\newcommand{\slr}{\mathrm{SL}(2,\R)}

\newcommand{\hol}[1][\omega]{\mathrm{hol}_{#1}}
\newcommand{\pr}{\mathrm{pr}}
\newcommand{\sys}{\mathrm{sys}}

\newcommand{\M}{\mathcal{M}}
\renewcommand{\L}{\mathcal{L}}

\newcommand{\X}{\mathrm{X}}
\newcommand{\Xh}{\X_h}
\newcommand{\Xv}{\X_v}
\newcommand{\W}{\mathrm{W}}
\newcommand{\Wh}{\W_h}
\newcommand{\Wv}{\W_v}
\newcommand{\Xhv}{\mathrm{Y}}

\newcommand{\PP}{\mathfrak{P}}
\renewcommand{\P}{\mathbf{P}}
\newcommand{\p}{\mathbf{p}}
\newcommand{\Pp}{\tilde{\p}}
\newcommand{\Phv}{\P}
\newcommand{\phv}{\p}
\newcommand{\Ph}{\P_h}
\newcommand{\Pv}{\P_v}
\newcommand{\Pph}{\Pp_h}
\newcommand{\Ppv}{\Pp_v}
\newcommand{\ph}{\p_h}
\newcommand{\pv}{\p_v}
\newcommand{\si}{{{}^{{}_{\mathbf{o}\!}}}}

\newcommand{\B}{\mathcal{B}}
\renewcommand{\H}{\mathcal{H}}
\newcommand{\Q}{\mathcal{Q}}
\newcommand{\C}{\mathcal{C}}
\newcommand{\WT}{\mathcal{WT}}

\newcommand{\rquo}[2]{
 \mathchoice
 {\text{\raisebox{3pt}{$#1$}}\!{\Bigm/}\!\text{\raisebox{-3pt}{$#2$}}} 
 {\text{\raisebox{1pt}{$#1$}}\!{\bigm/}\!\text{\raisebox{-1pt}{$#2$}}} 
 {#1/#2} 
 {#1/#2} 
}
\newcommand{\vcfigure}[1]{\centerline{\vtop to .6\textwidth{\vfill\clap{\hbox{#1}}\vfill}}}

\newcommand{\vcTikZ}[1]{\vcfigure{\begin{tikzpicture}{#1}\end{tikzpicture}}}

\begin{document}
\title{Counting problem on wind-tree models}
\author{Angel Pardo}
\begin{abstract}
We study periodic wind-tree models, billiards in the plane endowed with $\Z^2$-periodically located identical connected symmetric right-angled obstacles.
We show asymptotic formulas for the number of (isotopy classes of) closed billiard trajectories (up to $\Z^2$-translations) on the wind-tree billiard. We also compute explicitly the associated Siegel-Veech constant for generic wind-tree billiards depending on the number of corners on the obstacle.
\end{abstract}
\maketitle
\section{Introduction}

The classical wind-tree model corresponds to a billiard in the plane endowed with $\Z^2$-periodic obstacles of rectangular shape; the sides of the rectangles are aligned along the lattice, see Figure~\ref{figu:WTM1}.

\begin{figure}[ht]
\includegraphics[width=.5\textwidth]{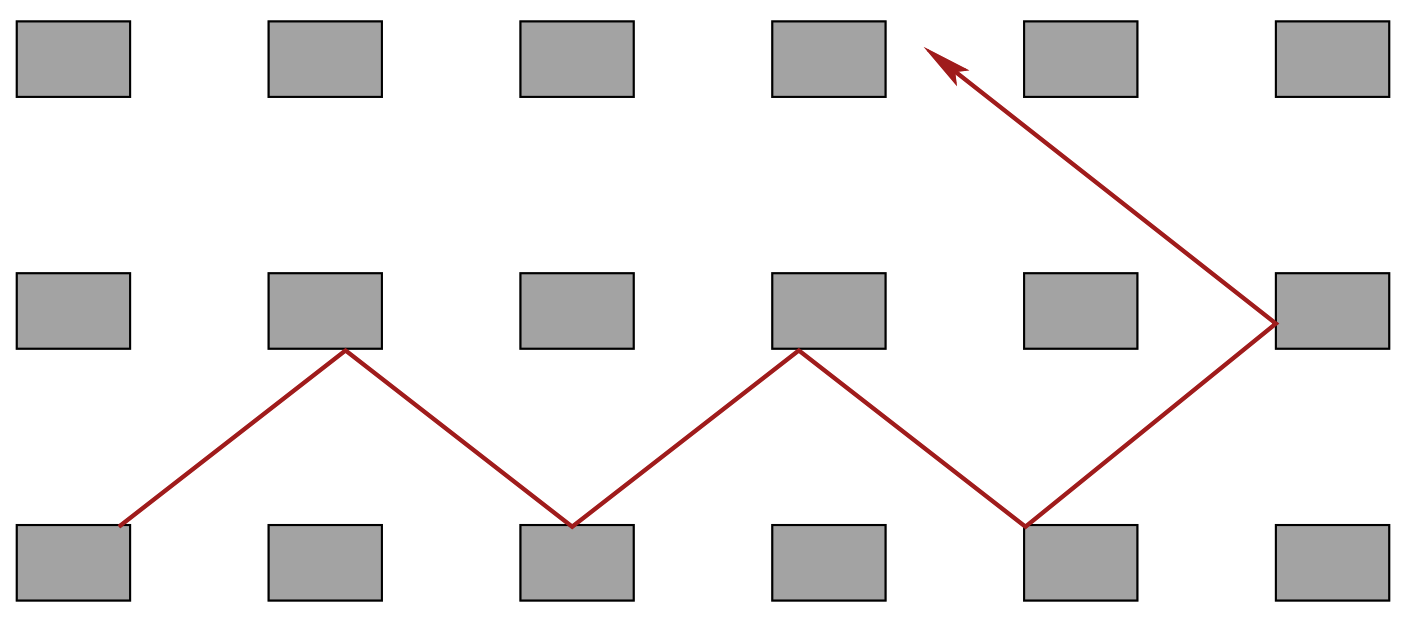}
\caption{Original wind-tree model.}
\label{figu:WTM1}
\end{figure}

The wind-tree model (in a slightly different version) was introduced by P.~Ehrenfest and T.~Ehrenfest~\cite{EE} in 1912. J.~Hardy and J.~Weber \cite{HW} studied the periodic version. All these studies had physical motivations.

Several advances on the dynamical properties of the billiard flow in the wind-tree model were obtained recently using geometric and dynamical properties on moduli space of (compact) flat surfaces; billiard trajectories can be described by the linear flow on a flat surface.

A.~Avila and P.~Hubert \cite{AH} showed that for all parameters of the obstacle and for almost all directions, the trajectories are recurrent. There are examples of divergent trajectories constructed by V.~Delecroix \cite{D}. The non-ergodicity was proved by K.~Fr\c{a}cek and C.~Ulcigrai \cite{FU}. It was proved by V.~Delecroix, P.~Hubert and S.~Leli\`evre \cite{DHL} that the diffusion rate is independent either on the concrete values of parameters of the obstacle or on almost any direction and almost any starting point and is equals to $2/3$. A generalization of this last result was shown by V.~Delecroix and A.~Zorich \cite{DZ} for more complicated obstacles. In this work we study this last variant, corresponding to a billiard in the plane endowed with $\Z^2$-periodic obstacles of right-angled polygonal shape; the obstacles being horizontally and vertically symmetric and the sides of the rectangles are aligned along the lattice, see Figure~\ref{figu:WTMm} for an example.

\begin{figure}[ht]
\includegraphics[width=.75\textwidth]{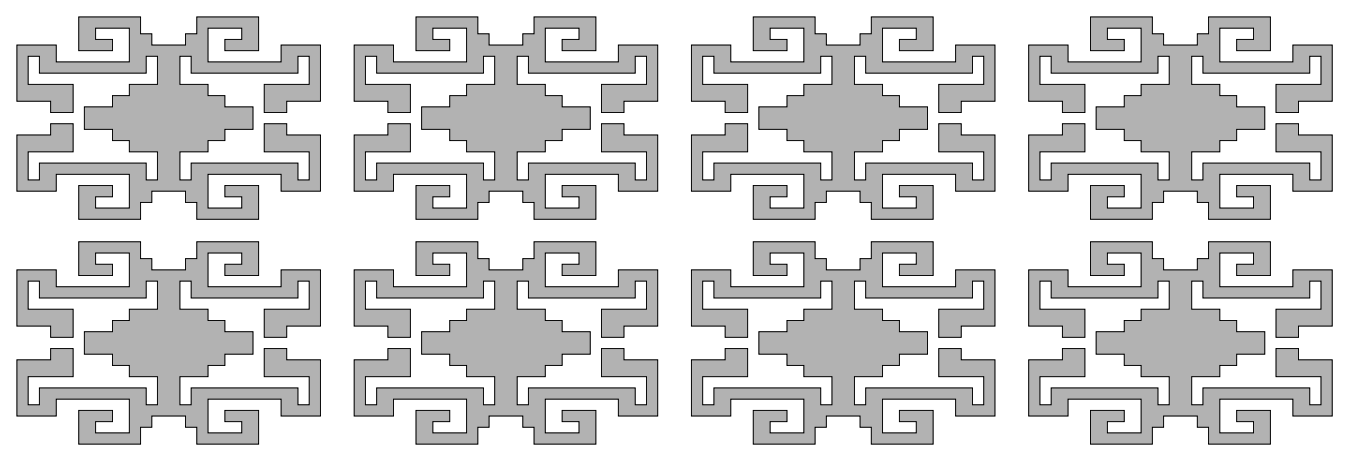}
\caption{Delecroix--Zorich variant.}
\label{figu:WTMm}
\end{figure}

This work concerns asymptotic formulas for the number of (isotopy classes of) closed billiard trajectories on the wind-tree model. Note that we do not count trajectories which go around a single closed trajectory several times, and we are counting unoriented trajectories. This question has been widely studied in the context of (finite) rational billiards and compact flat surfaces, and it is related to many other questions such as the calculation of the volume of normalized strata \cite{EMZ} or the sum of Lyapunov exponents of the geodesic Teichm\"uller flow \cite{EKZ} on strata of flat surfaces (Abelian or quadratic differentials).

H.~Masur~\cite{Ma1,Ma2} proved that for every flat surface $X$, there exist positive constants $c(X)$ and $C(X)$ such that the number $N(L,X)$ of (maximal) cylinders of closed geodesics of length at most $L$ satisfy \[c(X) L^2 \leq N(L,X) \leq C(X) L^2\]
for large enough $L$.
W.~Veech~\cite{Ve1} proved that for Veech surfaces there are in fact exact quadratic asymptotics; E.~Gutkin and C.~Judge~\cite{GJ} gave a different proof. 
Another proof for the upper quadratic bounds was given by Y.~Vorobets~\cite{Vo1}. A.~Eskin and H.~Masur~\cite{EMa} gave yet another one and proved that for each ergodic probability measure $\mu$ on strata of normalized (area~$1$) flat surfaces, there is a constant $c(\mu)$ such that for almost every surface, $N(L,X) \sim c(\mu)\cdot \pi L^2$, that is, \[\lim_{L\to\infty} \frac{N(L,X)}{\pi L^2} = c(\mu).\] The constant $c(\mu)$ is called the Siegel--Veech constant (\cite{EMa}) of the counting problem; it is the constant in the Siegel--Veech formula (\cite{EMa}), a Siegel-type formula introduced by W.~Veech~\cite{Ve2}.

It is still an open problem whether \emph{all} flat surfaces have exact quadratic asymptotics.
The particular constants for several Veech surfaces have been computed explicitly by W.~Veech~\cite{Ve1}, Y.~Vorobets~\cite{Vo1}, E.~Gutkin and C.~Judge~\cite{GJ} and M.~Schmoll~\cite{S}. Constants for some families of non-Veech surfaces were also given by A.~Eskin, H.~Masur and M.~Schmoll~\cite{EMS} and A.~Eskin, J.~Marklof and D.~Witte Morris~\cite{EMarWM}. A.~Eskin, H.~Masur and A.~Zorich~\cite{EMZ} computed the Siegel--Veech constants for connected components of all strata of Abelian differentials, and also described all possible configurations of cylinders of closed geodesics which might be found on a generic flat surface.
In general, the particular constants for Veech surfaces do not coincide with the Siegel--Veech constants of the strata where they live.

The case of quadratic differentials presents extra difficulties. However, J.~Athreya, A.~Eskin and A.~Zorich~\cite{AEZ} gave explicit values for the Siegel--Veech constants on strata of quadratic differentials of genus zero surfaces. E.~Goujard~\cite{Gou} generalized this approach to higher genera and obtained some exact values of Siegel--Veech constants for strata of quadratic differentials away from genus zero.

We prove asymptotic formulas for generic wind-tree models with respect to a natural Lebesgue-type measure (see \cite{AEZ,DZ}) on the parameters of the wind-tree billiards, that is, the side lengths of the obstacles.
Denote by $\WT(m)$ the family of wind-tree billiards such that the obstacle has $4m$ corners with the angle $\pi/2$. Say, all billiards from the original wind-tree family as in Figure~\ref{figu:WTM1} live in $\WT(1)$; the billiard in Figure~\ref{figu:WTMm} belongs to $\WT(17)$. We denote by $\operatorname{Area}\left(\Pi/\Z^2\right)$ the area of a fundamental domain of the $\Z^2$-periodic billiard table $\Pi\in\WT(m)$.

\begin{theo} \label{theo:main}
For almost every wind-tree billiard \,$\Pi\in\WT(m)$ the number $N(L,\Pi)$ of (isotopy classes of) closed billiard trajectories of length at most $L$ in $\Pi$ has quadratic asymptotic growth rate \[ N(L,\Pi) \sim c(m) \cdot \frac{\pi L^2}{\operatorname{Area}\left(\Pi/\Z^2\right)},\]
where \[c(m) = \left(20m^2 - 95m - 78 + 78\cdot 4^m\frac{(m!)^2}{(2m)!}\right)\frac{1}{6\pi^2}.\]
\end{theo}

The constant $c(m)$ is not the Siegel--Veech constant of one particular surface, but corresponds to Siegel--Veech constants of some particular configurations of cylinders on compact flat surfaces associated to generic wind-tree billiards.

On the other hand, A.~Eskin, M.~Mirzakhani and A.~Mohammadi~\cite{EMM} showed that for \emph{all} (area~$1$) flat surfaces we have \emph{weak} quadratic asymptotic formulas, 
\[\lim_{L\to\infty}\frac{1}{L}\int_0^L \frac{N(e^t,X)}{\pi e^{2t}}\d t=c(X),\]
which we write $N(L,X) \mathbin{\text{\rm{``}$\sim$\rm{''}}} c(X) \cdot \pi L^2$.
The constant $c(X)$ being the Siegel--Veech constant associated to the affine invariant measure supported on the $\slr$-orbit closure of the surface $X$ given by general invariant measure classification theorem of A.~Eskin and M.~Mirzakhani~\cite{EMi}.

Using this technology, one can prove weak asymptotic formulas for individual wind-tree billiards. In particular, the following holds.

\begin{theo} \label{theo:weak}
Let $\Pi\in\WT(m)$ be a wind tree billiard.
\begin{enumerate}[leftmargin=*]
\item Suppose that one of the following conditions holds
  \begin{enumerate}[leftmargin=*]
  \item All the parameters of $\Pi$ are rational, or
  \item $m=1$ and there exists a square-free integer $D > 0$ such that the two parameters of $\Pi$, say $a,b\in(0,1)$,
  can be written as $1/(1-a) = x+z\sqrt{D}$ and $1/(1-b) = y + z\sqrt{D}$ with $x, y, z \in\mathbb{Q}$ and $x+y=1$.
  \end{enumerate}
Then, \[N(L,\Pi) \sim c(\Pi) \cdot \frac{\pi L^2}{\operatorname{Area}\left(\Pi/\Z^2\right)}.\]
\item In any other case, we have the weak asymptotic formula
\[ N(L,\Pi) \mathbin{\text{\rm{``}$\sim$\rm{''}}} c(\Pi) \cdot \frac{\pi L^2}{\operatorname{Area}\left(\Pi/\Z^2\right)}.\]
\end{enumerate}
\end{theo}

The case (1) corresponds to (particular cases of) Veech surfaces and formulas for the Siegel--Veech constants can be obtained following an approach similar to the one of E.~Gutkin and C.~Judge~\cite[\S~6]{GJ}. In the case (a), when the parameters are rational, it corresponds to square-tiled surfaces and it is possible to obtain formulas similar to the obtained by A.~Eskin, M.~Kontsevich and A.~Zorich~\cite[Theorem~4]{EKZ}.
In the other cases
we do not know the Siegel--Veech constants for every wind-tree billiard. However, it depends only on $\slr$-orbit closures (of a compact flat surface associated to the wind-tree billiard) and, in particular, it coincides with $c(m)$ for generic billiards.

\subsection{Strategy of the proof}
We reformulate the counting problem on wind-tree billiards in terms of a counting problem on a $\Z^2$-periodic flat surface. This is quite elementary and straightforward. For details on the reduction of the study of the billiard flow into the study of a $\Z^2$-cocycle over the linear flow of a finite flat surface, see \cite[\S~3]{DHL}.

In general, we can consider an infinite flat surface $X_\infty$ which is a ramified $\Z^d$-cover over a compact flat surface $X$, $d\geq 1$ ($d=2$ in our case). Let $\Sigma$ be the finite set of singularity points of $X$. Since the intersection form $\langle\cdot,\cdot\rangle$ is non-degenerate between $H^1(X\setminus\Sigma,\Z)$ and $H^1(X,\Sigma,\Z)$, every such $\Z^d$-cover is defined by a $d$-tuple of independent elements $\mathbf{f}=(f_1,\dots,f_d)$ in the group of relative cohomology $H^1(S,\Sigma,\Z)$, but we restrict ourselves to the case when $\mathbf{f}\in H^1(X,\Z^d)$ ---this is the case of the infinite $\Z^2$-periodic flat surface associated to a wind-tree model.

We are interested in counting (maximal) cylinders of closed geodesics in $X_\infty$ (up to $\Z^d$-translations, of course). Cylinders of closed geodesics in the cover $X_\infty$ clearly descends to cylinders in $X$, but not the other way around. In fact, by definition of the covering, cylinders in the cover $X_\infty$ are exactly the lift of those cylinders $C$ in $X$ such that $\gamma_C$, (the Poincar\'e dual of the homology class of) its core curve, verifies $\langle\gamma_C,f_i\rangle = 0$, for each $i=1,\dots,d$. 

One of the main tools used in this kind of problems (and many others) is the $\slr$-action on strata of flat surfaces (see, e.g., \cite{EMa,EMZ}) and the associated cocycle over the Hodge bundle, the Kontsevich--Zorich cocycle.
Let $\M$ be the $\slr$-orbit closure of $X$, $F$ be a subbundle of the Hodge bundle over $\M$, invariant with respect to the Kontsevish--Zorich cocycle, and let $f\in F_X$.

Note that cylinders $C$ in $X$ such that $\langle\gamma_C,f\rangle = 0$ split naturally into two families:
(a) the family of cylinders such that $\langle\gamma_C,h\rangle = 0$ for all $h\in F_X$, which we call \emph{$F$-good cylinders}, and
(b) the family of cylinders that are not $F$-good, but $\langle\gamma_C,f\rangle = 0$. These later are called \emph{$(F,f)$-bad cylinders}.
This notion of $F$-good cylinders was first introduced by A.~Avila and P.~Hubert~\cite{AH} in order to give a geometric criterion for recurrence of $\Z^d$-periodic flat surfaces.

Thus, counting cylinders in a $\Z^d$-periodic flat surface can be reduced to count separately cylinders which are $(\oplus_j F^{(j)})$-good cylinders and $(F^{(j_i)},f_i)$-bad cylinders in the compact surface, for some appropriate subbundles $(F^{(j)})_j$.

In the case of the classical wind-tree model, that is, for $m=1$, V.~Delecroix, P.~Hubert and S.~Leli\`evre~\cite{DHL} gave a complete description of the cocycles defining the surfaces and the corresponding decomposition of the Hodge bundle, which allows us to successfully apply this approach. This is extended naturally to the Delecroix--Zorich variant ($m>1$). In fact, for every $\Pi\in\WT(m)$, there are two cocycles $h$ and $v$ in a compact flat surface $\X=\X(\Pi)$ defining the $\Z^2$-periodic flat surface $\X_\infty=\X_\infty(\Pi)$ associated to $\Pi$ and two $2$-dimensional equivariant subbundles, which we denote by $F^{+-}$ and $F^{-+}$, such that $h\in F^{+-}$ and $v\in F^{-+}$.

Using the main result of A.~Eskin and H.~Masur in~\cite{EMa}, it is a straightforward remark that we have asymptotic formulas for the number of $F$-good cylinders with an associated Siegel--Veech constant, for generic surfaces, for any $\slr$-ergodic finite measure on any normalized strata. In the case of $(F,f)$-bad cylinders, this is no longer true. However, in the case of the wind-tree model, we prove the following.

\begin{theo} \label{theo:bad-cylinders-WTM}
Let $\Pi\in\WT(m)$ be a wind-tree billiard, $\X=\X(\Pi)$ the associated compact flat surface and let $F$ be one of the associated subbundles $F^{+-}$ or $F^{-+}$. Then, for any $f\in F_\X$ the number $N_F(L,f)$, of $(F,f)$-bad cylinders in $\X$ of length at most $L$, has \emph{subquadratic} asymptotic growth rate, that is, $N_F(L,f)=o(L^2)$ or, which is the same, \[\lim_{L\to\infty} \frac{N_F(L,f)}{\pi L^2} = 0.\]
\end{theo}

We use technology for asymptotic formulas developed by A.~Eskin and H.~Masur~\cite{EMa} in order to prove (a slightly more general version of) Theorem~\ref{theo:bad-cylinders-WTM}. For this, we need in addition the condition of non-zero Lyapunov exponents for to the relevant subbundles $F^{+-}$ and $F^{-+}$. This is true for almost every wind-tree billiards thanks to one of the main results of V.~Delecroix and A.~Zorich in \cite{DZ}. For the statement to be true for every wind-tree billiard, we use (a slightly more general version of) the so called Forni's criterion due to G.~Forni~\cite{Fo}, a geometric criterion for the positivity of Lyapunov exponents, applied to integer equivariant subbundles.

As a consequence of Theorem~\ref{theo:bad-cylinders-WTM}, the proof of Theorem~\ref{theo:main} is reduced to compute the Siegel--Veech constant associated to configurations of $F^{+-}\oplus F^{-+}$-good cylinders.
Furthermore, Theorem~\ref{theo:weak} becomes a compilation of several different results and we omit its proof here; it is almost identical to the proof of Theorem~1.7 in~\cite{AEZ}, after the reduction given by Theorem~\ref{theo:bad-cylinders-WTM}, to the problem of counting only $F^{+-}\oplus F^{-+}$-good cylinders.

For the computation of the Siegel--Veech constant associated to configurations of $F^{+-}\oplus F^{-+}$-good cylinders, we make use of extra symmetries in the surface $\X(\Pi)$ to describe it as a cover of lower genus surfaces. In particular, configurations of $F^{+-}\oplus F^{-+}$-good cylinders are related to configurations of cylinders on some strata of genus zero surfaces, 
such that they lift to homologically trivial cylinders on some strata of genus one surfaces.

C.~Boissy~\cite{Bo} described all possible configurations on generic surfaces in genus zero. Using this, we describe all possible configurations of cylinders satisfying the homological conditions ensuring they correspond to $F^{+-}\oplus F^{-+}$-good cylinders. Then, we relate Siegel--Veech constants of configurations in the genus zero surface with the constant for the higher genus surface and do the combinatorics. Finally, plugging in the resulting expression the explicit values of the Siegel--Veech constants for configurations on generic surfaces of genus zero obtained by J.~Athreya, A.~Eskin and A.~Zorich~\cite{AEZ} and, proving certain combinatorial identities for resulting hypergeometric sums, we obtain the desired explicit value of $c(m)$.

\subsection{Side results}

As a by-product of our methods, we obtain several results as detailed below.

\subsection*{Area Siegel-Veech constant}
Following the same strategy, we are able to compute the area Siegel--Veech constant, associated to the counting of the area of maximal families of isotopy classes of compact trajectories. More precisely, we have the analogous of Theorem~\ref{theo:main}:

\begin{theo} \label{theo:c-area}
For almost every \,$\Pi\in\WT(m)$ the weighted number $N_{area}(L,\Pi)$ of maximal families of isotopic closed billiard trajectories of length at most $L$ in $\Pi$, where the weight is the area covered by the family, has quadratic asymptotic growth rate \[ N_{area}(L,\Pi) \sim c_{area}(m) \cdot \frac{\pi L^2}{\operatorname{Area}\left(\Pi/\Z^2\right)},\]
where \[c_{area}(m) = \left(8m - 33 + 39\cdot 4^m\frac{(m!)^2}{(2m+1)!}\right)\frac{1}{3\pi^2}.\]
\end{theo}

\subsection*{Polynomial diffusion}
Let $d(\cdot,\cdot)$ be the Euclidean distance on $\R^2$ and consider the wind-tree billiard table $\Pi\in\WT(m)$ as a subset of $\R^2$. Let $(\phi^\theta_t)_{t\in\R}$ be the billiard flow in direction $\theta\in[0,2\pi)$ on $\Pi$, that is, $\phi^\theta_t(x)$ is the position of a particle after time $t$ starting from position $x\in\Pi$ in direction $\theta$.

The application of the Forni's criterion to the relevant subbundles $F^{+-}$ and $F^{-+}$ allows us to show that they have non-zero Lyapunov exponents. Applying the result \cite[Corollary~1]{DZ} of V.~Delecroix and A.~Zorich, which is a generalization of the analogous result for the classical model due to V.~Delecroix, P.~Hubert and S.~Leli\`evre~\cite{DHL}, we obtain the following.

\begin{theo} \label{theo:diffusion}
For every wind-tree billiard \,$\Pi\in\WT(m)$ there exists $\delta(\Pi)>0$ such that for almost every direction $\theta\in[0,2\pi)$ and every starting point (with infinite forward orbit) \[\limsup_{t\to\infty} \frac{\log d(x,\phi^\theta_t(x))}{\log t} = \delta(\Pi).\]
\end{theo}

Here, $\delta(\Pi)$ is the polynomial diffusion rate and coincides with the Lyapunov exponent mentioned above.
Note that this result is already known for $m=1$ and the diffusion rate $\delta$ is $2/3$ independently of the billiard table (see \cite[Theorem~1]{DHL}), and for almost all $\Pi\in\WT(m)$, for $m>1$, with $\delta(m)=4^m (m!)^2/(2m+1)!$, also independent of the billiard (see \cite[Theorem~1]{DZ}). 
Moreover, the value of $\delta(\Pi)$ depends only on $\slr$-orbit closures (of the compact flat surface associated to the wind-tree billiard).
Anyway, the interest of this result relies in the fact that the diffusion rate $\delta(\Pi)$ is \emph{positive} for \emph{every} $\Pi\in\WT(m)$.

\subsection*{Recurrence}
A.~Avila and P.~Hubert~\cite{AH} gave a geometric criterion for the recurrence of a $\Z^d$-periodic flat surfaces in terms of good cylinders and proved the recurrence for the original wind-tree model. Using this criterion, our approach allows us to prove the recurrence for the Delecroix--Zorich variant. More precisely, we have the following.

\begin{theo} \label{theo:recurrence}
For every wind-tree billiard \,$\Pi\in\WT(m)$ the billiard flow in $\Pi$ is recurrent for almost every direction $\theta\in[0,2\pi)$.
\end{theo}

This result is already known for $m=1$ (see \cite[Theorem~1]{AH}). Moreover, as explained to us by V.~Delecroix, a criterion of recurrence due to N.~Chevallier and J.-P.~Conze~\cite[Corollary~1.2]{CC} allows us to conclude that the billiard flow $\phi^\theta_t$ is recurrent in $\Pi$ for almost every direction $\theta\in[0,2\pi)$ if the polynomial diffusion rate (see above) $\delta(\Pi)<1/2$. However, we only know that the polynomial diffusion rate is less than $1/2$ for \emph{almost every} $\Pi\in\WT(m)$ and only for $m>2$.

\subsection{Structure of the paper}

In \S~\ref{sect:background} we briefly recall all the background necessary to formulate and prove the results. 
In \S~\ref{sect:counting-periodic-surfaces} we do the reduction of the counting problem on general $\Z^d$-periodic flat surfaces to the counting of $(\oplus_j F^{(j)})$-good cylinders and $(F^{(j_i)},f_i)$-bad cylinders in the compact surface, for some appropriate subbundles $(F^{(j)})_j$ of the Hodge bundle.
In \S~\ref{sect:bad-cylinders} we prove Theorem~\ref{theo:bad-cylinders}, a slightly more general version of Theorem~\ref{theo:bad-cylinders-WTM}, but with the extra condition that some particular Lyapunov exponent is positive.
In \S~\ref{sect:application-wtm} we show that the relevant Lyapunov exponent is positive applying the Forni's criterion to integer equivariant subbundles, which ends the proof of Theorem~\ref{theo:bad-cylinders-WTM} and allows us to reduce the problem to the counting of $F^{+-}\oplus F^{-+}$-good cylinders.
In \S~\ref{sect:configurations} we study configurations of cylinders on generic genus zero surfaces in order to describe $F^{+-}\oplus F^{-+}$-good cylinders.
In \S~\ref{sect:which-good} we show which configurations of cylinders on generic genus zero surfaces lift to $F^{+-}\oplus F^{-+}$-good cylinders in the higher genus surface by means of topological considerations. Then, in \S~\ref{sect:how-good}, we describe how these cylinders lift to the higher genus surface, that is, the number of cylinders we obtain and their length. With this, we are able to relate in \S~\ref{sect:relation-SV} the Siegel--Veech constants of the genus zero and the higher genus surfaces.

Finally, in \S~\ref{sect:SV-good} we compute the Siegel--Veech constant of $F^{+-}\oplus F^{-+}$-good cylinders: we count the possible configurations taking part in the computations and plug in the explicit values of the Siegel--Veech constants obtained by J.~Athreya, A.~Eskin and A.~Zorich~\cite{AEZ}. This allows us to conclude the computations by means of a combinatorial identity for certain hypergeometric sums proved separately in an appendix.

Side results mentioned above are proved in \S~\ref{sect:side-results}.

\subsection*{Acknowledgments}
The author is greatly indebted to Pascal Hubert for his guide and invaluable help at every stage of this work, to Anton Zorich for introducing me in the theory of flat surfaces, to both of them for their constant encouragement, kind explanations and useful discussions.
The author is grateful to Vincent Delecroix for his kind explanation of his work with A.~Zorich about the polynomial diffusion rate of generalized wind-tree models and for pointing out the recurrence in these models when there is a low polynomial diffusion rate.
The author is grateful to Carlos Matheus for corroborate and give some details about the validity of Forni's criterion for an integer equivariant subbundle.

\section{Background} \label{sect:background}
\subsection{Flat surfaces}
For an introduction and general references to this subject, we refer the reader to the surveys of Zorich~\cite{Zo}, Forni--Matheus~\cite{FM}, Wright~\cite{Wr2}.

\subsection*{Flat surfaces and strata}
Let $S$ be a compact Riemann surface of genus $g$. Let $\alpha=\{n_1,\dots,n_k\}\subset\N$ be a partition of $2g - 2$ and $\H(\alpha)$ be a stratum of Abelian differentials on $S$, that is, the space of pairs $X=(S,\omega)$ where $\omega$ is a holomorphic $1$-form on $S$ with zeros of degrees $n_1,\dots,n_k\in\N$. Let $\Sigma = \Sigma(\omega)$ be the set of singularities of $X$, the zeros of $\omega$.
The form $\omega$ defines a canonical flat metric 
on $S$ with conical singularities of angle $2\pi(n+1)$ at zeros of degree~$n$ of $\omega$.

We also consider strata $\Q(d_1,\dots,d_k)$ of meromorphic quadratic differentials with at most simple poles on $S$, the spaces of pairs $(S, q)$ where $q$ is a meromorphic quadratic differential on $M$ with zeros of order $d_1,\dots,d_k$, $d_i\in\{-1\}\cup\N$ for $i=1,\dots,k$ (in a slight abuse of vocabulary, we are considering poles as zeros of order~$-1$) and $\sum_{i=1}^k d_i=4g-4$.
The quadratic differential $q$ also defines a canonical flat metric with conical singularities of angle $\pi(d+2)$ at zeros of order~$d$ of $q$.

In this paper, a quadratic differential is not the square of an Abelian differential and a flat surface is the Riemann surface with the flat metric corresponding to an Abelian or quadratic differential.

The area of a flat surface is the one obtained from the flat metric. 
Let $\H_1(\alpha)$ denote the codimension~$1$ subspace of area~$1$ on $\H(\alpha)$ denote the codimension~$1$ subspace of (flat) area~$1$. 

\subsection*{$\slr$-action and the Teichm\"uller geodesic flow}
There is a natural action of $\slr$ on strata of Abelian differentials,
which generalizes the action of $\slr$ on the space $\rquo{\mathrm{GL}(2,\R)}{\mathrm{SL}(2,\Z)}$ of flat tori.
Let \[g_t = \left(\begin{matrix} e^t & 0 \\ 0 & e^{-t} \end{matrix}\right) \quad \text{ and } \quad r_\theta = \left(\begin{matrix} \cos\theta & \sin\theta \\ -\sin\theta & \cos\theta \end{matrix}\right).\]
The element $r_\theta\in\slr$ acts by $(S,\omega)\mapsto (S,e^{i\theta}\omega)$. This has the effect of rotating the flat surface by the angle $\theta\in[0,2\pi)$. The action of $(g_t)_{t\in\R}$ is called the Teichm\"uller geodesic flow.


\subsection*{Affine invariant measures and manifolds}
Each stratum carries a natural Lebesgue measure, invariant under the action of $\slr$, which is given by the pullback of the Lebesgue measure on $H^1(S,\Sigma,\mathbb{C})\cong\mathbb{C}^{2g+k-1}$.

An affine invariant manifold is an $\slr$-invariant closed subset of $\H_1(\alpha)$, which looks like an affine subspace in period coordinates (see, e.g., \cite[\S~3]{Zo}). Each affine invariant manifold $\M$ is the support of an ergodic $\slr$-invariant probability measure $\nu_\M$. Locally, in period coordinates, this measure is (up to normalization) the restriction of Lebesgue measure to the subspace $\M$ (see~\cite{EMi} for the precise definitions). Eskin--Mirzakhani--Mohammadi~\cite{EMM} proved that any $\slr$-orbit closure is an affine invariant manifold.
The most important case of an affine invariant manifold is a connected component of a stratum $\H_1(\alpha)$. Masur~\cite{Ma0} and Veech~\cite{Ve0} independently proved that in this case, the total mass of this measure is finite and ergodic with respect to the Teichm\"uller geodesic flow. The associated affine measure is known as the Masur--Veech measure.

\subsection*{Hodge bundle and the Kontsevich--Zorich cocycle}
The (real) Hodge bundle $H^1$ is the real vector bundle of dimension $2g$ over an affine invariant manifold $\M$, where the fiber over $X=(S,\omega)$ is the real cohomology $H^1_X=H^1(S,\R)$. Each fiber $H^1_X$ has a natural lattice $H^1_X(\Z)=H^1(S,\Z)$ which allows identification of nearby fibers and definition of the Gauss--Manin (flat) connection. 
The monodromy of the Gauss--Manin connection restricted to $\slr$-orbits provides a cocycle called the Kontsevich--Zorich cocycle, which we denote by $A(g,X)$, for $g\in\slr$ and $X\in\M$. 
The Kontsevich--Zorich cocycle is a symplectic cocycle because it preserves the intersection form $\langle f_1,f_2\rangle=\int_S f_1\wedge f_2$ on $H^1(S,\R)$, which is a symplectic form on the $2g$-dimensional real vector space $H^1(S,\R)$.
Let $\|\cdot\|_\omega$ be the Hodge norm 
(for precise definition see, e.g., \cite[\S~3.4]{FM}). 
The Hodge norm depend continuously on $(S,\omega)$, but is not preserved by the Kontsevich--Zorich cocycle in general.


\subsection*{Lyapunov exponents}
Given any affine invariant manifold $\M$, we know from Oseledets theorem that there are real numbers $\lambda_1(\M) \geq \dots \geq \lambda_{2g}(\M)$, the Lyapunov exponents, 
and a measurable $g_t$-equivariant filtration of the Hodge bundle $H^1(S,\R) = V_1(X)\supset\dots\supset V_{2g}(X)=\{0\}$ at $\nu_\M$-almost every $X=(S,\omega)\in\M$ and \[\lim_{t\to\infty} \frac{1}{t} \log \| A(g_t,X) f\|_{g_t\omega}=\lambda_i\] for every $f\in V_i\setminus V_{i+1}$.
\begin{theo}
    [Chaika-Eskin \cite{CE}] \label{theo:Chaika-Eskin}
    Let $X$ be a flat surface and $\M$ be the $\slr$-orbit closure of $X$. Then, for almost every $\theta\in[0,2\pi)$ we have the $g_t$-equivariant filtration $H^1(S,\R) = V_1(r_\theta X)\supset\dots\supset V_{2g}(r_\theta X)=\{0\}$ and, for every $f\in V_i\setminus V_{i+1}$, \[\lim_{t\to\infty} \frac{1}{t} \log \| A(g_t,r_\theta X) f\|_{g_tr_\theta\omega}=\lambda_i(\M).\]
\end{theo}

The set $\Lambda(\M)$ of Lyapunov exponents is called Lyapunov spectrum (of the Kontsevich--Zorich cocycle over the Teichm\"uller flow on $\M$).
The fact that the Kontsevich--Zorich cocycle is symplectic means that the Lyapunov spectrum is always symmetric, $\Lambda(\M)=-\Lambda(\M)$.

\subsection*{Equivariant subbundles of the Hodge bundle}
Let $\M$ be an affine invariant submanifold and $F$ a subbundle of the Hodge bundle over $\M$. We say that $F$ is equivariant if it is invariant under the Kontsevich--Zorich cocycle, and we say that $F$ is irreducible if it has no proper equivariant subbundles.
Since $\M$ is $\slr$-invariant, by the definition of the Kontsevich--Zorich cocycle, a flat (locally constant) subbundle is always equivariant.

Previous discussion about Lyapunov exponents applies in this context as well and we have that, as before, for every $X=(S,\omega)\in\M$ such that $\M$ is the $\slr$-orbit closure of $X$ and almost every $\theta\in[0,2\pi)$, there is a $g_t$-equivariant filtration $F_{r_\theta X} = U_1(r_\theta X)\supset\dots\supset U_r(r_\theta X)=\{0\}$, where $r=\operatorname{rank}F=\dim F_X$ and, for every $f\in U_i\setminus U_{i+1}$, \[\lim_{t\to\infty} \frac{1}{t} \log \| A(g_t,r_\theta X) f\|_{g_tr_\theta\omega}=\lambda_i(\M,F).\]
The Lyapunov spectrum restricted to $F$ is $\Lambda(\M,F)=\{\lambda_i(\M,F)\}_{i=1}^r\subset\Lambda(\M)$.

\begin{rema}\label{rema:Lyapunov-symplectic}
If $F$ is irreducible and admits a non-zero Lyapunov exponent in its Lyapunov spectrum, then $F$ is symplectic with respect to the intersection form, that is, the symplectic intersection form is non-degenerate on $F$ (this is a nontrivial fact that can be deduced from \cite[Theorem~A.9]{EMi}, which in turn is deduced from \cite{FoMatZ}). In particular, $F$ is an even-dimensional subbundle and, as before, the associated Lyapunov spectrum is symmetric, $\Lambda(\M,F)=-\Lambda(\M,F)$.\end{rema}

We denote by $F^\dagger$ the symplectic complement of $F$ and, when $F$ is symplectic, define $F_X^\pr(\Z)=\pr_{F_X} H^1_X(\Z)$, where $\pr_{F_X}:H^1_X\to F_X$ is the symplectic projection, that is, the first component of the decomposition $H^1_X=F_X\oplus F_X^\dagger$.

We denote by $F_X(\Z)= F_X\cap H^1_X(\Z)$ the set of integer cocycles in $F_X$.
We say that $F$ is defined over $\Z$ if it is generated by integer cocycles, that is, if $F_X=\left<F_X(\Z)\right>_\R$. When $F$ is defined over $\Z$, $F_X(\Z)$ is a lattice in $F_X$. If, in addition, $F$ is symplectic, we have that $F_X^\pr(\Z)$ is also a lattice and $F_X(\Z)\subset F_X^\pr(\Z)$.

\subsection{Counting problem}

We are interested in the counting of closed geodesics of bounded length on flat surfaces.

\subsection*{Cylinders of closed geodesics and saddle connections}
Together with every closed regular geodesic in a flat surface $X=(S,\omega)$ (resp. $(S,q)$) we have a bunch of parallel closed regular geodesics. 
A cylinder on a flat surface is a maximal open annulus filled by isotopic simple closed regular geodesics. A cylinder $C$ is isometric to the product of an open interval and a circle, and its core curve $\gamma_C$ is the geodesic projecting to the middle of the interval. A saddle connection is a geodesic joining two different singularities  or a singularity to itself, with no singularities in its interior. Cylinders are always bounded by parallel saddle connections.

\subsection*{Holonomy}
Integrating $\omega$ (resp. a locally defined square-root of $q$) along the core curve of a cylinder, a saddle connection or, more generally, any homology class $\gamma\in H_1(S,\Sigma,\Z)$, we get a complex number. Considered as a planar vector, this complex number represents the affine holonomy along $\gamma$ and we denote this holonomy vector by $\hol(\gamma)$. In particular, in the case of a cylinder or saddle connection, its euclidean length corresponds to the modulus of its holonomy vector.

\subsection*{Systole}
Let $\sys(X)$ be the systole of the flat surface $X$, that is, the length of its shortest saddle connection, and let $K_\epsilon=\{X:\sys(X)\geq\epsilon\}$. $K_\epsilon$ form a compact exhaustion on any affine invariant manifold (which are never compact).

\subsection*{Counting problem and Siegel--Veech constants}
Consider the set of all cylinders on a flat surface $X$ and consider its image $V(X)\in\R^2\cong \mathbb{C}$ under the holonomy map, $V(X)=\left\{ \hol[] \gamma_C : C \text{ is a cylinder in } X \right\}$. This is a discret set of $\R^2$. We are concerned with the asymptotic behavior of the number $N(L,X)=\# V(X)\cap B(L)$ of cylinders in $X$ of length at most $L$, when $L\to\infty$.

\begin{theo}[Eskin--Masur {\cite{EMa}}] \label{theo:EM-SV-constant} Let $\M$ be an affine invariant manifold. Then, there is a constant $c(\M)$ such that for $\nu_\M$-almost all $X\in\M$
\begin{equation} \label{equa:EM}
\lim_{L\to\infty} \frac{N(L,X)}{\pi L^2}=c(\M),
\end{equation}
where $c(\M)$ is the Siegel--Veech constant given by the Siegel--Veech formula
\begin{equation} \label{equa:SV-formula}
c(\M)=\frac{1}{\pi \rho^2}\int_\M N(\rho,Y)\d\nu_\M(Y).
\end{equation}
\end{theo}

We use some of the tools developed by Eskin--Masur when proving this theorem. In particular, the following are of special utility to us.

\begin{theo}[{\cite[Theorem~5.1(b)]{EMa}}] \label{EM:5.1} For any $X\in\H(\alpha)$ and all $\delta,\rho > 0$, \[N(\rho,X) \leq \frac{c(\rho,\delta)}{\sys(X)^{1+\delta}}.\]
\end{theo}

\begin{theo}[{\cite[Theorem~5.2]{EMa}}] \label{EM:5.2} For any $X\in\H(\alpha)$, any $\beta < 2$ and all $t>0$, \[\int_0^{2\pi}\frac{\d\theta}{\sys(g_tr_\theta X)^\beta} \leq c(X,\beta).\]
\end{theo}

We remark that these two results are true for \emph{every} flat surface, in contrast to Theorem~\ref{theo:EM-SV-constant}, which holds for \emph{almost every} flat surface.

\subsection*{Configurations of cylinders} \label{sect:configuration}
A collection $\mathbf{C}=\{C_1,\dots,C_n\}$ of cylinders 
determines the data on combinatorial geometry of the decomposition of $S\setminus\mathbf{C}$. It determines the number of components, their boundary structure, the singularity data for each component and how the components are glued to each other. These data are referred to as configuration of cylinders (see \cite{EMZ}). The multiplicity of a configuration is the number of cylinders it defines. Remark that we reserve the notion of configuration for geometric types of possible collections of cylinders, and not for the collections themselves.

In this work, we are only concerned with multiplicity one configurations, that is, those defining a single cylinder. We are also concerned with some homological conditions ---and not only the geometric combinatorics--- when considering configurations (see \S~\ref{sect:counting-periodic-surfaces}). However, this information is also carried by configurations because of topological considerations. 

\begin{rema} \label{rema:EM-SV-constant} Let $\C$ be a configuration of cylinders and consider now $N_\C(L,X)$, the number of cylinders in $X$ of length at most $L$ forming a configuration of type $\C$. Then, the analogous of Theorem~\ref{theo:EM-SV-constant} is also true in this context (see \cite{EMa,EMZ}), with the Siegel-Veech constant associated to this counting problem depending also on the configuration, $c_\C(\M)=c(\C,\M)$.
\end{rema}

\subsection{Configuration of cylinders in genus zero and associated Siegel--Veech constants} \label{sect:configurations-sphere}
In the following, we recall briefly results from \cite{Bo} describing configurations of periodic geodesics for flat surfaces in genus zero, and the results from \cite{AEZ} providing the values of the corresponding Siegel--Veech constants.

According to \cite{Bo} and \cite{MZ}, almost any flat surface in any stratum $\Q(d_1,\dots,d_k)$ of meromorphic quadratic differentials with at most simple poles ($d_i \in\{-1\}\cup\N$) on the sphere ($\sum_{i=1}^{k}d_i = -4$), different from the pillowcase stratum $\Q(-1^4)$, does not have a single regular closed geodesic not contained in one of the two families described below.

\subsection*{Pocket configurations} In a pocket configuration, we have a single cylinder filled with closed regular geodesics, such that the cylinder is bounded by a saddle connection joining a fixed pair of poles $P_{j_1},P_{j_2}$ on one side and by a saddle connection joining a fixed zero $P_i$ of order $d_i \geq 1$ to itself, on the other side (see Figure~\ref{pocket}).
By convention, the affine holonomy associated to this configuration corresponds to the closed geodesic and not to the saddle connection joining the two poles. Such a saddle connection is twice as short as the closed geodesic.

\begin{figure}[h]
\centering
\begin{tikzpicture}
\pocketConfiguration[.67]{0,0}
\end{tikzpicture}
\caption{A pocket configuration with a cylinder bounded by a saddle connection joining two poles on one side and, by a saddle connection joining a zero to itself, on the other side.}
\label{pocket}
\end{figure}
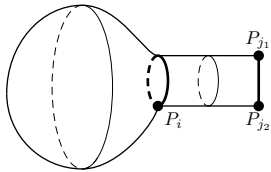

By Theorem~4.5 and formula~(4.28) in \cite{AEZ}, the Siegel--Veech constant $c^{\text{pocket}}_{j_1, j_2; i}$ corresponding to this configuration has the form
\[c^{\text{pocket}}_{j_1, j_2; i} = \frac{d_i+1}{k-4}\frac{1}{2\pi^2}.\]

One can consider the union of several configurations as above fixing the pair of poles $P_{j_1}, P_{j_2}$, but considering any zero $P_i$ on the boundary of the cylinder. 
By Corollary~4.7 and formula~(4.36) in \cite{AEZ}, the resulting Siegel--Veech constant $c^{\text{pocket}}_{j_1, j_2}$ corresponding to this configuration has the form
\begin{equation} \label{SV-pocket}
c^{\text{pocket}}_{j_1, j_2} = \frac{1}{2\pi^2}.
\end{equation}

\subsection*{Dumbbell configurations} For the second configuration, we still have a single cylinder filled with closed regular geodesics. But this time the cylinder is bounded by saddle connections joining a zero to itself on each side. We assume that the saddle connection bounding the cylinder on one side joins a fixed zero $P_{i_1}$ of order $d_{i_1} \geq 1$ to itself and that the other saddle connection bounding the cylinder on the other side joins a fixed zero $P_{i_2}$ of order $d_{i_2} \geq 1$ to itself (see Figure~\ref{dumbbell}).
Such a cylinder separates the original surface $\W$ in two parts. Let $P_{i_{11}},\dots, P_{i_{1k_1}}$ be the list of singularities (zeros and poles) which get to the first part and $P_{i_{21}},\dots, P_{i_{2k_2}}$ be the list of singularities (zeros and poles) which get to the second part. In particular, we have $i_1 \in \{i_{11},\dots,i_{1k_1}\}$ and $i_2 \in \{i_{21},\dots,i_{2k_2}\}$. We assume that there is not any marked point. Denoting, as usual, by $d_i$ the order of the singularity $P_i$ we can represent the sets (with multiplicities) $\alpha\coloneqq\{d_1,\dots, d_k\}$ as a disjoint union of the two subsets
\[\alpha=\{d_{i_{11}},\dots, d_{i_{1k_1}}\} \sqcup \{d_{i_{21}},\dots, d_{i_{2k_2}}\} \eqqcolon \alpha_1\sqcup\alpha_2.\]
(Recall that $\{d_1,\dots, d_k\}$ denotes all zeros and poles.) This information is considered to be part of the configuration.

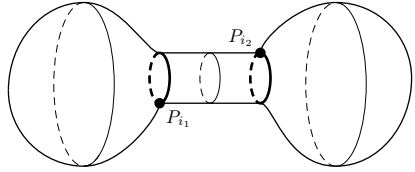
\begin{figure}[h]
\centering
\begin{tikzpicture}
\dumbbellConfiguration[.67]{0,0}
\end{tikzpicture}
\caption{A dumbbell configuration, composed of two flat spheres joined by a cylinder. Each boundary component of the cylinder is a saddle connection joining a zero to itself.}
\label{dumbbell}
\end{figure}

By Theorem~4.8 and equation~(4.38) in \cite{AEZ}, the corresponding Siegel--Veech constant $c^{\text{dumbbell}}_{i_1,i_2;\alpha_1,\alpha_2}$ is given by
\begin{equation} \label{SV-dumbbell}
c^{\text{dumbbell}}_{i_1,i_2;\alpha_1,\alpha_2} = (d_{i_1}+1)(d_{i_2}+1)\frac{(k_1-3)!(k_2-3)!}{(k-4)!}\frac{1}{2\pi^2}
\end{equation}

\subsection{From billiards to flat surfaces}
Recall that in the classical case of a billiard in a rectangle we can glue a flat torus out of four copies of the billiard table and unfold billiard trajectories to flat geodesics of the same length on the resulting flat torus.

\subsection*{Wind-tree model}
The wind-tree model corresponds to a billiard $\Pi$ in the plane endowed with $\Z^2$-periodic horizontally and vertically symmetric right-angled obstacles, where the sides of the obstacles are aligned along the lattice as in Figure~\ref{figu:WTM1} and Figure~\ref{figu:WTMm}.

In the case of the wind-tree model we also start from gluing a flat surface out of four copies of the infinite billiard table $\Pi$. The resulting surface $\X_\infty=\X_\infty(\Pi)$ is $\Z^2$-periodic with respect to translations by vectors of the original lattice. Passing to the $\Z^2$-quotient we get a compact flat surface $\X=\X(\Pi)$. For the case of the original wind-tree billiard, with rectangular obstacles, the resulting flat surface is represented at Figure~\ref{figu:compact-surface}. It has genus~$5$ and belongs to the stratum $\H(2^4)$ (see \cite[\S~3]{DHL} for details).

\begin{figure}[h]
  \includegraphics[width=.65\textwidth]{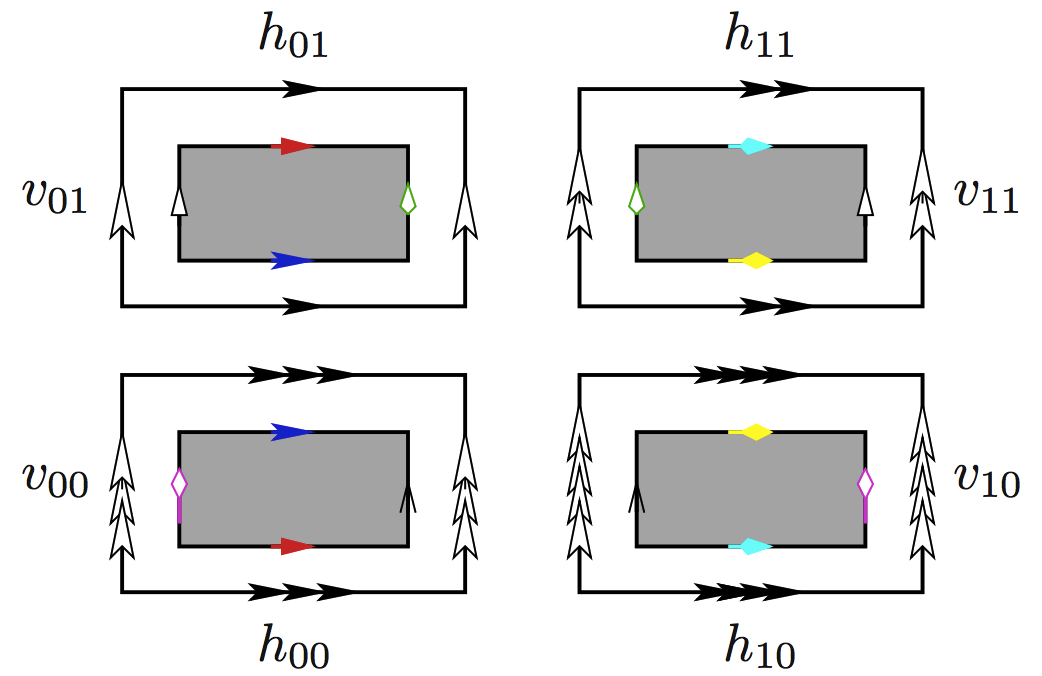}
  \caption{The flat surface $\X$ obtained as quotient over $\Z^2$ of an unfolded wind-tree billiard table.}
  \label{figu:compact-surface}
\end{figure}

Similarly, when the obstacle has $4m$ corners with the angle $\pi/2$ ---and $4(m-1)$ with angle $3\pi/2$---, the same construction gives a flat surface of genus~$4m+1$ in $\H(2^{4m})$, consisting in four flat tori with holes (four copies of a $\Z^2$ fundamental domain of $\Pi$, the holes corresponding to the obstacles) with corresponding identifications, as in the classical setting ($m=1$, see Figure~\ref{figu:compact-surface}). Let $\WT(m)$ denote the set of wind-tree billiards $\Pi$ whose obstacles have $4m$ corners with angle $\pi/2$. The space $\WT(m)$ has a natural Lebesgue measure coming from the consideration of lengths and position of the sides of the obstacle. 
The construction $\Pi\mapsto\X(\Pi)$ defines a map $\WT(m)\to\H(2^{4m})$ and we define $\B(m)$ to be the image of this map, that is, the set of all compact surfaces $\X(\Pi)$ such that $\Pi\in\WT(m)$, and we consider in $\B(m)$ the pushforward of the measure on $\WT(m)$.

Note that any resulting flat surface $\X\in\B(m)$ has (at least) the group $(\Z_2)^3$ as a group of isometries. We have the isometry $\tau_h$, interchanging the pairs of flat tori with holes in the same rows by parallel translations, the isometry $\tau_v$, interchanging columns, and $\iota$, the isometry acting on each of the four tori with holes as the central symmetry with the center in the center of the hole (rotation by $\pi$).

Consider the quotient $\Wh$ of the flat surface $\X$ over the subgroup $(\Z_2)^2$ of isometries spanned by $\tau_h$ and $\iota\circ\tau_v$. The resulting surface $\Wh$ (see Figure~\ref{figu:torus}) belongs to the stratum $\Q(1^{2m},-1^{2m})$. In particular, it has genus~$1$, $\Wh=(\mathbb{T}^2,q_h)$. Similarly, $\Wv=\rquo{\X}{\langle\tau_v,\iota\circ\tau_h\rangle}=(\mathbb{T}^2,q_v)\in\Q(1^{2m},-1^{2m})$.
The surface $\W$ obtained as the quotient of the original flat surface $\X$ over the entire group $(\Z_2)^3$ (see Figure~\ref{figu:sphere}) belongs to the stratum $\Q(1^m,-1^{m+4})$. In particular, it has genus zero, $\W=(\mathbb{CP}^1,q)$. 
Clearly, $\Wh$ and $\Wv$ are ramified double covers over $\W$ with ramification points at four (out of $m+4$) simple poles of the flat surface $\W$ (see \cite[\S~3.1, 3.2]{DZ} for details). Moreover, $\Wh$ and $\Wv$ share three out of their four ramified simple poles.

\begin{figure}[h]
\begin{subfigure}[t]{.45\textwidth}
  \begin{centering}
  \includegraphics[height=4cm]{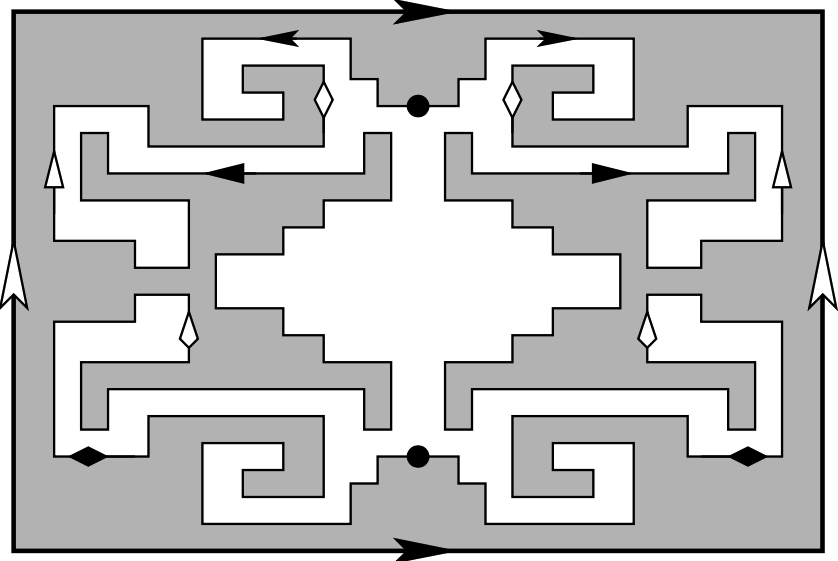}
  \caption{$\Wh=(\mathbb{T}^2,q_h)\in\Q(1^{2m},-1^{2m})$.}
  \label{figu:torus}
  \end{centering}
\end{subfigure}
\begin{subfigure}[t]{.45\textwidth}
  \begin{centering}
  \includegraphics[height=4cm]{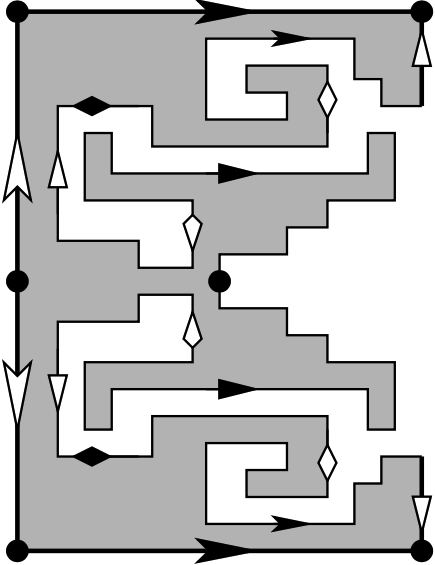}
  \caption{$\W=(\mathbb{CP}^1,q)\in\Q(1^m,-1^{m+4})$.}
  \label{figu:sphere}
  \end{centering}
\end{subfigure}
\caption{The flat surface $\Wh$ is a double cover over the underlying surface $\W$ branched at the four simple poles represented by bold dots.}
\label{figu:t2-cp1}
\end{figure}


Furthermore, the isometries $\tau_h$ and $\tau_v$ decompose the Hodge bundle over $\M$. In fact, we have that \[H^1_\X=E^{++}\oplus E^{+-}\oplus E^{-+}\oplus E^{--},\] where $E^{++}$ is the vector space invariant by $\tau_h$ and $\tau_v$, $E^{+-}$ the vector space invariant by $\tau_h$ and anti-invariant by $\tau_v$, etc. This decomposition is flat, defined over $\Z$ and symplectic; each subbundle is symplectic and the sum is orthogonal with respect to the intersection form.

Consider now the cohomology classes $h,v\in H^1(\X,\Z)$ Poincar\'e-dual to the cycles $h_{00}-h_{01}+h_{10}-h_{11}$ and $v_{00}-v_{10}+v_{01}-v_{11}$ respectively (see Figure~\ref{figu:compact-surface}) as elements of the fiber over the point $\X$ of the (real) Hodge bundle $H^1$ over the $\slr$-orbit closure of $\X\in\B(m)$. The pair $(h,v)\in H^1(\X,\Z^2)$ defines the $\Z^2$-covering $\X_\infty$ of $\X$ and the coordinates of this $\Z^2$-cocycle defining $\X_\infty$ belong to $E^{+-}\oplus E^{-+}$, more precisely, we have that $h\in E^{+-}$ and $v\in E^{-+}$.

We further consider $F^{+-}\subset E^{+-}$, the vector space invariant by $\tau_h$ and $\iota\circ\tau_v$, which is naturally isomorphic to the Hodge bundle over $\Wh=(\mathbb{T}^2,q_h)$. Then, $F^{+-}$ is a two dimensional, defined over $\Z$, flat ---it is locally defined by two cocycles in $H^1(\X,\Z)$ and the Gauss--Manin connection---  and symplectic subbundle of the Hodge bundle. In particular, it is continuous and equivariant (invariant with respect to the Kontsevich--Zorich cocycle). 
Analogously, we consider $F^{-+}\subset E^{-+}$, the vector space invariant by $\tau_v$ and $\iota\circ\tau_h$, with the analogous properties. We have that $h\in F^{+-}$ and $v\in F^{-+}$ (see~\cite[Lemma~3.1]{DZ}).


\begin{theo}[Delecroix--Zorich \cite{DZ}] \label{theo:DZ}
    For almost every billiard $\Pi\in\WT(m)$, the $\mathrm{GL}(2,\R)$-orbit closure of $\W(\Pi)$ coincides with the whole stratum $\Q(1^m,-1^{m+4})$ and the Lyapunov exponents on the $\slr$-orbit closure of $\X(\Pi)$ over the subbundles $F^{+-}$ and $F^{-+}$ are $\pm\delta(m)$, 
    where \[\delta(m)=\frac{(2m)!!}{(2m+1)!!}=4^m\frac{(m!)^2}{(2m+1)!}>0.\]
\end{theo}
Here, the double factorial means the product of all even (correspondingly odd) natural numbers from $2$ to $2m$ (correspondingly from $1$ to $2m+1$). For the original wind-tree model, that is, when $m=1$, this was first shown by Delecroix--Hubert--Leli\`evre~\cite{DHL}. In this case we have, in particular, that $F^{+-}=E^{+-}$, $F^{-+}=E^{-+}$ and $\delta(1)=2/3$.

Since the subbundles $F^{+-}$ and $F^{-+}$ have non-zero Lyapunov exponents and are $2$-dimensional, they are irreducible and then, symplectic (see Remark~\ref{rema:Lyapunov-symplectic}).

In this work, we are concerned with counting closed trajectories in the wind-tree billiard. Obviously, any closed trajectory can be translated by an element in $\Z^2$ to obtain a new closed trajectory. Then, we shall count (isotopy classes of) closed trajectories of bounded length in the wind-tree billiard up to $\Z^2$-translations.
There is a one to one correspondence between billiard trajectories in $\Pi$ and geodesics in $\X_\infty$. But $\X_\infty$ is the $\Z^2$-covering of $\X$ given by $h,v\in H^1(\X,\Z)$, which means that closed curves $\gamma$ in $\X$ lift to closed curves in $\X_\infty$ if and only if $\langle\gamma,h\rangle=\langle\gamma,v\rangle=0$. This is a general fact about $\Z^d$-periodic flat surfaces.

\section{Counting problem in $\Z^d$-periodic flat surfaces} \label{sect:counting-periodic-surfaces}

We consider an infinite $\Z^d$-periodic flat surface $X_\infty$ which is a ramified cover over a compact flat surface $X=(S,\omega)$, the covering group being $\Z^d$, $d\geq 1$. Let $\Sigma$ be the finite set of singularity points of $X$. Since the intersection form $\langle\cdot,\cdot\rangle$ is non-degenerate between $H^1(S\setminus\Sigma,\Z)$ and $H^1(S,\Sigma,\Z)$, every such $\Z^d$-cover is defined by a $d$-tuple of independent elements $\mathbf{f}=(f_1,\dots,f_d)$ in the group of relative cohomology $H^1(S,\Sigma,\Z)$. 

We are interested in counting cylinders in $X_\infty$ modulo $\Z^d$-translations. Cylinders in the cover $X_\infty$ clearly descends to cylinders in $X$, but not the other way around. In fact, by definition of the covering, the monodromy of a closed curve $\gamma$ is translation by $(\langle\gamma,f_i\rangle)_{i=1}^{d}\in\Z^d$. It follows that cylinders in the cover $X_\infty$ are exactly the lift of those cylinders $C$ in $X$ such that its core curve $\gamma_C$ verifies $\langle\gamma_C,f_i\rangle = 0$, for each $i=1,\dots,d$. Note that, in this case, the monodromy is always trivial and cylinders in $X_\infty$ are always isometric to their projection on $X$. When a cylinder $C$ does not satisfy this condition, it lifts to $X_\infty$ as a strip, isometric to the product of an open interval and a straight line.

We restrict ourselves to the case when $\mathbf{f}$ is an absolute covector, that is, it is a $d$-tuple of independent elements in the group of \emph{absolute} cohomology $H^1(S,\Z)$.
Let $\M$ be the $\slr$-orbit closure of $X$, $F$ be an equivariant subbundle of the Hodge bundle over $\M$ and $f\in F_X$.

Note that cylinders $C$ in $X$ such that $\langle\gamma_C,f\rangle = 0$, split naturally into two families:
(a) the family of cylinders such that $\langle\gamma_C,h\rangle = 0$ for all $h\in F_X$, which we call \emph{$F$-good cylinders}, and
(b) the family of cylinders that are not $F$-good, but $\langle\gamma_C,f\rangle = 0$. These later are called \emph{$(F,f)$-bad cylinders}.
The notion of $F$-good cylinders was first introduced by Avila--Hubert~\cite{AH} in order to give a geometric criterion for recurrence of $\Z^d$-periodic flat surfaces.

Thus, counting cylinders in the $\Z^d$-periodic flat surface can be reduced to counting separately cylinders which are $(\oplus_j F^{(j)})$-good cylinders and $(F^{(j_i)},f_i)$-bad cylinders in the compact surface, for some appropriate subbundles $(F^{(j)})_j$.

\begin{rema} \label{rema:hyp-subbundle} When $F$ is symplectic, in particular, if $\Lambda(F)\neq\{0\}$ (see Remark~\ref{rema:Lyapunov-symplectic}), $F$-good cylinders are exactly those that $\pr_{F_X} \gamma_C = 0$.
If, in addition, $F$ is $2$-dimensional (in particular, irreducible if $\Lambda(F)\neq\{0\}$), $C$ is an $(F,f)$-bad cylinder if and only if $\pr_{F_X} \gamma_C \neq 0$ is colinear to $f$.
\end{rema}

Since the Kontsevich--Zorich cocycle preserves the intersection form and $F$ is equivariant, it is clear that the set of $F$-good cylinders is $\slr$-equivariant. Then, classical results can be applied. In particular, applying the main result of \cite{EMa}, if there is at least one $F$-good cylinder in $X$, then we can deduce that $F$-good cylinders have quadratic asymptotic growth rate (with positive Siegel--Veech constant) for $\nu_\M$-almost every flat surface in $\M$, the $\slr$-orbit closure of $X$.
However, this is no longer true in the case of $(F,f)$-bad cylinders.

For $f\in F_X$ define the set $V_F(f)$ of holonomy vectors of $(F,f)$-bad cylinders in $X$.
We have that $V_F(A(g,X)f)=gV_F(f)$, since $F$ is equivariant and the Kontsevich-Zorich cocycle respects the intersection form. Finally, let \[N_F(L,f) = \# V_F(f)\cap B(L)\] be the number of $(F,f)$-bad cylinders in $X$ of length bounded by $L$.

\section{Bad cylinders have subquadratic asymptotic growth rate} \label{sect:bad-cylinders}

In this section, we prove the following general result about bad cylinders which applies to \emph{some} $\Z^d$-periodic flat surfaces and, in particular, to the family of wind-tree models we are interested in.

\begin{theo} \label{theo:bad-cylinders} Let $X$ be a flat surface and $F$ a $2$-dimensional equivariant continuous subbundle of the Hodge bundle on $\M$, the $\slr$-orbit closure of $X$. Suppose that $F$ is defined over $\Z$ and has non-zero Lyapunov exponents. Then, for all $f\in F_X$ the number $N_F(L,f)$, of $(F,f)$-bad cylinders in $X$ of length at most $L$, has subquadratic asymptotic growth rate, that is, $N_F(L,f)=o(L^2)$ or, which is the same, \[\lim_{L\to\infty} \frac{N_F(L,f)}{\pi L^2} = 0.\]
\end{theo}

\begin{rema}
When $F$ is $2$-dimensional, symplectic (in particular, when it has non-zero Lyapunov exponents) and defined over $\Z$, if $f\in F_X$ is not colinear to an integer cocycle, then, there are no $(F,f)$-bad cylinders, since $\pr_{F_X} \gamma_C$ is always a rational multiple of an integer cocycle. Since the notion of bad cylinder is clearly projective, the proof of Theorem~\ref{theo:bad-cylinders} is then reduced to prove the conclusion only for $f\in F_X(\Z)$, instead that for all $f\in F_X$. 
\end{rema}

To prove Theorem~\ref{theo:bad-cylinders} we use technology for asymptotic formulas for counting closed geodesics developed by Eskin--Masur~\cite{EMa}. In particular, the following proposition, which is a restatement of Proposition~{3.5} and Lemma~{8.1} in \cite{EMa}, is a key step in the proof.

\begin{prop}[Eskin--Masur] \label{prop:bound} Let $\mathcal{V}\subset \R^2\setminus \{0\}$, define $\mathcal{N}(T,\mathcal{V})\coloneqq\# \mathcal{V}\cap B(T)$ and suppose that $\mathcal{N}(T,\mathcal{V}) < \infty$ for all $T>0$. Then, for all $\rho,t>0$ \[\mathcal{N}(2\rho e^t,\mathcal{V})-\mathcal{N}(\rho e^t,\mathcal{V}) \leq c(\rho) e^{2t}\int_0^{2\pi} \mathcal{N}(4\rho,g_tr_\theta \mathcal{V})\d\theta.\]
\end{prop}

Hence, the proof of Theorem~\ref{theo:bad-cylinders} is reduced to show the following.

\begin{theo} \label{theo:lim} Under the hypothesis of Theorem~\ref{theo:bad-cylinders}, for every $f\in F_X(\Z)$ and all $\rho>0$, \[\lim_{t\to\infty} \int_0^{2\pi} N_F(\rho,A(g_t, r_\theta X)f)\d\theta = 0.\]
\end{theo}

\begin{proof}[Proof of Theorem~\ref{theo:bad-cylinders}]

It is clear that $V_F(\cdot)\subset \R^2\setminus \{0\}$ is $\slr$-equivariant and $N_F(L,f)$ is finite, since it is bounded by $N(L,X)$, the number of all cylinders of length bounded by $L$ and $N(L,X)\leq c(X)L^2$ (\cite{Ma2}). Then, by Proposition~\ref{prop:bound}, we have that, for all $f\in F_X(\Z)$, all $\rho>0$ and all $t>0$, \[N_F(2\rho e^t,f)-N_F(\rho e^t,f) \leq c(\rho) e^{2t}\int_0^{2\pi} N_F(4\rho,A(g_t, r_\theta X)f)\d\theta.\]

But then, by Theorem~\ref{theo:lim}, \[\limsup_{t\to\infty} \frac{N_F(2\rho e^t,f)-N_F(\rho e^t,f)}{\rho^2 e^{2t}} \leq \frac{c(\rho)}{\rho^2}\lim_{t\to\infty} \int_0^{2\pi} N_F(4\rho,A(g_t, r_\theta X)f)\d\theta = 0.\] That is \[\limsup_{T\to\infty} \frac{N_F(2T,f)-N_F(T,f)}{T^2} = 0.\]
It follows that
\begin{align*}
\bar{c}_F(x) & \coloneqq \limsup_{L\to\infty} \frac{N_F(L,f)}{\pi L^2} = \limsup_{T\to\infty} \frac{1}{4\pi}\frac{N_F(2T,f)}{T^2} \\
 & = \frac{1}{4\pi} \limsup_{T\to\infty} \left(\frac{N_F(2T,f)-N_F(T,f)}{T^2} + \frac{N_F(T,f)}{T^2} \right) \\
 & \leq \frac{1}{4\pi}\left(\limsup_{T\to\infty} \frac{N_F(2T;X,f)-N_F(T;X,f)}{T^2} + \limsup_{T\to\infty} \frac{N_F(T;X,f)}{T^2} \right) \\
 & = \frac{1}{4\pi}\left( 0 + \bar{c}_F(x) \right) = \frac{1}{4\pi}\bar{c}_F(x)
\end{align*}
and then, $\bar{c}_F(x) = 0$. We conclude that \[\lim_{L\to\infty} \frac{N_F(L,f)}{\pi L^2} = 0.\]
\end{proof}

\subsection{Proof of Theorem~\ref{theo:lim}}

In order to show that \[\lim_{t\to\infty} \int_0^{2\pi} N_F(\rho, A(g_t, r_\theta X)f)\d\theta = 0,\] we split the integral in whether $g_tr_\theta X\in K_\epsilon=\{\sys \geq \epsilon\}$ or not, and show that both parts tend to zero as $t$ tends to infinity and $\epsilon$, to zero.

When $g_tr_\theta X\in K_\epsilon$, the corresponding part of the integral tends to zero as a consequence of the following proposition, whose proof is postponed to \S~\ref{sec:prop:zero}.

\begin{prop} \label{prop:zero} Under the hypothesis of Theorem~\ref{theo:lim}, for all $f\in F_X(\Z)$, all $\rho,\epsilon > 0$ and almost every $\theta$ \[ N_F(\rho, A(g_t, r_\theta X)f)\cdot\ind_{K_\epsilon}(g_tr_\theta X)=0\] for sufficiently large $t$, $t \geq t_0(x,\rho,\epsilon,\theta)$.
\end{prop}

\begin{rema} 
The intuition behind this apparently technical proposition is the following. By hypothesis, the Lyapunov exponent of $f\in F_X(\Z)$ is positive and then, for almost every $\theta$, $A(g_t,r_\theta X)f$ becomes very long for large $t$. Without loss of generality, we can suppose that $f$ is primitive. Therefore, no short cycle (of length bounded by $\rho$) can have projection on $F_X$ colinear to $A(g_t,r_\theta X)f$, because this latter is primitive and longer. We formalize this idea in \S~\ref{sec:prop:zero}.
\end{rema}

Recall that $N_F(L,f)\leq N(L,X)$. Furthermore, $N(\rho,\cdot)$ is bounded in $K_\epsilon$. Indeed, by Theorem~\ref{EM:5.1}, for $\delta = 1$,
\[\ind_{K_\epsilon} N(\rho,\cdot) \leq \ind_{K_\epsilon}\frac{c(\rho,1)}{\sys^{2}} \leq \frac{c(\rho,1)}{\epsilon^{2}} = c(\rho,\epsilon).\]
Then, for fixed $\rho,\epsilon > 0$,
\begin{multline*}
\int_0^{2\pi} N_F(\rho, A(g_t, r_\theta X)f)\cdot\ind_{K_\epsilon}(g_tr_\theta X)\d \theta \\
  \leq c(\rho,\epsilon)\cdot|\{\theta\in [0,2\pi):N_F(\rho, A(g_t, r_\theta X)f)\cdot\ind_{K_\epsilon}(g_tr_\theta X) \neq 0 \}|,
\end{multline*}
where $|\cdot|$ is the Lebesgue measure on $[0,2\pi)$. Finally, by Proposition~\ref{prop:zero}, the right side of the inequality tends to zero as $t$ tends to infinity. That is,
\begin{equation} \label{equa:int-ke}
\lim_{t\to\infty} \int_0^{2\pi} N_F(\rho, A(g_t, r_\theta X)f)\cdot\ind_{K_\epsilon}(g_tr_\theta X)\d \theta = 0.
\end{equation}

For the rest of the integral we use the following.

\begin{lemm} \label{lemm:circles-cusp} For any flat surface $X$, any $\beta < 2$ and all $\epsilon > 0$, \[|\{\theta\in [0,2\pi):\sys(g_tr_\theta X) < \epsilon \}| < c(X,\beta)\epsilon^\beta \] for all $t > 0$.
\end{lemm}

\begin{proof}
\begin{align*}
|\{\theta\in [0,2\pi):\sys(g_tr_\theta X) < \epsilon \}|
 & = \int_0^{2\pi} \ind_{\sys < \epsilon}(g_tr_\theta X)\d \theta \\
 & \leq \int_0^{2\pi} \ind_{\sys < \epsilon}(g_tr_\theta X)\cdot \frac{\epsilon^\beta}{\sys(g_tr_\theta X)^\beta}\d \theta \\
 & \leq \epsilon^\beta \int_0^{2\pi} \frac{\d \theta}{\sys(g_tr_\theta X)^\beta} 
\end{align*}
Then, by Theorem~\ref{EM:5.2}, we conclude that \[ \mid\{\theta\in [0,2\pi):\sys(g_tr_\theta X) < \epsilon \}| \leq c(X,\beta) \epsilon^\beta.\]
\end{proof}

Moreover, since $N_F(\rho,f)\leq N(\rho,X)$ and, by Theorem~\ref{EM:5.1}, for any $\delta > 0$ \[ N(\rho,X) \leq \frac{c(\delta,\rho)}{\sys(X)^{1+\delta}},\]  it follows that

{\small
\begin{align*}
\mathrlap{\int_0^{2\pi} N_F(\rho, A(g_t, r_\theta X)f)\cdot\ind_{\sys < \epsilon}(g_tr_\theta X)\d \theta} \hspace{6em} & \\
 & \leq \sum_{n=0}^\infty \int_0^{2\pi} N(\rho, g_tr_\theta X)\cdot \ind_{\sys \in \left[\frac{\epsilon}{2^{n+1}},\frac{\epsilon}{2^{n}}\right)} (g_tr_\theta X)\d \theta \\
 & \leq c(\delta,\rho)\sum_{n=0}^\infty \int_0^{2\pi} \frac{1}{\sys(g_tr_\theta X)^{1+\delta}}\cdot \ind_{\sys \in \left[\frac{\epsilon}{2^{n+1}},\frac{\epsilon}{2^{n}}\right)} (g_tr_\theta X)\d \theta \\
 & \leq c(\delta,\rho)\sum_{n=0}^\infty \int_0^{2\pi} \frac{1}{\left(\frac{\epsilon}{2^{n+1}}\right)^{1+\delta}}\cdot \ind_{\sys \in \left[\frac{\epsilon}{2^{n+1}},\frac{\epsilon}{2^{n}}\right)} (g_tr_\theta X)\d \theta \\
 & \leq c(\delta,\rho)\sum_{n=0}^\infty \frac{2^{(n+1)(1+\delta)}}{\epsilon^{1+\delta}} \int_0^{2\pi} \ind_{\sys < \frac{\epsilon}{2^{n}}}(g_tr_\theta X)\d \theta \\
 & \leq c(\delta,\rho )\sum_{n=0}^\infty \frac{2^{(n+1)(1+\delta)}}{\epsilon^{1+\delta}} |\{ \theta\in [0,2\pi): \sys(g_tr_\theta X) < \frac{\epsilon}{2^{n}} \}|.
\end{align*}
}

Then, by Lemma~\ref{lemm:circles-cusp}, for $1+\delta < \beta < 2$,
\begin{align}
\lim_{t\to\infty} \int_0^{2\pi} N_F(\rho, A(g_t, r_\theta X)f)\cdot\ind_{\sys < \epsilon}(g_tr_\theta X)\d \theta
 & \leq c(\delta,\rho)\sum_{n=0}^\infty \frac{2^{(n+1)(1+\delta)}}{\epsilon^{1+\delta}} c(X,\beta) \frac{\epsilon^\beta}{2^{n\beta}} \nonumber \\
 & \leq c(\delta,\rho,X,\beta) \epsilon^{\beta - (1+\delta)}. \label{equa:int-coke} 
\end{align}

Joining both parts of the integral, (\ref{equa:int-ke}) and (\ref{equa:int-coke}), we obtain that, for every $\epsilon,\delta,\rho>0$, $f\in F_X(\Z)$ and $1+\delta<\beta<2$, \[\lim_{t\to\infty} \int_0^{2\pi} N_F(\rho, A(g_t, r_\theta X)f)\d\theta \leq 0 + c(\delta,\rho,X,\beta) \epsilon^{\beta - (1+\delta)}.\] Then, fixing $\rho>0$, $0<\delta < 1$ and $1+\delta<\beta<2$, and letting $\epsilon\to 0$, we conclude that \[\lim_{t\to\infty} \int_0^{2\pi} N_F(\rho, A(g_t, r_\theta X)f)\d\theta = 0.\]
\qed

\subsection{Proof of Proposition~\ref{prop:zero}} \label{sec:prop:zero}


The first step is to show that, for a cylinder, being bounded in length implies having bounded projection in $F_X$. 

\begin{lemm} \label{lemm:bounded-holonomy-homology} Let $\rho>0$ and $K\subset\M$ be a compact subset. Then, for all $X'\in K$ and all cylinder $C$ on $X'$ such that $|\hol[\omega']\gamma_C|\leq\rho$ we have that
\[\| \pr_{F_{X'}}[\gamma_C]\|_{\omega'} \leq c(\rho,K,F) .\]
\end{lemm}


\begin{proof}

Let $\mathbf{C}(\rho,X')$ be the finite set of cylinders on $X'$ of length at most $\rho$.
Then, $c_0(\rho,X',F)=\max\{\|\pr_{F_{X'}}[\gamma]\|_{\omega'}:C\in\mathbf{C}(\rho,X')\}$ is finite.

Define $\Gamma(\rho,X')=\{\gamma_C:C\in\mathbf{C}(\rho,X')\}$. Then, since $F$ is continuous, $\pr_{F_{(\cdot)}}(\cdot)$ is continuous and since the Hodge norm $\|\cdot\|_{(\cdot)}$ is continuous, there is a neighborhood $U(X')$ of $X'$ in $\M$ such that, for all $\bar{X}=(\bar{S},\bar{\omega})\in U(X')$,
\begin{itemize}
\item $\Gamma(\rho,\bar{X})\subset \Gamma(2\rho,X')$ (after local identification), and
\item $\|\pr_{F_{\bar{X}}}\cdot\|_{\bar{\omega}}\leq 2\|\pr_{F_{X'}}\cdot\|_{\omega'}$.
\end{itemize}
Therefore, if $\bar{C}$ is a cylinder in $\bar{X}\in U(X')$ with $|\hol[\bar{\omega}]\gamma_{\bar{C}}|\leq \rho$, then
\[\|\pr_{F_{\bar{X}}}[\gamma_{\bar{C}}]\|_{\bar{\omega}} \leq 2\|\pr_{F_{X'}}[\gamma_{\bar{C}}]\|_{\omega'} \leq 2c_0(2\rho,X',F) \eqqcolon c(\rho,X',F). \]

Since $U(X')$ is open and $K$ is compact, there is a finite set $A\subset K$ such that $K\subset \cup_{X'\in A} U(X')$. We conclude, taking $\max_{X'\in A} c(\rho,X',F)$ to be $c(\rho,K,F)$.
\end{proof}

Since $F$ is $2$-dimensional and has non-zero Lyapunov exponents, it is symplectic and its Lyapunov spectrum is symmetric (see Remark~\ref{rema:Lyapunov-symplectic}), say $\Lambda(\M,F)=\{\pm\lambda\}$, $\lambda>0$. Moreover, since $f\in F_X(\Z)$ is an integer covector, its associated Lyapunov exponent has to be positive. Then, for almost every $\theta$, we have that \[\lim_{t\to\infty} \frac{\log \| A(g_t,r_\theta X) f\|_{g_tr_\theta\omega}}{t} = \lambda>0, \] in particular, for almost every $\theta$ and sufficiently large $t$, $t\geq t_0(r_\theta X,f)$, 
\begin{equation} \label{equa:ineq1}
\| A(g_t,r_\theta X) f\|_{g_tr_\theta\omega} \geq e^{\frac{\lambda}{2}t}.
\end{equation}

Recall that, since $F$ is defined over $\Z$, $F_{X}^\pr(\Z)=\pr_{F_X}H^1_X(\Z)$ is a lattice and $F_X(\Z)\subset F_X^\pr(\Z)$. Let $m=m(f)$ be a positive integer such that $\frac{1}{m}f$ is a primitive element in the lattice $F_X^\pr(\Z)$, and let $c(\rho,\epsilon,F)$ be the constant given by Lemma~\ref{lemm:bounded-holonomy-homology} for $K = K_\epsilon$. Then, for large $t$, $t\geq t_0(\epsilon,\rho,f)$,
\begin{equation} \label{equa:ineq2}
e^{\frac{\lambda}{2}t} > m(f) c(\rho,\epsilon,F).
\end{equation}
Therefore, putting (\ref{equa:ineq1}) and (\ref{equa:ineq2}) together, for almost every $\theta$ and all $t$ sufficiently large, $t\geq t_0(\epsilon,\rho,\theta,X,f)$, we have that \[\| A(g_t,r_\theta X) f\|_{g_tr_\theta\omega} \geq e^{\frac{\lambda}{2}t} > m(f) c(\rho,\epsilon,F).\]

Fix $\theta$ and $t$ as before, consider $X_t=g_tr_\theta X$, $\omega_t=g_tr_\theta\omega$ and $f_t=A(g_t,r_\theta X)f$, and suppose that $X_t\in K_\epsilon$.
Now, if $\gamma$ is the core curve of a cylinder in $X_t$ such that $|\mathrm{hol}_{\omega_t} \gamma| \leq \rho$, then
\[ \|\pr_{F_{X_t}}[\gamma]\|_{\omega_t} \leq c(\rho,\epsilon,F) < \frac{1}{m}\|f_t\|_{\omega_t},\]
where the first inequality is given by Lemma~\ref{lemm:bounded-holonomy-homology}, for $X'=X_t$ and $K=K_\epsilon$.

%

Recall that under our hypothesis, an $(F,f_t)$-bad cylinder $C$ in $X_t$ has to verify that $\pr_{F_{X_t}}[\gamma_C]\neq 0$ is colinear to $f_t$ (see Remark~\ref{rema:hyp-subbundle}).
But no element in $F_{X_t}^\pr(\Z)$ colinear to $f_t$ can be shorter than $\frac{1}{m}f_t$, since this last is primitive in the lattice $F_{X_t}^\pr(\Z)$, by definition of $m$ and, evidently, $\pr_{F_{X_t}}[\gamma]$ belongs to $F_{X_t}^\pr(\Z)$

Then $\gamma$, as before, cannot be the core curve of an $(F,f_t)$-bad cylinder in $X_t$. And thus, $N_F(\rho,A(g_t,r_\theta X)f))=N_F(\rho,f_t)= 0$, for $\theta$ and $t$ as before. That is, for all $f\in F_X(\Z)$, all $\rho,\epsilon > 0$ and almost every $\theta$ \[ N_F(\rho, A(g_t, r_\theta X)f)\cdot\ind_{K_\epsilon}(X_t)=0\] for sufficiently large $t$, $t \geq t_0(x,\rho,\epsilon,\theta)$. \qed

\section{Application to wind-tree models} \label{sect:application-wtm}

In this section we apply previous discussion to wind-tree models. As we have seen, there is an identification between cylinders (up to $\Z^2$-translations) in the infinite billiard $\Pi\in\WT(m)$ and the union of $(F^{+-}\oplus F^{-+})$-good cylinders, $(F^{+-},h)$-bad cylinders and $(F^{-+},v)$-bad cylinders in $\X=\X(\Pi)\in\B(m)$.
Moreover, the subbundles $F^{+-}$ and $F^{-+}$ are always $2$~dimensional flat subbundles defined over $\Z$ and, by Theorem~\ref{theo:DZ}, we know that for almost every $\X\in\B(m)$, $\Lambda(\M,F^{+-})=\Lambda(\M,F^{-+})=\{\pm\delta(m)\}$, where $\M$ is the $\slr$-orbit closure of $\X$ and $\delta(m)>0$. In particular, for almost every $\X\in\B(m)$, $F^{+-}$ and $F^{-+}$ satisfy the hypothesis of Theorem~\ref{theo:bad-cylinders}.

This suffices for the almost everywhere statement of Theorem~\ref{theo:main}, but it does not for the everywhere statement of Theorem~\ref{theo:weak}. However, an adaptation of Forni's criterion~\cite{Fo} allows us to prove that the top Lyapunov exponents of $F^{+-}$ and $F^{-+}$ are in fact positive.

\begin{theo}[Forni's criterion for integer equivariant subbundles] \label{theo:Forni}
Let $\M$ be an affine invariant manifold and $F$ be an equivariant subbundle of the Hodge bundle on $\M$ defined over~$\Z$. Suppose that there exists a flat surface $X\in\M$ and a family of parallel closed geodesics in $X$ such that the space generated by the (Poincar\'e dual of the) homology classes of these closed geodesics is a subspace of $F_X$ of dimension $d\geq 1$. Then, the top $d$ Lyapunov exponents on $F$ are strictly positive, that is, \[\lambda_1(\M,F) \geq \dots \geq \lambda_{d}(\M,F)>0.\]
\end{theo}

\begin{proof} The proof follows as the original proof of \cite[Theorem~1.6]{Fo}. In fact, as communicated to as by C.~Matheus, the main steps of the proof are:
\begin{enumerate}[leftmargin=*]
  \item \cite[\S~3]{Fo}: The unstable bundle of the Kontsevich--Zorich cocycle is $\nu_\M$-almost everywhere transverse to all integral isotropic subspaces (see \cite[Lemma~3.1]{Fo}). In our case, we can restrict the unstable bundle to the equivariant subbundle $F$ and this statement remains true since the subbundle $F$ is defined over~$\Z$.
  \item \cite[\S~4]{Fo}: $d\times d$-block of the second fundamental form converges to $-\mathrm{Id}$ along an isotropic subspace transverse to the (Poincar\'e dual of the) $d$-dimensional subspace generated by the closed geodesics (see \cite[Lemma~4.4]{Fo}). This remains true when restricting to the subbudle $F$; the proof relies only on classical formulas for the period matrix near the boundary of the Deligne--Mumford compactification of the moduli space of abelian differentials (see \cite[Lemma~4.1]{Fo}).
  \item \cite[\S~5]{Fo}: Finally, the proof of \cite[Theorem~1.6]{Fo} remains valid since the argument combines the two previous points with a hypothesis of local product structure, which is always true after Eskin--Mirzakhani~\cite{EMi}.
\end{enumerate}
\end{proof}

\begin{coro} \label{coro:hypothesis}
For every $\X\in\B(m)$, the subbundles $F^{+-}$ and $F^{-+}$ defined on the $\slr$-orbit closure of $\X$, satisfy the hypothesis of Theorem~\ref{theo:bad-cylinders}.
\end{coro}
\begin{proof}
We already know that the subbundles $F^{+-}$ and $F^{-+}$ are $2$~dimensional flat subbundles defined over $\Z$. Then, it remains to prove that they have non-zero Lyapunov exponents. 

Let $F_\X$ be the (Poincar\'e dual of the) symplectic subspace generated by cycles $h_{00},h_{10},h_{01},h_{11},v_{00},v_{10},v_{01},v_{11}$ (see Figure~\ref{figu:compact-surface}). This defines a flat (that is, a locally constant) subbundle of the Hodge bundle, which is clearly defined over $\Z$. Moreover, $F$ has rank~$8$ and is symplectic. In particular, its Lyapunov spectrum is symmetric. Taking the closed geodesics given by $h_{00},h_{10},h_{01},h_{11}$, which are horizontal and homologically independent, and applying Theorem~\ref{theo:Forni}, we conclude that $F$ has $4$ positive Lyapunov exponents and therefore all eight Lyapunov exponents are non-zero.
Finally, we note that $F^{+-}$ and $F^{-+}$ are subbundles of $F$ and, in particular, their Lyapunov spectra are contained in the one of $F$. Thus, they have non-zero Lyapunov exponents. 
\end{proof}

Thus, by Theorem~\ref{theo:bad-cylinders}, $(F^{+-},h)$-bad cylinders and $(F^{-+},v)$-bad cylinders in $\X$ have subquadratic asymptotic growth rate, proving Theorem~\ref{theo:bad-cylinders-WTM}. Thus, asymptotic formulas for the wind-tree model correspond to those of $(F^{+-}\oplus F^{-+})$-good cylinders. In particular, this justifies why we can conclude Theorem~\ref{theo:weak}, so we have weak asymptotic formulas for \emph{every} wind-tree model.

For simplicity, henceforth, we will call simply good cylinders the $(F^{+-}\oplus F^{-+})$-good cylinders, and by bad cylinders we will refer to $(F^{+-},h)$ and $(F^{-+},v)$-bad cylinders.

As a direct consequence of Theorem~\ref{theo:bad-cylinders-WTM} and an adapted version of Theorem~\ref{theo:EM-SV-constant} (see Remark~\ref{rema:EM-SV-constant}), we have the following.

\begin{coro}
For almost every wind-tree billiard $\Pi\in\WT(m)$, the number $N(L,\Pi)$ of closed billiard trajectories of length bounded by $L$ in $\Pi$ has quadratic asymptotic growth rate, 
\[N(L,\Pi)\sim \frac{1}{4} c_{good}(\M)\frac{\pi L^2}{\mathrm{Area}\left(\Pi/\Z^2\right)},\]
where $c_{good}(\M)$ is the Siegel-Veech constant associated to the counting problem of good cylinders in $\M$, the $\slr$-orbit closure of $\X(\Pi)$.
\end{coro}
The factor $1/4$ coming from the fact that $\mathrm{Area}\left(\X(\Pi)\right)=4\cdot\mathrm{Area}\left(\Pi/\Z^2\right)$.

In addition, a cylinder in $\X$ is a good cylinder if (and only if) the homology class of its core curve projects trivially to $F^{+-}$ and to $F^{-+}$ (see Remark~\ref{rema:hyp-subbundle}). We have also the following useful characterization of good cylinders (see Figure~\ref{diagram} for notation).

\begin{lemm} \label{homology-good} Let $C$ be a cylinder in $\X$. Then $C$ is a good cylinder in $\X$ if and only if the core curve of $C$ projects to ho\-mo\-lo\-gi\-cally trivial curves in $\Wh$ and $\Wv$.
\end{lemm}
\begin{proof} Let $\gamma$ be the core curve of $C$. Then $C$ is an $F^{+-}$-good cylinder in $\X$ if and only if $\pr_{F^{+-}} [\gamma] = 0$. But $F^{+-}$ is naturally isomorphic to $H^1(\Wh)$ by the pushforward of the covering map $\ph$. Then  $\pr_{F^{+-}} [\gamma] = 0$ if and only if ${\ph}_*[\gamma] = [\ph \gamma] = 0$. Analogously, the same holds for $F^{-+}$ and $\Wv$. And good cylinders are exactly those which are $F^{+-}$ and $F^{-+}$-good cylinders.
\end{proof}

\begin{figure}[h]
\centering
\begin{tikzpicture}
\coveringDiagram
\end{tikzpicture}
\caption{Surfaces and covering maps notation}
\label{diagram}
\end{figure}
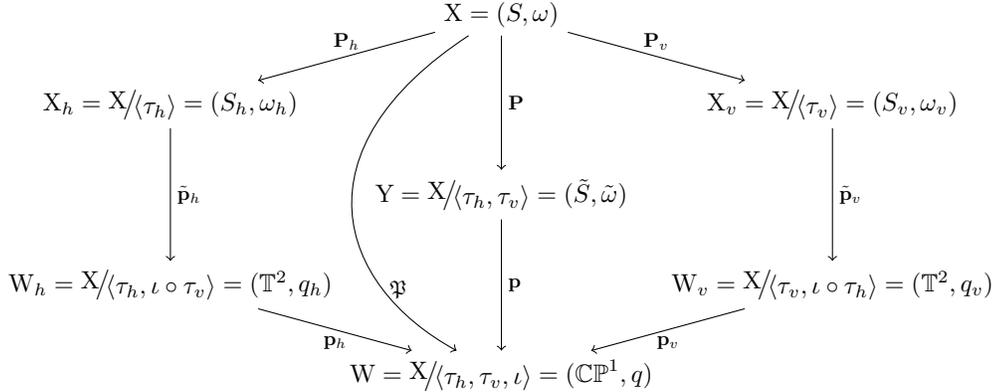

Then, good cylinders in $\X$ are exactly those which project to ho\-mo\-lo\-gi\-cally trivial cylinders in the flat surfaces $\Wh$ and $\Wv$.
Cylinders in $\X$ also project to the flat surface $\W$, of genus zero. The $\slr$-orbit closure $\M$ of $\X$ projects to the $\slr$-orbit closure $\L$ of $\W$, and for almost every $\X\in\B(m)$, $\R\L$ coincides with the whole stratum $\Q(1^m,-1^{m+4})$ (\cite[Proposition~2]{DZ}). Moreover, we have seen in \S~\ref{sect:configurations-sphere} that generic flat surfaces in $\Q(1^m,-1^{m+4})$ have only two types of configurations of cylinders, the so called pocket and dumbbell configurations. But generic flat surfaces are not pertinent to our study. In fact, the set of flat surfaces $\W\in\Q(1^m,-1^{m+4})$ coming from wind-tree billiards 
is negligible. However, we have the following.

\begin{prop} For almost any wind-tree billiard $\Pi\in\WT(m)$ the following property holds. Consider a cylinder in $\W(\Pi)=\rquo{\X(\Pi)}{\langle\iota,\tau_h,\tau_v\rangle}$ and suppose it is not horizontal nor vertical. Then, the cylinder make part of one of the configurations described in \S~\ref{sect:configurations-sphere}, that is, a pocket or a dumbbell configurations.
\end{prop}
\begin{proof} See~\cite[Proposition~2.2]{AEZ} (the proof of which mimics the proof of \cite[Theorem~7.4]{EMZ}).
\end{proof}

\begin{coro}
For almost every wind-tree billiard $\Pi\in\WT(m)$, 
\[c_{good}(\M)=c_{good}^{\text{pocket}}(\M)+c_{good}^{\text{dumbbell}}(\M),\]
where $c_{good}^{\text{pocket}}(\M)$ (resp. $c_{good}^{\text{dumbbell}}(\M)$) corresponds to the Siegel--Veech constant associated to the counting problem of configurations of good cylinders in $\M$, the $\slr$-orbit closure of $\X(\Pi)$, such that those configurations project to pocket (resp. dumbbell) configurations in $\Q(1^m,-1^{m+4})$.

\end{coro}

It follows that the study of configurations of cylinders on generic flat surfaces in $\Q(1^m,-1^{m+4})$ suffices for our purposes.

\section{Configurations of good cylinders} \label{sect:configurations}

Here we show which conditions a cylinder in $\W=(\mathbb{CP}^1,q)\in\L=\Q(1^m,-1^{m+4})$ has to satisfy so that it lifts to a good cylinder in $\X=(S,\omega)\in\M$, and then we interpret this in terms of configurations of generic surfaces of genus zero, that is, pocket and dumbbell configurations (see \S~\ref{sect:configurations-sphere}).

Recall that, by Lemma~\ref{homology-good}, a cylinder in $\X$ is good if it projects to a ho\-mo\-lo\-gi\-cally trivial cylinder in the surfaces $\Wh$ and $\Wv$, of genus~$1$. Then, our classification will consist in finding the configurations on $\W$ which lift to ho\-mo\-lo\-gi\-cally trivial closed geodesics in $\Wh$ and $\Wv$.

Since there are clear analogies between objects with subindex $h$ and subindex $v$ (see Figure~\ref{diagram}), in this section we will use the label $\si$ for both labels $h$ and $v$. Thus, any result in terms of labels $\si$ will give the corresponding result for $h$ and $v$.

\subsection{Cylinders in $\W$ who lift to good cylinders in $\X$} \label{sect:which-good}

Let $C$ be a cylinder in the genus zero surface $\W$. Then, since all curves are ho\-mo\-lo\-gi\-cally trivial on $\W$, the core curve of $C$, say $\gamma$, cuts the surface in two components, say $\W_1$ and $\W_2$.

For our purposes here, the only relevant information about $C$ we need, is the number $q_l$ of cone singularities of angle $3\pi$ and the number $r_l$ of ramified poles in $\W_l$ for the double cover $\p_\si:\W_\si\to\W$, $l=1,2$. The number $p_l$, of unramified poles for $\p_\si$ in $\W_l$ is also relevant, but since $\W_l$ is a genus zero surface with only simple zeros and poles, and a single boundary component, then 
\[4g(\W_l)-4 = -4 = q_l - p_l - r_l - 2,\]
and $p_l$ can be written in terms of $q_l$ and $r_l$ as $p_l = q_l - r_l + 2$, $l=1,2$. Also, $q_2 = m-q_1$ and $r_2=4-r_1$, so we will only consider $r=r_1$ and $q=q_1$.

Remark that the number $r$ depends on the configuration as well as on the double cover $\p_\si$ (of which there are two, $\ph$ and $\pv$), while $q$ does not depend on the double cover. Call then, the former number $r_\si = r(C,\p_\si)$. Furthermore, since $\W_1$ and $\W_2$ were arbitrarily chosen, we can fix them such that $r_\si=r_1\leq r_2$. Note that $|r_h-r_v|\leq 1$, since three out of four ramified poles are shared by both covering maps. In particular, we can always choose $\W_1$ and $\W_2$ coherently such that $r_\si=r_{\si 1}\leq r_{\si 2}$, for both coverings. Furthermore, there is only one way to do this unless $r_h=r_v=2$. Note that with this setting, $r_h,r_v \in\{0,1,2\}$. Call $\W^\prime=\W_2$ and $\W_\si^\prime = \p_\si^{-1}\W^\prime$, and recall that ${\p_\si}_*:\pi_1(\W_\si)\to\pi_1(\W)$ is the pushforward of the projection $\p_\si:\W_\si\to\W$, which sends closed curves in $\W_\si$ to closed curves in $\W$. In particular, $b_\si = \# {\p_\si}_*^{-1}(\gamma)$ is the number of curves (connected components) in $\p_\si^{-1}(\gamma)$, and $b_\si\in\{1,2\}$, since $\p_\si$ is a double cover.

\begin{rema} \label{monodromy}
In particular, the number $b_\si$ corresponds to the number of boundary components of the surface $\W_\si^\prime$. This number also defines the monodromy of the core curve of $C$, $\gamma$, for $\p_\si$. In fact, $b_\si=2$ means that $\gamma$ has two ${\p_\si}_*$-preimages and, since $\p_\si$ is a double cover, this gives trivial monodromy. While non trivial monodromy, and equals to $\Z_2$, arises when $b_\si=1$.
\end{rema}

\begin{lemm} \label{parity} Let $C$ be a cylinder in $\W$, $\gamma$ its core curve and consider $b_\si = \# {\p_\si}_*^{-1}(\gamma)$. Then, $b_\si = 4 - r_\si - 2g(\W_\si^\prime)$. In particular, $b_\si \equiv r_\si \mod 2$.
\end{lemm}
\begin{proof}
Clearly, $\W^\prime$ has one boundary component, which is equal $\gamma$. Note that $b_\si$ is the number of boundary components of $\W_\si^\prime$, $b_\si = \# {\p_\si}_*^{-1}(\gamma) \in\{1,2\}$.

In $\W^\prime$, there are $4-r_\si$ ramified and $m-(q - r_\si + 2)$ unramified poles for $\p_\si$, and $m-q$ simple zeros.
Thus, we have $2(m-q+2-r_h)$ poles and $2(m-q)$ simple zeros in $\W_\si^\prime$. But then,
\[4g(\W_\si^\prime)-4 = 2(m-q) - 2(m-q+2-r_\si) - 2b_\si,\]
That is, $b_\si = 4 - r_\si - 2g(\W_\si^\prime)$ and, in particular, $b_\si \equiv r_\si \mod 2$.
\end{proof}

\begin{prop} \label{good-liftings} Let $C$ be a cylinder in $\W$. Then $C$ lifts to good cylinders in $\X$ if and only if $r_h,r_v \in\{0,1\}$.
\end{prop}
\begin{proof}
Let $\gamma$ be the core curve of $C$. Then, we want to show that if $\gamma_\si \in {\p_\si}_*^{-1}(\gamma)$, $[\gamma_\si]=0$ if and only if $r_\si\neq 2$. Note that, since $g(\W_\si) = 1$, a ho\-mo\-lo\-gi\-cally trivial curve always cut the surface into a genus zero surface and a genus one surface.

As before, let $\W^\prime=\W_2$ and $\W_\si^\prime={\p_\si}^{-1} \W^\prime$.
By the previous lemma, we know that $\# {\p_\si}_*^{-1}(\gamma) = b_\si = 4 - r_\si - 2g(\W_\si^\prime)$, $b_\si \equiv r_\si \mod 2$. Then,
\begin{itemize}[leftmargin=*]
\item If $r_\si=0$, then $b_\si=2$ and $g(\W_\si^\prime) = 1$. That is, $\gamma$ has two ${\p_\si}_*$-preimages ($b_\si=2$) bounding a genus one surface ($g(\W_\si^\prime) = 1$) in $\W_\si$. But $g(\W_\si) = 1$, and therefore both ${\p_\si}_*$-preimages of $\gamma$ are ho\-mo\-lo\-gi\-cally trivial (see, e.g., Figure~\ref{p-lift-pocket-r-0} and Figure~\ref{p-lift-dumbbell-r-0}).
\item When $r_\si=1$, we have $b_\si=1$ and $g(\W_\si^\prime) = 1$. It follows that $\gamma$ has one ${\p_\si}_*$-preimage which is ho\-mo\-lo\-gi\-cally trivial (see, e.g., Figure~\ref{p-lift-pocket-r-1} and Figure~\ref{p-lift-dumbbell-r-1}).
\item Finally, if $r_\si=2$, then $b_\si=2$ and $g(\W_\si^\prime) = 0$. Therefore, $\gamma$ has two ${\p_\si}_*$-preimages and both together bounds each of two genus zero surfaces which form the whole surface $\Wh$ of genus one (see, e.g., Figure~\ref{p-lift-pocket-r-2} and Figure~\ref{p-lift-dumbbell-r-2}).
Then, both preimages of $\gamma$ are not ho\-mo\-lo\-gi\-cally trivial.
\end{itemize}
\end{proof}

Thus, we know which cylinders in $\W$ lift to good cylinders in $\X$. It remains to see how these cylinders lift, that is, the number of cylinders in $\X$ we obtain and their length.

\begin{figure}[h]
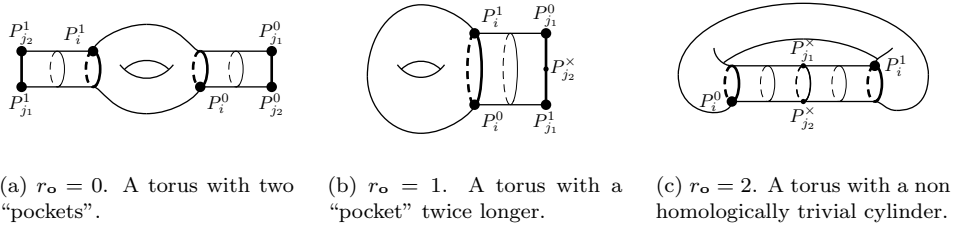

\begin{subfigure}[t]{.31\textwidth}
  \vcTikZ{\liftPocketRZero[.475]{0,0}}
  \caption{$r_\si=0$. A torus with two ``pockets''.}
  \label{p-lift-pocket-r-0}
\end{subfigure}
\hfill
\begin{subfigure}[t]{.31\textwidth}
  \vcTikZ{\liftPocketROne[.475]{0,0}}
  \caption{$r_\si=1$. A torus with a ``pocket'' twice longer.}
  \label{p-lift-pocket-r-1}
  \end{subfigure}
\hfill
\begin{subfigure}[t]{.31\textwidth}
  \vcTikZ{\liftPocketRTwo[.475]{0,0}}
  \caption{$r_\si=2$. A torus with a non ho\-mo\-lo\-gi\-cally trivial cylinder.}
  \label{p-lift-pocket-r-2}
\end{subfigure}
\caption{Possible liftings for $\p_\si$ of a pocket configuration.}
\label{p-lift-pocket}
\end{figure}

\begin{figure}[h]
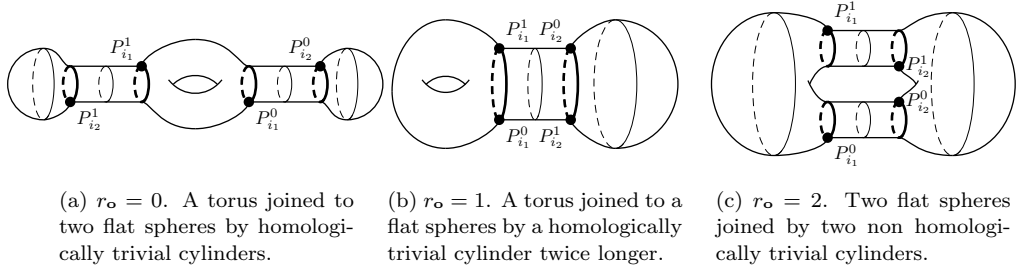

\begin{subfigure}[t]{.31\textwidth}
  \hspace*{-1ex}\vcTikZ{\liftDumbbellRZero[.475]{0,0}}
  \caption{$r_\si=0$. A torus joined to two flat spheres by ho\-mo\-lo\-gi\-cally trivial cylinders.}
  \label{p-lift-dumbbell-r-0}
\end{subfigure}
\hfill
\begin{subfigure}[t]{.31\textwidth}
  \vcTikZ{\liftDumbbellROne[.475]{0,0}}
  \caption{$r_\si=1$. A torus joined to a flat spheres by a ho\-mo\-lo\-gi\-cally trivial cylinder twice longer.}
  \label{p-lift-dumbbell-r-1}
\end{subfigure}
\hfill
\begin{subfigure}[t]{.31\textwidth}
  \vcTikZ{\liftDumbbellRTwo[.475]{0,0}}
  \caption{$r_\si=2$. Two flat spheres joined by two non ho\-mo\-lo\-gi\-cally tri\-vial cylinders.}
  \label{p-lift-dumbbell-r-2}
\end{subfigure}
\caption{Possible liftings for $\p_\si$ of a dumbbell configuration.}
\label{p-lift-dumbbell}
\end{figure}

\subsection{How cylinders in $\W$ lift to good cylinders in $\X$} \label{sect:how-good}
Here we show how lift to $\X$ those cylinders in $\W$ who lift to good cylinders in $\X$. More precisely, we determine the number of cylinders in $\X$ we obtain and their length. To do this, we will lift one by one the covering maps $\p_\si:\W_\si\to\W$, then $\Pp_\si:\X_\si\to\W_\si$ and finally $\P_\si:\X\to\X_\si$ (see Figure~\ref{diagram}). Recall we are using the label $\si$ instead of $h$ and $v$.

The following is a direct consequence of Remark~\ref{monodromy} and Lemma~\ref{parity}.

\begin{lemm} \label{monodromy-p}
Let $C$ be a cylinder in $\W$. Then, the core curve $\gamma$ of $C$ has trivial monodromy for $\p_\si$ if $r_\si\neq 1$, and equals to $\Z_2$, if $r_\si=1$.
\end{lemm}
\begin{proof}
    From Remark~\ref{monodromy}, we know that the number $b_\si$ defines the monodromy of $\gamma$, being trivial for $b_\si=2$ and equals to $\Z_2$ when $b_\si=1$.
    But, by Lemma~\ref{parity}, we also know that $b_\si\equiv r_\si\mod 2$, and $r_\si\in\{0,1,2\}$.
\end{proof}

The meaning of previous lemma can be noticed in Figure~\ref{p-lift-pocket} and Figure~\ref{p-lift-dumbbell}.

\begin{lemm} \label{monodromy-Pp} Let $C_\si$ be a cylinder in $\W_\si$ such that $r_\si(\p_\si(C_\si))\neq 2$. 
Then, the core curve of $C_\si$ has trivial monodromy for $\Pp_\si:\X_\si\to\W_\si$.
\end{lemm}
\begin{proof}
    Let $\gamma_\si$ be the core curve of $C_\si$.
    Since $r_\si(\p_\si(C_\si))\neq 2$, by Proposition~\ref{good-liftings} and Lemma~\ref{homology-good}, $\gamma_\si$ is ho\-mo\-lo\-gi\-cally trivial. Then, it cuts the surface $\W_\si$ in two components. Let $\W_\si^\backprime$ be one of these two components and consider $\X_\si^\backprime = \Pp_\si^{-1}\W_\si^\backprime$.
    
    Let $q^\backprime$ be the number of double zeros and $b^\backprime$, the number of boundary components, on $\X_\si^\backprime$. Then, $4g(\X_\si^\backprime)-4=4q^\backprime-2b^\backprime$, and $b^\backprime \equiv 0 \mod 2$. That is, $b^\backprime=2$ and $\gamma_\si$ has two $\Pp_{\si *}$-preimages. Since $\Pp_\si$ is a double cover, then $\gamma_\si$ has trivial monodromy.
\end{proof}

Thus, the possible $\Pp_\si$-liftings in the surface $\X_\si$ of a cylinder $C_\si$ in the surface $\W_\si$ (with $r_\si(\p_\si(C_\si))\neq 2$) looks like as in Figure~\ref{pp-lift-pocket} or Figure~\ref{pp-lift-dumbbell}. 

\begin{figure}[h]
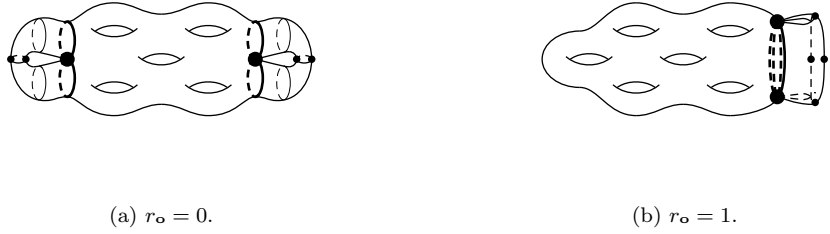

\begin{subfigure}[t]{.45\textwidth}
  \vcTikZ{\liftliftPocketRZero[.5]{0,0}}
  \caption{$r_\si=0$.}
\end{subfigure}
\hfill
\begin{subfigure}[t]{.45\textwidth}
  \vcTikZ{\liftliftPocketROne[.5]{0,0}}
  \caption{$r_\si=1$.}
\end{subfigure}
\caption{Possible $\Pp_\si$-liftings in $\X_\si$ of cylinders in $\W_\si$ coming from a pocket configuration in $\W$.}
\label{pp-lift-pocket}
\end{figure}

\begin{figure}[h]
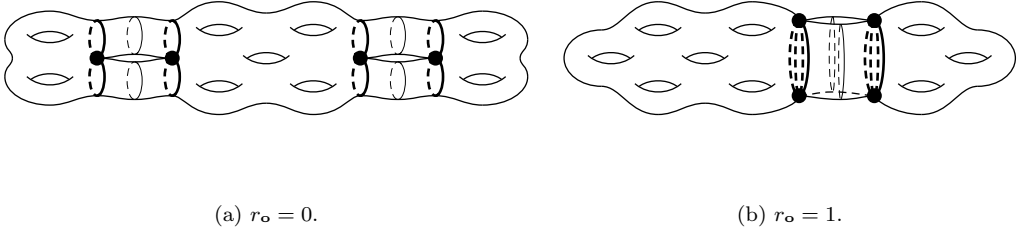

\begin{subfigure}[t]{.45\textwidth}
  \vcTikZ{\liftliftDumbbellRZero[.5]{0,0}}
  \caption{$r_\si=0$.}
\end{subfigure}
\hfill
\begin{subfigure}[t]{.45\textwidth}
  \vcTikZ{\liftliftDumbbellROne[.5]{0,0}}
  \caption{$r_\si=1$.}
\end{subfigure}
\caption{Possible $\Pp_\si$-liftings in $\X_\si$ of cylinders in $\W_\si$ coming from a dumbbell configuration in $\W$.}
\label{pp-lift-dumbbell}
\end{figure}

Finally, we can describe how cylinders in $\W$ lift to good cylinders in $\X$.
Recall that $\PP:\X\to\W$ is a covering of degree~$8$.

\begin{lemm} \label{monodromy-P}
Let $C$ be a cylinder in $\W$ and $\gamma$ be its core curve. Suppose that $r_h,r_v\in\{0,1\}$. Then,
\begin{enumerate}
  \item If $r_h=r_v=0$, then $\gamma$ has trivial monodromy for $\PP$. In particular, $\gamma$ has eight $\PP_*$-preimages of the same length than $\gamma$.
  \item In any other case, $\gamma$ has monodromy $\Z_2$ for $\PP$. In particular, $\gamma$ has four $\PP_*$-preimages twice longer than $\gamma$.
\end{enumerate}
\end{lemm}	
\begin{proof}
    Recall first that $\PP:\X\to\W$ is a covering of degree~$8$, $\PP=\p_\si\circ\Pp_\si\circ\P_\si$ and also $\PP=\phv\circ\Phv$, where $\Phv:\X\to\Xhv$ and $\phv:\Xhv\to\W$ (see the diagram in Figure~\ref{diagram} for a recall in notation).
\begin{enumerate}[leftmargin=*]
  \item Suppose $r_h=r_v=0$. By Lemma~\ref{monodromy-p}, we know that $\gamma$ has trivial monodromy for both $\ph$ and $\pv$. Then, by Lemma~\ref{monodromy-Pp}, we deduce that $\gamma$ has trivial monodromy for $\ph\circ \Pph$ and for $\pv\circ \Ppv$. Then, the monodromy of $\gamma$ for $\PP=\p_\si\circ\Pp_\si\circ\P_\si$ can be at most $\Z_2$, since $\P_\si:\X\to\X_\si$ is a double cover.
  
  Suppose it is $\Z_2$. Then, the monodromy for $\P_\si$ of the corresponding curves $\bar{\gamma_\si}_i$, $i=1,\dots,4$, in $\X_\si$ is $\Z_2$. This means, in particular, that $\tau_h$ and $\tau_v$ fix the corresponding curves $\bar\gamma_i$, $i=1,\dots,4$, in $\X$. Consider $D={\Phv}_*(\{\bar\gamma_i\}_{i=1}^4)$ and note that $D={\phv}_*^{-1}(\gamma)$. Then, since $\tau_h$ and $\tau_v$ fix each $\bar\gamma_i$, $i=1,\dots,4$, we have that $\# D=4$, but $\phv$ is a double cover, so this is impossible.
  Thus, assuming that the monodromy for $\PP$ of $\gamma$ is $\Z_2$, we get a contradiction. Therefore, the monodromy is trivial (see Figure~\ref{P-lift-pocket-r-0} and Figure~\ref{P-lift-dumbbell-r-0}).
  \item For the other cases, we will prove that $\gamma$ has monodromy $\Z_2$. Remember we are assuming that $r_h,r_v\neq 2$.
  \begin{enumerate}[leftmargin=*]
    \item Suppose $r_h=r_v=1$. From Lemma~\ref{monodromy-p} we know that $\gamma$ has monodromy $\Z_2$ for both $\ph$ and $\pv$. Then, by Lemma~\ref{monodromy-Pp}, we deduce that $\gamma$ has monodromy $\Z_2$ for $\ph\circ \Pph$ and for $\pv\circ \Ppv$. Then, the monodromy of $\gamma$ for $\PP=\p_\si\circ\Pp_\si\circ\P_\si$ can be $\Z_2$ or $\Z_4$, since $\P_\si$ is a double cover.
    
    Suppose it is $\Z_4$. Then, the monodromy for $\P_\si$ of the corresponding curves $\bar{\gamma_\si}_i$, $i=1,2$, in $\X_\si$ is $\Z_2$, and $\tau_h$ and $\tau_v$ fix each $\bar\gamma_i$, $i=1,2$ in $\X$. To continue with the argument, we need to remark first that $\tau_h$ and $\tau_v$ are orientation preserving isometric involutions. Then, when they fix a cylinder, the only way to do this is, either being the identity or a rotation by half the length of the cylinder, when restricted to the cylinder. In particular, ${\check\gamma}_i\coloneqq \PP(\bar\gamma_i)=\rquo{\bar\gamma_i}{\langle \tau_h,\tau_v\rangle}$ has at least half the length of $\bar\gamma_i$, $i=1,2$, that is, at least twice the length of $\gamma$. But $\check\gamma_i\in {\phv}_*^{-1}(\gamma)$, $i=1,2$, and $\phv$ is a double cover, so it is impossible to have two $\phv$-preimages of at least twice the length.
    Thus, assuming that the monodromy of $\gamma$ for $\PP$ is $\Z_4$, we get a contradiction. Therefore, the monodromy is $\Z_2$ (see Figure~\ref{P-lift-pocket-r-1} and Figure~\ref{P-lift-dumbbell-r-1}).
    \item Suppose that $r_h=0$ and $r_v=1$. Then, as before, we find that $\gamma$ has trivial monodromy for $\ph\circ \Pph$, and monodromy $\Z_2$ for $\pv\circ \Ppv$. Then, since $\Ph$ and $\Pv$ are double covers, $\gamma$ has trivial or $\Z_2$ monodromy for $\ph\circ \Pph\circ\Ph$ and monodromy $\Z_2$ or $\Z_4$ for $\pv\circ \Ppv\circ\Pv$. But $\ph\circ \Pph\circ \Ph=\pv\circ \Ppv\circ \Pv=\PP$, and therefore, the only alternative is to have monodromy equals to $\Z_2$ (see Figure~\ref{P-lift-pocket-r-0} and Figure~\ref{P-lift-dumbbell-r-0}). Analogously, we have monodromy $\Z_2$ for $r_h=1$ and $r_v=0$.
  \end{enumerate}
\end{enumerate}
\end{proof}

\begin{figure}[h]
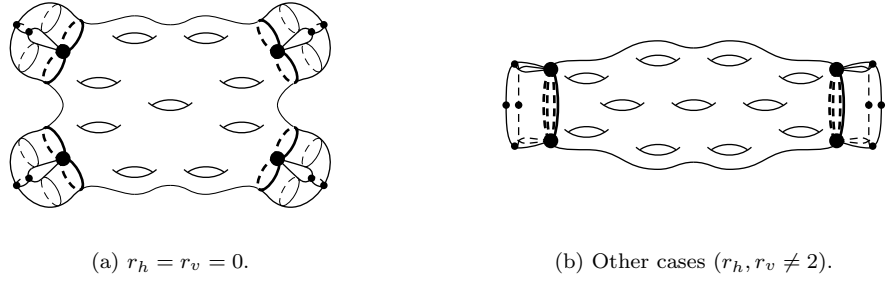

\begin{subfigure}[t]{.45\textwidth}
  \vcTikZ{\liftliftliftPocketRZero[.475]{0,0}}
  \caption{$r_h=r_v=0$.}
  \label{P-lift-pocket-r-0}
\end{subfigure}
\hfill
\begin{subfigure}[t]{.45\textwidth}
  \vcTikZ{\liftliftliftPocketROne[.475]{0,0}}
  \caption{Other cases ($r_h,r_v\neq 2$).}
  \label{P-lift-pocket-r-1}
\end{subfigure}
\caption{Lifting of a pocket configuration in $\W$ to $\X$.}
\label{P-lift-pocket}
\end{figure}

\begin{figure}[h]
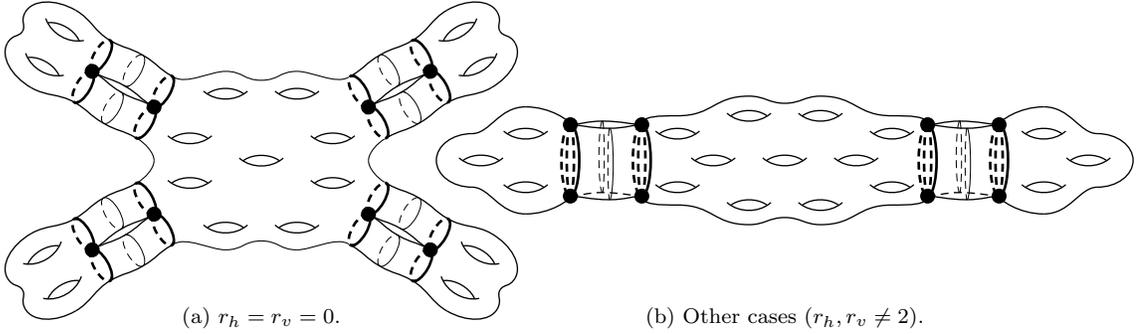

\begin{subfigure}[t]{.45\textwidth}
  \vcTikZ{\liftliftliftDumbbellRZero[.475]{0,0}}
  \vspace{3ex}
  \caption{$r_h=r_v=0$.}
  \label{P-lift-dumbbell-r-0}
\end{subfigure}
\hfill
\begin{subfigure}[t]{.45\textwidth}
  \vspace{3ex}
  \vcTikZ{\liftliftliftDumbbellROne[.475]{0,0}}
  \caption{Other cases ($r_h,r_v\neq 2$).}
  \label{P-lift-dumbbell-r-1}
\end{subfigure}
\caption{Lifting of a dumbbell configuration in $\W$ to $\X$.}
\label{P-lift-dumbbell}
\end{figure}

\subsection{Relation between Siegel-Veech constants in $\Q(1^m,-1^{m+4})$ and its lifting to $\M$} \label{sect:relation-SV}
We conclude the study of which and how cylinders in $\W$ lift to good cylinders in $\X$ by relating the Siegel-Veech constants of configurations in $\W$ and its liftings to $\X$.


Let $\L$ be an invariant affine submanifold in $\Q(1^m,-1^{m+4})$ and let $\mu$ be the associated affine invariant measure on $\L$. Consider the locus $\M$ of all possible $\PP$-covers surfaces from $\L$. Note that, by construction, this gives an $\slr$-equivariant one-to-one correspondence between $\L$ and $\M$. In particular, $\M$ is an affine invariant submanifold on $\H(2^{4m})$. Let $\nu$ be the affine invariant measure on $\M$. Note that that $\mu$ is the direct image of $\nu$ with respect to the projection $\M\to\L$.

Let $c=c_{\C}(\L)$ be the Siegel-Veech constant associated to the counting of a multiplicity one configuration $\C$ of cylinders in $\L$ (see \S~\ref{sect:configuration} for the definitions).
Then, the configuration $\C$ induces a cylinder configuration $\bar\C$ on the covering space $\M$, defined by the covering maps $\PP$. Let $\bar c=c_{\bar\C}(\M)$ be the associated Siegel-Veech constant. The lemma below relates $c$ and $\bar c$. It is the analogous of Lemma~1.1 in \cite{EKZ} and Lemma~4.1 in \cite{DZ}, adapted for our purposes.

We say that $\C$ is a \emph{pocket-like} configuration, if the singularities in one of the boundary components of the cylinder are only poles. Note that, in particular, there are exactly two poles in that boundary component. Denote by $r_h(\C)$ and $r_v(\C)$ the values of $r_h$ and $r_v$ in the cylinders defined by configuration $\C$. These values are well defined, since a configuration defines all that data. Call the pair $(r_h,r_v)$ the \emph{profile of the configuration $\C$}. We say that $\C$ is a \emph{good configuration} if it is a multiplicity one configuration of cylinders in $\L$ such that $r_h(\C),r_v(\C)\in\{0,1\}$.

\begin{lemm} \label{relation-SV-constants}
Let $\C$ be a good configuration.
\begin{enumerate}
\item If $\C$ is pocket-like, then
  \begin{enumerate}
  \item If $\C$ has profile $(0,0)$, then $\bar c = 32c$.
  \item In any other case, $\bar c= 4c$.
  \end{enumerate}
\item If $\C$ is not pocket-like, then
  \begin{enumerate}
  \item If $\C$ has profile $(0,0)$, then $\bar c = 64c$.
  \item In any other case, $\bar c= 8c$.
  \end{enumerate}
\end{enumerate}
\end{lemm}

\begin{proof}
First of all, suppose we know the exact number and the relative length of cylinders in $\X$ we obtain by lifting a cylinder from configuration $\C$ in $\W$.
Say, a cylinder from $\C$ in $\W$ is lifted to $n$ cylinders in $\X$ and their lengths are $s$ times the length of $\gamma$. Then, \[N_{\bar\C}\left(\X,L\right)=nN_{\C}\left(\W,s^{-1}L\right)\] and therefore,
\[\bar c = \frac{n}{s^2}\frac{\mathrm{Area}(\X)}{\mathrm{Area}(\W)}c = 8\frac{n}{s^2} c,\]
where we used the fact that $\mathrm{Area}(\X)=8\mathrm{Area}(\W)$, since $\X$ is a metric $8$-fold covering of $\W$. 
But we know, by Lemma~\ref{monodromy-P}, the exact number of $\PP_*$-preimages of the core curve of $C$, $\gamma$, and the relative length of these, depending on $r_h$ and $r_v$.

If $\C$ is not a pocket-like configuration, then, there is at least one singularity in each boundary of the cylinder in $\W$ which is not a pole. Then, for each $\PP_*$-preimage, $\bar\gamma$, of its core curve $\gamma$, there is a cylinder in $\X$ with core curve $\bar\gamma$ (see Figure~\ref{P-lift-dumbbell}). Thus, the values of $n$ and $s$ are given by Lemma~\ref{monodromy-P}. That is, $n=8$ and $s=1$ for profile $(0,0)$, and $n=4$, $s=2$, for all other profiles of good configurations.

In the case of pocket-like configurations, the poles defining the pocket-like configuration become regular points in the interior of the corresponding cylinders in $\X$ (see Figure~\ref{P-lift-pocket}) and, therefore, each cylinder in $\X$ has two $\PP_*$-preimages of $\gamma$ in its interior, instead of one, as in the case of non pocket-like configurations. Hence, the number $n$, of cylinders in $\X$ obtained by lifting a cylinder in $\W$ is half the number of $\PP_*$-preimages of $\gamma$, which is given by Lemma~\ref{monodromy-P}. 
That is, in the case of pocket-like configurations, we have that $n=4$ and $s=1$ for profile $(0,0)$, and $n=2$, $s=2$, for all other profiles of good configurations.

\end{proof}

\begin{rema} \label{rema:c-area}
If we were working with the area Siegel-Veech constant, instead of the classical Siegel-Veech constant, there would be no difference for pocket-like or not pocket-like configurations in the previous result, since area Siegel-Veech constant depends only on monodromy.
\end{rema}

\section{Siegel-Veech constants of good configurations for generic surfaces} \label{sect:SV-good}

In this section we use the results of the previous section to compute the exact value of the Siegel-Veech constant of good configurations for generic surfaces in $\Q(1^m,-1^{m+4})$ with respect to the Masur--Veech measure.

Recall that for almost every surface in $\L=\Q(1^m,-1^{m+4})$, the only possible configurations are pocket and dumbbell configurations. Note that both configurations are multiplicity one configurations, that is, they define a single cylinder.

By Proposition~\ref{good-liftings}, a multiplicity one configuration is a good configuration if and only if $r_h,r_v\in\{0,1\}$, where $r_h$ and $r_v$ are the number of ramified poles for $\p_h$ and $\p_v$, respectively, in a component of the surface $\W$ after cutting along the core curve of the cylinder defined by the configuration.
Lastly, recall that $\p_h$ and $\p_v$ have four ramified poles each, from which they share three. In particular, there are five ``special'' poles, the three shared ramified poles and one more for each one of $\p_h$ and $\p_v$. 

\subsection*{Good pocket configurations}  Recall that in a pocket configuration, we have a single cylinder bounded by a saddle connection joining a fixed pair of poles $P_{j_1},P_{j_2}$ on one side and by a separatrix loop emitted from a fixed zero $P_i$ of order $d_i \geq 1$, on the other side (see Figure~\ref{pocket}). Then, $r_h$ and $r_v$, as defined in the previous section, is the number of ramified poles among the poles $P_{j_1}$ and $P_{j_2}$ of the configuration for the double cover $\ph$ and $\pv$, respectively. By Proposition~\ref{good-liftings}, the configuration is good if and only if $r_h,r_v\in\{0,1\}$. Recall that the profile of the configuration is the pair $(r_h,r_v)$.

Profile $(0,0)$ means that none of the ramified poles, for $\ph$ and $\pv$, is one of the poles defining the pocket configuration, $P_{j_1}$ or $P_{j_2}$. Then, since there are $m-1=(m+4)-5$ poles which are unramified poles for both $\ph$ and $\pv$, there are exactly $\binom{m-1}{2}=(m-1)(m-2)/2$ pocket configurations of profile $(0,0)$.

In order to have profile $(1,1)$, we should have one ramified and one unramified pole for both $\ph$ and $\pv$, or one which is ramified for $\ph$ but unramified for $\pv$ and vice versa. This latter case occurs once, because $\ph$ and $\pv$ share three out of four of their ramified poles. The former case happens exactly $\binom{3}{1}\binom{m-1}{1}=3m-3$ times. Therefore, we have $3m-2$ pocket configurations of profile $(1,1)$.

Profile $(1,0)$ occurs when one of the poles is ramified for $\ph$ but unramified for $\pv$ and the other is unramified for both $\ph$ and $\pv$. Then, there are $\binom{1}{1}\binom{m-1}{1}=m-1$ pocket configurations of profile $(1,0)$. Similarly, we have $m-1$ pocket configurations of profile $(0,1)$.

Summarizing good profiles and applying Lemma~\ref{relation-SV-constants}, we get that good pocket configurations contribute to the Siegel-Veech constant of good cylinders in $\M$ by $c_{good}^\text{pocket}(\M)$, which is $16(m-1)(m-2)+4((3m-2)+2(m-1))$ times the Siegel-Veech constant for pocket configurations in $\L$. Thus, by formula~(\ref{SV-pocket}), \[c_{good}^\text{pocket}(\M)=\left(4m^2 - 7m + 4\right)\frac{2}{\pi^2}.\]

\subsection*{Good dumbbell configurations} Recall that in this configuration, we have a single cylinder, bounded by a saddle connection joining a zero to itself on each side (see Figure~\ref{dumbbell}). Such a cylinder separates the original surface $\W$ in two parts. This yields a partition of $\alpha=\{1^m,-1^{m+4}\}$ (where superindices stand for the multiplicities) into two subsets $\alpha=\alpha_1\sqcup \alpha_2$, which is also considered to be part of the configuration, and we consider $\alpha_1$ to contain the $r_h$ ramified poles for $\p_h$ and the $r_v$ ramified poles for $\p_v$. We stress in the fact that, even if there are several singularities with the same degree, we differentiate them, so they are named and, by a slight abuse of notation, we consider this information is also carried by the partition.

Let $k_l=\#\alpha_l$, counting multiplicities, $l=1,2$, and note that $k=k_1+k_2=2m+4$. Let $q$ be the number of simple zeros in $\alpha_1$. Then, there are $k_1-q$ poles in $\alpha_1$, but also, by topological considerations, we have that this number is equal to $q+2$, since we are restricted to a genus zero surface with one boundary component. Therefore, we will always have that $\alpha_1=\{1^q,-1^{q+2}\}$ and $\alpha_2=\{1^{m-q},-1^{m-q+2}\}$ (up to the names of the singularities). In particular, $k_1 = 2q+2$ and $k_2=2m-2q+2$. Thus, in this context, formula~(\ref{SV-dumbbell}) becomes
\begin{equation} \label{SV-dumbbell-q}
c^{\text{dumbbell}}_{i_1,i_2;\alpha_1,\alpha_2} = \frac{(2q-1)!(2m-2q-1)!}{(2m)!}\frac{2}{\pi^2}.
\end{equation}
Since this value depends only on $q$, it is natural to try to group configurations sharing this number $q$ and study the corresponding combinatorics. But, by Lemma~\ref{relation-SV-constants}, different profiles give different weights when lifted to $\M$. Hence, we have to consider different profiles separately.

For dumbbell configurations, profile $(0,0)$ means that there are only unramified poles in $\alpha_1$, that is, all the five ramified poles for $\ph$ and $\pv$, are in $\alpha_2$. Then, 
the combinatorics are given by the remaining $m-1$ poles and the $m$ simple zeros.

Hence, to compute the number of these configurations, that is, dumbbell configurations of profile $(0,0)$ with $q$ simple zeros in $\alpha_1$, we remark that we have to choose $q$ of the $m$ (named) simple zeros and $q+2$ of the remaining $m-1$ (named) poles, to have in total $q+2$ poles in $\alpha_1$, as required by the topology. Finally, note that we have to choose one of $q$ zeros to be located at the boundary of the cylinder on one side and one of $m-q$ zeros to be located at the boundary of the cylinder on the other side.
For any given $q$, where $1\leq q\leq m-1$, the count gives \[\binom{m}{q}\binom{m-1}{q+2}q(m-q)\] dumbbell configurations of profile $(0,0)$.

In order to have profile $(1,1)$, there are two possibilities. The first one is to have one simple pole in  $\alpha_1$ which is ramified for $\ph$ but unramified for $\pv$ and vice versa. In this case, there is only one choice for this two ramified poles, because $\ph$ and $\pv$ share three out of four of their ramified poles. The three ramified poles shared by $\ph$ and $\pv$ are then in $\alpha_2$. As before, we have to choose $q$ of the $m$ simple zeros to be in $\alpha_1$, one of them to be in a boundary component of the cylinder and one of the remaining $m-q$ simple zeros to be in the other boundary component. For poles, since we have already taken two poles to be in $\alpha_1$, we have to choose $q$ poles among the $m-1$ unramified poles, to have $q+2$ poles in total, as required by the topology. Then, this case of profile $(1,1)$ occurs $\binom{m}{q}\binom{m-1}{q}q(m-q)$ times.

The other case which gives profile $(1,1)$ is when there is only one ramified pole for both $\ph$ and $\pv$ in $\alpha_1$ and all the remaining ramified poles (for $\ph$ or $\pv$) are in $\alpha_2$. Thus, there are $3$ possibilities in choosing the common ramified pole and therefore, by an analogous computation, this case happens $\binom{m}{q}\binom{3}{1}\binom{m-1}{q+1}q(m-q)$ times. Then, for fixed $q$, $1\leq q\leq m-1$, we have \[\binom{m}{q}\left[3\binom{m-1}{q+1}+\binom{m-1}{q}\right]q(m-q)\] dumbbell configurations of profile $(1,1)$.

Profile $(1,0)$ occurs when only one of the poles in $\alpha_1$ is ramified for $\ph$ but unramified for $\pv$ and all others are unramified for both $\ph$ and $\pv$. Then, by an analogous computation, there are $\binom{m}{q}\binom{1}{1}\binom{m-1}{q+1}q(m-q)$ dumbbell configurations of profile $(1,0)$. Similarly, we have \[\binom{m}{q}\binom{m-1}{q+1}q(m-q)\] dumbbell configurations of profile $(0,1)$.

In summary, by Lemma~\ref{relation-SV-constants}, good dumbbell configurations contribute to the Siegel-Veech constant of good cylinders in $\M$ by
\[\textstyle\binom{m}{q}\left[64\binom{m-1}{q+2} + 8\left(3\binom{m-1}{q+1}+\binom{m-1}{q}+2\binom{m-1}{q+1}\right)\right]q(m-q)\]
times the Siegel-Veech constant for a dumbbell configurations in $\L$ with $q$ simple zeros in $\alpha_1$, that is,
\[c_{q,good}^\text{dumbbell}(\M)=8\binom{m}{q}\left[8\binom{m-1}{q+2} + 5\binom{m-1}{q+1}+\binom{m-1}{q}\right]q(m-q)c_q^\text{dumbbell},\]
where $c_q^\text{dumbbell}$ is given by formula~(\ref{SV-dumbbell-q}).
Finally, summing up all the contribution of good dumbbell configurations and plugging in formula~(\ref{SV-dumbbell-q}), we obtain that
\begin{align}
c_{good}^\text{dumbbell}(\M) & = 8\sum_{q=1}^{m-1}\textstyle\binom{m}{q}\left[8\binom{m-1}{q+2} + 5\binom{m-1}{q+1}+\binom{m-1}{q}\right]q(m-q)\frac{(2q-1)!(2m-2q-1)!}{(2m)!}\frac{2}{\pi^2} \nonumber \\
  & = 8\sum_{q=1}^{m-1}\textstyle\binom{m}{q}\left[8\binom{m-1}{q+2} + 5\binom{m-1}{q+1}+\binom{m-1}{q}\right]\frac{1}{4}\frac{(2q)!(2m-2q)!}{(2m)!}\frac{2}{\pi^2}\nonumber \\
\label{eq-dumbbell}  & = \frac{4}{\pi^2}\sum_{q=1}^{m-1}\frac{\binom{m}{q}}{\binom{2m}{2q}}\textstyle\left[8\binom{m-1}{q+2} + 5\binom{m-1}{q+1}+\binom{m-1}{q}\right].
\end{align}

But, by Proposition~\ref{combinatorial-identities}, formula~(\ref{eq-dumbbell}) can be written as
\begin{align*}
c_{good}^\text{dumbbell}(\M) & = \frac{4}{\pi^2}\sum_{q=1}^{m-1}\frac{\binom{m}{q}}{\binom{2m}{2q}}\textstyle\left[8\binom{m-1}{q+2} + 5\binom{m-1}{q+1}+\binom{m-1}{q}\right] \\
  & = \frac{4}{\pi^2}\textstyle\left[8\left(\frac{1}{6}m^2 -\frac{13}{6}m -3 +\frac{5}{2}4^m\frac{(m!)^2}{(2m)!}\right)\right. \\
  & \qquad\textstyle\left. {}+ 5\left(m +2 -\frac{3}{2}4^m\frac{(m!)^2}{(2m)!}\right)+ \left(-1 +\frac{1}{2}4^m\frac{(m!)^2}{(2m)!}\right)\right] \\
  & = \frac{2}{3\pi^2}\textstyle\left[8\left(m^2 -13m -18 +15\cdot 4^m\frac{(m!)^2}{(2m)!}\right)\right. \\
  & \qquad\textstyle\left. {}+ 5\left(6m + 12 -9\cdot 4^m\frac{(m!)^2}{(2m)!}\right)+ \left(-6 + 3\cdot 4^m\frac{(m!)^2}{(2m)!}\right)\right] \\
  & = \frac{2}{3\pi^2}\textstyle\left(8m^2 - 74m - 90 + 78\cdot 4^m\frac{(m!)^2}{(2m)!}\right).
\end{align*}

We conclude the computation of the Siegel-Veech constant for good cylinders in $\M$, for generic surfaces, summing up the contribution of pocket and dumbbell good configurations
\begin{align}
c_{good}(\M) & = c_{good}^\text{pocket}(\M) + c_{good}^\text{dumbbell}(\M) \nonumber \\
 & = \left(4m^2-7m+4\right)\frac{2}{\pi^2} + \left(8m^2 - 74m - 90 + 78\cdot 4^m\frac{(m!)^2}{(2m)!}\right)\frac{2}{3\pi^2} \nonumber \\
 & = \left(20m^2 - 95m - 78 + 78\cdot 4^m\frac{(m!)^2}{(2m)!}\right)\frac{2}{3\pi^2}.
\end{align}

\section{Side results} \label{sect:side-results}

\subsection{Area Siegel-Veech constant}

Following the same treatment, we can deduce that
for almost every wind-tree billiard $\Pi\in\WT(m)$, the number $N_{area}(L,\Pi)$ has quadratic asymptotic growth rate and 
\[N_{area}(L,\Pi)\sim \frac{1}{4} c_{a,good}(\M)\frac{\pi L^2}{\mathrm{Area}\left(\Pi/\Z^2\right)},\]
where $c_{a,good}(\M)$ is the area Siegel-Veech constant associated to the counting problem of the area of good cylinders in $\M$, the $\slr$-orbit closure of $\X(\Pi)$.

Moreover, 
for almost every wind-tree billiard $\Pi\in\WT(m)$, 
\[c_{a,good}(\M)=c_{a,good}^{\text{pocket}}(\M)+c_{a,good}^{\text{dumbbell}}(\M),\]
where $c_{a,good}^{\text{pocket}}(\M)$ (resp. $c_{a,good}^{\text{dumbbell}}(\M)$) corresponds to the area Siegel--Veech constant associated to configurations of good cylinders in $\M$ 
which project to pocket (resp. dumbbell) configurations in $\Q(1^m,-1^{m+4})$.

Furthermore, there exist a relation between classical Siegel--Veech constants and area Siegel-Veech constants for configurations $\C$ of cylinders in $\L=\Q(1^m,-1^{m+4})$:
\[c_{a,\C}(\L) = \frac{1}{2m+1}c_\C(\L).\]
This is a consequence of a generalization of Vorobets formula~\cite[Theorem~1.6(b)]{Vo2}, proved by Athreya--Eskin--Zorich~\cite[Proposition~4.9]{AEZ} for any configuration of cylinders on any strata $\Q(d_1,\dots,d_k)$ of quadratic differentials on $\mathbb{CP}^1$.

Then, we can relate the Siegel--Veech constant on $\M$ with that of $\L$, using the analogous of Lemma~\ref{relation-SV-constants} (keeping in mind Remark~\ref{rema:c-area}).

Finally, we have\begin{align*}
c_{a,good}(\M) & = c_{a,good}^\text{pocket}(\M) + c_{a,good}^\text{dumbbell}(\M) \\
 & = \frac{1}{2m+1}\left(4m^2-7m+4\right)\frac{4}{\pi^2} \\
 & \qquad + \frac{1}{2m+1}\left(8m^2 - 74m - 90 + 78\cdot 4^m\frac{(m!)^2}{(2m)!}\right)\frac{2}{3\pi^2} \\
 & = \frac{1}{2m+1}\left(16m^2 - 58m - 33 + 39\cdot 4^m\frac{(m!)^2}{(2m)!}\right)\frac{4}{3\pi^2} \\
 & = \left(8m - 33 + 39\cdot 4^m\frac{(m!)^2}{(2m+1)!}\right)\frac{4}{3\pi^2}
\end{align*}

\subsection{Polynomial diffusion rate}
The main result of Delecroix--Hubert--Leli\`evre in~\cite{DHL} relates the polynomial diffusion rate on the classical model to the Lyapunov exponents of the subbundles $F^{+-}$ and $F^{-+}$. In this case, the polynomial diffusion rate is $2/3$ for \emph{every} wind-tree billiard in $\WT(1)$. This result was generalized by Delecroix--Zorich~\cite{DZ} for $m\geq 2$. However, in the general case, the value of the diffusion rate is also explicitly known but only for \emph{almost every} wind-tree billiard in $\WT(m)$ and numerically for some explicit examples (see~\cite[Remark~2]{DZ}).

The explicit values of the polynomial diffusion rate for \emph{all} wind-tree billiards in $\WT(m)$, $m\geq 2$, is still an open problem. However, an application of Forni's criterion for integer equivariant subbundles (Theorem~\ref{theo:Forni}) allows us to show that the relevant Lyapunov exponents is always positive, for every wind-tree billiard in $\WT(m)$, for all $m\geq 1$ (Corollary~\ref{coro:hypothesis}).

Thus, we can conclude that we have always positive polynomial diffusion rate.

\subsection{Recurrence}

A geometric criterion for the recurrence of the directional linear flow on $\Z^d$-periodic flat surfaces in terms of good cylinders by Avila--Hubert~\cite{AH} says that if the positive $g_t$-orbit of the compact surface accumulates on a flat surfaces with a vertical good cylinder, then the vertical linear flow on the $\Z^d$-periodic flat surface is recurrent (\cite[Proposition~2]{AH}).

A result of Chaika--Eskin~\cite{CE} allows us to extend this criterion. In fact, we have the following.

\begin{theo} \label{theo:recurrence-criterion}
Let $X$ be a flat surface, $\M$ its $\slr$-orbit closure and $F$ a continuous equivariant subbundle. Let $\mathbf{f}$ be a $d$-tuple of elements in $F_X(\Z)$ and consider $X_\infty$, the $\Z^d$-periodic flat surface defined by $X$ and $\mathbf{f}$. Suppose that there exists $Y\in\M$ with an $F$-good cylinder. Then, for almost every $\theta\in[0,2\pi)$, the linear flow in direction $\theta$ is recurrent on $X_\infty$.
\end{theo}
\begin{proof}
By \cite[Theorem~1.1]{CE}, for almost every $\theta\in[0,2\pi)$, $r_{-\theta} X$ is Birkhoff generic for the $g_t$-flow with respect to $\nu_\M$. Since $Y\in\M$ has a $F$-good cylinder, then $Y'=r_\phi Y$ has a vertical cylinder for some $\phi\in[0,2\pi)$. Obviously $Y'\in\M$ and, since $r_{-\theta} X$ is Birkhoff generic, its positive $g_t$-orbit accumulates on $Y'$.
Then, by \cite[Proposition~2]{AH}, the linear flow in direction $\theta$ is recurrent in $X_\infty$.
\end{proof}

Thus, to prove the recurrence of every wind-tree billiard $\Pi\in\WT(m)$, we shall show that we can find good cylinders in the compact surface $\X(\Pi)$.

For $m=1$ this was first proved by Avila--Hubert~\cite[Lemma~4]{AH}. Consider $m\geq 2$ and recall that the obstacles of a wind-tree billiard $\Pi\in\WT(m)$ are horizontal and vertically symmetric right-angled polygons with $4m$ corners with the angle $\pi/2$ and $4(m-1)$, with the angle $3\pi/2$.

\begin{figure}[h]
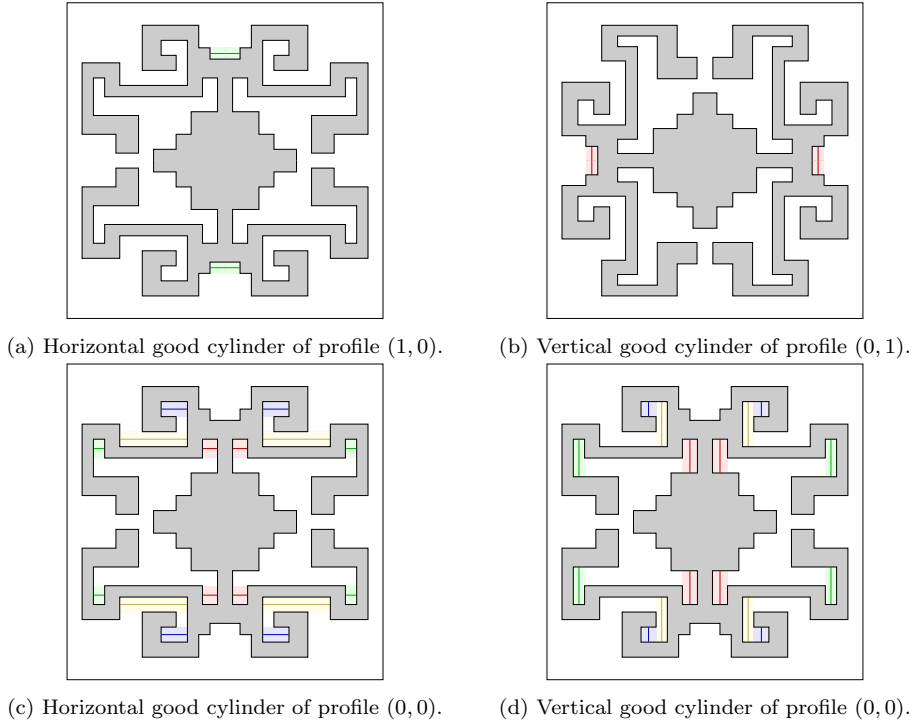

\begin{subfigure}[t]{.495\textwidth}
  \begin{centering}
  \tikz{\goodCylinderOneZero{0,0}}
  \caption{Horizontal good cylinder of profile $(1,0)$.}
  \label{figu:good-10}
  \end{centering}
\end{subfigure}
\begin{subfigure}[t]{.495\textwidth}
  \begin{centering}
  \tikz{\goodCylinderZeroOne{0,0}}
  \caption{Vertical good cylinder of profile $(0,1)$.}
  \label{figu:good-01}
  \end{centering}
\end{subfigure}
\begin{subfigure}[t]{.495\textwidth}
  \begin{centering}
  \tikz{\goodCylinderZeroZeroH{0,0}}
  \caption{Horizontal good cylinder of profile $(0,0)$.}
  \label{figu:good-00-h}
  \end{centering}
\end{subfigure}
\begin{subfigure}[t]{.495\textwidth}
  \begin{centering}
  \tikz{\goodCylinderZeroZeroV{0,0}}
  \caption{Vertical good cylinder of profile $(0,0)$.}
  \label{figu:good-00-v}
  \end{centering}
\end{subfigure}
\caption{Good cylinders for obstacles with two consecutive corners with angle $3\pi/2$.}
\end{figure}

If the obstacle has two consecutive angles $3\pi/2$, then we have (horizontal or vertical) good cylinders of profile $(1,0)$, $(0,1)$ or $(0,0)$. In fact, if the two consecutive angles are symmetric with respect to the vertical reflection, then we obtain horizontal good cylinders of profile $(1,0)$ as in Figure~\ref{figu:good-10}. Similarly, if the angles are symmetric with respect to the horizontal reflection, then we have vertical good cylinders of profile $(0,1)$ as in Figure~\ref{figu:good-01}. In other case, we obtain horizontal or vertical good cylinders of profile $(0,0)$ as in Figure~\ref{figu:good-00-h} and Figure~\ref{figu:good-00-v}.

\begin{figure}[h]
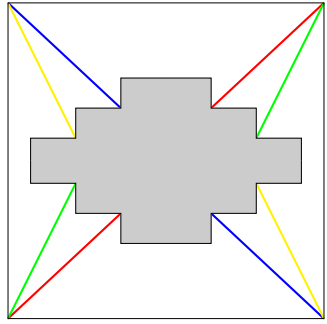

\centering
  \tikz{\goodCylinderOneOne{0,0}}
\caption{Core curves of good cylinders of profile $(1,1)$ for obstacles with no consecutive corners with angle $3\pi/2$.}
\label{figu:good-11}
\end{figure}

If there are no consecutive corners of angles $3\pi/2$, then there are good cylinders of profile $(1,1)$ as in Figure~\ref{figu:good-11}.

Thus, for every $\Pi\in\WT(m)$ we can exhibit good cylinders in $\X(\Pi)$ and then, by Theorem~\ref{theo:recurrence-criterion}, we conclude that the billiard flow in direction $\theta$ is recurrent for almost every $\theta\in[0,2\pi)$.

\appendix
\section{Combinatorial identities} \label{appe:combinatorial-identities}
In this appendix we prove the following identities.
\begin{prop} \label{combinatorial-identities} For any $m\in\N$ the following identities hold
\begin{alignat}{5}
\label{sum-q+2} \sum_{q=1}^{m-1}\frac{\binom{m}{q}\binom{m-1}{q+2}}{\binom{2m}{2q}} & = \frac{1}{6}m^2 & -\frac{13}{6}m & -3 & +\frac{5}{2}4^m\frac{(m!)^2}{(2m)!} \\
\label{sum-q+1} \sum_{q=1}^{m-1}\frac{\binom{m}{q}\binom{m-1}{q+1}}{\binom{2m}{2q}} & =  & m & +2 & -\frac{3}{2}4^m\frac{(m!)^2}{(2m)!} \\
\label{sum-q} \sum_{q=1}^{m-1}\frac{\binom{m}{q}\binom{m-1}{q}}{\binom{2m}{2q}} & = & & -1 & +\frac{1}{2}4^m\frac{(m!)^2}{(2m)!}
\end{alignat}
\end{prop}
\begin{proof}
Define \[B(m,s)\coloneqq \sum_{q=1}^{m-1}\frac{\binom{m}{q}\binom{m-1}{q+s}}{\binom{2m}{2q}}\]
and note that \[\frac{\binom{m}{q}\binom{m-1}{q+s}}{\binom{2m}{2q}}=\frac{m!(m-1)!}{(2m)!}\binom{2q}{q}\binom{2m-2q}{m-q}\frac{q!}{(q+s)!}\frac{(m-q)!}{(m-1-q-s)!}.\]
Consider \[A(m,s) =  \sum_{q=0}^{m} \binom{2q}{q}\binom{2m-2q}{m-q}\frac{q!}{(q+s)!}\frac{(m-q)!}{(m-1-q-s)!}.\]
Then
\begin{equation}\label{equa:B-A}
B(m,s) = \frac{m!(m-1)!}{(2m)!} A(m,s) - \binom{m-1}{s}.
\end{equation}

Note now than we can write \[\frac{(m-q)!}{(m-1-q-s)!} = \prod_{i=0}^{s} (m-q-i) \eqqcolon P^{(m,s)}(q),\] where $P^{(m,s)}$ is a computable polynomial of degree~$s+1$, and suppose \[P^{(m,s)}(q)=\sum_{j=0}^{s+1} p^{(m,s)}_j q^{j}.\]
Then, we can write
\begin{equation*}
A(m,s) = \sum_{j=0}^{s+1} p^{(m,s)}_j \sum_{q=0}^{m} \binom{2q}{q}\binom{2m-2q}{m-q}\frac{q!}{(q+s)!} q^{j}
\end{equation*}
and define \[D(m,s,j)=\sum_{q=0}^{m} \binom{2q}{q}\binom{2m-2q}{m-q}\frac{q!}{(q+s)!} q^{j},\]
so that
\begin{equation}\label{equa:A-D}
A(m,s) = \sum_{j=0}^{s+1} p^{(m,s)}_j D(m,s,j).
\end{equation}

Note that
{\small\begin{align*}
D(m,s,j) & =\sum_{q=0}^{m} \binom{2q}{q}\binom{2m-2q}{m-q}\frac{q!}{(q+s)!} q^{j} \\
 & =\sum_{q=0}^{m} \binom{2q}{q}\binom{2m-2q}{m-q}\frac{q!}{(q+s)!} q^{j} \frac{q+s+1}{q+s+1} \\
 & =\sum_{q=0}^{m} \binom{2q}{q}\binom{2m-2q}{m-q}\frac{q!}{(q+s+1)!} q^{j} (q+s+1) \\
 & = D(m,s+1,j+1) + (s+1)\; D(m,s+1,j).
\end{align*}}
Then, $D$ satisfies the following recurrence relation,
\begin{equation}\label{equa:recurrence-D}
  D(m,s,j) = D(m,s-1,j-1) - s\; D(m,s,j-1)
\end{equation}
and, in particular, we can deduce that $D(m,s,j)$ can be written as a linear combination of $D(m,i,0)$, $i=1,\dots,s$, and $D(m,0,l)$, $0\leq l\leq j-s$.
But, since $j$ takes values in $\{0,\dots,s+1\}$, for the $D(m,0,l)$ terms, we are interested only in $D(m,0,1)$ and $D(m,0,0)$. The value of $D(m,0,0)$ is given in \cite[(3.90)]{Go},
\begin{equation}\label{equa:D(m,0,0)}
D(m,0,0)=\sum_{q=0}^{m} \binom{2q}{q}\binom{2m-2q}{m-q}=4^m.
\end{equation}

On the other hand,
{\small\begin{align*}
D(m,0,1) & =\sum_{q=0}^{m} \binom{2q}{q}\binom{2m-2q}{m-q}q \\
 & =\sum_{r=0}^{m} \binom{2m-2r}{m-r}\binom{2r}{r}(m-r) \\
 & = m\; D(m,0,0) - D(m,0,1).
\end{align*}}
Then, $2\; D(m,0,1) = m\; D(m,0,0)$ and, by the identity~(\ref{equa:D(m,0,0)}),
\begin{equation}\label{equa:D(m,0,1)}
D(m,0,1)=\frac{m}{2}4^m.
\end{equation}
\begin{rema} In fact, it is not difficult to show that $D(m,0,l)=(m/2)^l\; 4^m$, $l\geq 0$.
\end{rema}

For the other terms, of the form $D(m,i,0)$, we use the following identity (\cite[(3.95)]{Go})
\begin{equation}\label{equa:3.95}
\mathcal{X}(m,i)\coloneqq \sum_{q=0}^{m}\binom{2q}{q}\binom{2m-2q}{m-q} \frac{i}{q+i} = \frac{\binom{2m+2i-1}{m+i}}{\binom{2i-1}{i}}.
\end{equation}

But, a simple partial fraction decomposition gives
{\small\begin{equation*}
\frac{q!}{(q+i)!} =\prod_{j=1}^{i} \frac{1}{q+j}
  = \sum_{j=1}^{i} \frac{(-1)^{j-1}}{(j-1)!(i-j)!}\frac{1}{q+j}
  = \sum_{j=1}^{i} \frac{(-1)^{j-1}}{j!(i-j)!}\frac{j}{q+j}
\end{equation*}}
and thus,
\begin{equation}\label{equa:D(m,i,0)} 
D(m,i,0) =\sum_{j=1}^{i} \frac{(-1)^{j-1}}{(j)!(i-j)!}\mathcal{X}(m,j).
\end{equation}

\subsection*{Proof of identity~(\ref{sum-q})}

Following previous discution, $P^{(m,0)}(q)=m-q$ and then, by (\ref{equa:A-D}), we have that,
\[A(m,0) = m\; D(m,0,0) - D(m,0,1) = \frac{m}{2} 4^m,\]
where last equality comes from (\ref{equa:D(m,0,0)}) and (\ref{equa:D(m,0,1)}).
Finally, from~(\ref{equa:B-A}), we have that
\[B(m,0) = \frac{m!(m-1)!}{(2m)!} A(m,0) - \binom{m-1}{0} = \frac{1}{2}4^m\frac{(m!)^2}{(2m)!} - 1,\]
which is~(\ref{sum-q}).

\subsection*{Proof of identity~(\ref{sum-q+1})}

Note that $P^{(m,1)}(q)=m^2-m - (2m-1)q + q^2$. Then, by (\ref{equa:A-D}), we have that,
\[ A(m,1) = (m^2-m)\; D(m,1,0) - (2m-1)\; D(m,1,1) + D(m,1,2).\]
Using the recurrence rule~(\ref{equa:recurrence-D}), we have that
\begin{align*}
D(m,1,1) & = D(m,0,0) - D(m,1,0), \text{ and} \\
D(m,1,2) & = D(m,0,1) - D(m,1,1) = D(m,0,1) - D(m,0,0) + D(m,1,0).
\end{align*}

It follows that 
{\small\begin{align*}
A(m,1)
 & = (m^2-m + (2m-1) + 1)\; D(m,1,0) - (2m-1 +1)\; D(m,0,0) + D(m,0,1) \\
 & = (m^2 + m)\; D(m,1,0) - 2m\; D(m,0,0) + D(m,0,1).
\end{align*}}
By identity~(\ref{equa:D(m,i,0)}) for $i=1$, $D(m,1,0) = \mathcal{X}(1)$, and from~(\ref{equa:3.95}),
\[D(m,1,0) = \mathcal{X}(1) = \binom{2m+1}{m+1} = \frac{(2m+1)!}{m!(m+1)!}.\]
Therefore,
\begin{align*}
A(m,1) 
 & = (m^2 + m)\frac{(2m+1)!}{m!(m+1)!} - 2m\;4^m + \frac{m}{2} 4^m \\
 & = \frac{(2m+1)!}{m!(m-1)!} - \frac{3m}{2} 4^m,
\end{align*}
where we have also used (\ref{equa:D(m,0,0)}) and (\ref{equa:D(m,0,1)}).
Thus, from~(\ref{equa:B-A}),
\begin{align*}
B(m,1) &= \frac{m!(m-1)!}{(2m)!} A(m,1) - \binom{m-1}{1} \\
 & = \frac{m!(m-1)!}{(2m)!}\left(\frac{(2m+1)!}{m!(m-1)!} - \frac{3m}{2} 4^m\right) - (m-1)\\
 & = 2m+1 - \frac{3}{2} 4^m \frac{(m!)^2}{(2m)!} - (m-1)\\
 & = m + 2 - \frac{3}{2} 4^m \frac{(m!)^2}{(2m)!},
\end{align*}
which is~(\ref{sum-q+1}).

\subsection*{Proof of identity~(\ref{sum-q+2})} (For the sake of readability, we will omit $m$ from notation in this part.) From (\ref{equa:A-D}), we have that
\[A(2) =  p^{(2)}_0 D(2,0) + p^{(2)}_1 D(2,1) + p^{(2)}_2 D(2,2) + p^{(2)}_3 D(2,3),\]
where \[P^{(2)}(q)= \sum_{j=0}^{3} p^{(2)}_j q^j = (m^3-3m^2+2m) - (3m^2-6m+2)q + (3m-3)q^2 - q^3.\]

Using the recurrence rule~(\ref{equa:recurrence-D}), we have that
\begin{align*}
D(2,1) & = D(1,0) - 2\;D(2,0), \\
D(2,2) & = D(1,1) - 2\;D(2,1) \\
 & = D(0,0) - D(1,0) - 2\;(D(1,0) - 2\;D(2,0)) \\
 & = D(0,0) - 3\;D(1,0) + 4\;D(2,0), \text{ and} \\
D(2,3) & = D(1,2) - 2\;D(2,2) \\
 & = D(0,1) - D(1,1) - 2\;(D(0,0) - 3\;D(1,0) + 4\;D(2,0)) \\
 & = D(0,1) - D(0,0) + D(1,0) - 2\;D(0,0) + 6\;D(1,0) - 8\;D(2,0) \\
 & = D(0,1) - 3\;D(0,0) + 7\;D(1,0) - 8\;D(2,0).
\end{align*}
It follows that,
\begin{align*}
A(2)
 & = p^{(2)}_0 D(2,0) + p^{(2)}_1 D(2,1) + p^{(2)}_2 D(2,2) + p^{(2)}_3 D(2,3) \\
 & = p^{(2)}_3 D(0,1) + (p^{(2)}_2 - 3\;p^{(2)}_3) D(0,0) + (p^{(2)}_1 - 3\;p^{(2)}_2 + 7\;p^{(2)}_3) D(1,0) \\
 & \qquad + (p^{(2)}_0 - 2\;p^{(2)}_1 + 4\;p^{(2)}_2 - 8\;p^{(2)}_3) D(2,0) \\
 & = - D(0,1) + 3m\; D(0,0) + q^{(2)}_1 D(1,0) + q^{(2)}_2 D(2,0) \\
 & = \frac{5m}{2} 4^m + q^{(2)}_1 D(1,0) + q^{(2)}_2 D(2,0),
\end{align*}
where we have used (\ref{equa:D(m,0,0)}), (\ref{equa:D(m,0,1)}) and the values of $p^{(2)}_3=-1$ and $p^{(2)}_2 = 3m-3$. We have also defined $q^{(2)}_1 = p^{(2)}_1 - 3p^{(2)}_2 + 7p^{(2)}_3$ and $q^{(2)}_2=p^{(2)}_0 - 2p^{(2)}_1 + 4p^{(2)}_2 - 8p^{(2)}_3$.

Thus, by identity~(\ref{equa:D(m,i,0)}),
{\small\begin{align*}
A(2) 
 & = \frac{5m}{2} 4^m + q^{(2)}_1\mathcal{X}(1) + q^{(2)}_2\left(\mathcal{X}(1)-\frac{1}{2}\mathcal{X}(2)\right)\\
 & = \frac{5m}{2} 4^m + (q^{(2)}_1 + q^{(2)}_2)\mathcal{X}(1) - \frac{1}{2} q^{(2)}_2 \mathcal{X}(2) \\
 & = \frac{5m}{2} 4^m + (p^{(2)}_0 - p^{(2)}_1 + p^{(2)}_2 - p^{(2)}_3)\mathcal{X}(1) - \frac{1}{2}(p^{(2)}_0 - 2p^{(2)}_1 + 4p^{(2)}_2 - 8p^{(2)}_3)\mathcal{X}(2) \\
 & = \frac{5m}{2} 4^m + (m^3-m)\mathcal{X}(1) - \frac{1}{2}(m^3+3m^2+2m)\mathcal{X}(2).
\end{align*}}
Plugging in identity~(\ref{equa:3.95}), we obtain
{\small\begin{align*}
A(2)
 & = \frac{5m}{2} 4^m + (m^3-m)\binom{2m+1}{m+1} - \frac{1}{2}(m^3+3m^2+2m)\frac{\binom{2m+3}{m+2}}{\binom{3}{2}} \\
 & = \frac{5m}{2} 4^m + (m-1)m(m+1)\frac{(2m+1)!}{m!(m+1)!} - \frac{1}{6}m(m+1)(m+2)\frac{(2m+3)!}{(m+1)!(m+2)!} \\
 & = \frac{5m}{2} 4^m + \left((m-1) - \frac{1}{3}(2m+3)\right)\frac{(2m+1)!}{m!(m-1)!} 
 = \frac{5m}{2} 4^m + \frac{1}{3}(m-6)\frac{(2m+1)!}{m!(m-1)!}.
\end{align*}}
Finally, by~(\ref{equa:B-A}),
{\small\begin{align*}
B(2) &= \frac{m!(m-1)!}{(2m)!} A(2) - \binom{m-1}{2} \\
 & = \frac{m!(m-1)!}{(2m)!}\left(\frac{5m}{2} 4^m + \frac{1}{3}(m-6)\frac{(2m+1)!}{m!(m-1)!}\right) - \frac{1}{2}(m-1)(m-2)\\
 & = \frac{5}{2} 4^m \frac{(m!)^2}{(2m)!} + \frac{1}{3}(2m^2-11m-6) - \frac{1}{2}(m^2-3m+2)\\
 & = \frac{5}{2} 4^m \frac{(m!)^2}{(2m)!} + \frac{1}{6}(m^2 -13m - 18),
\end{align*}}
which is~(\ref{sum-q+2}).

\end{proof}

\begin{rema} Note that the proof of Theorem~\ref{combinatorial-identities} states a procedure or algorithm in order to compute $A(m,s)$ and $B(m,s)$ for all $s\geq 0$. Anyway, an algorithm is not a formula, and evidently, the complexity increase enormously when $s$ becomes larger. However, with this method, it is possible to show that $B(m,s)$ has the form \[(2m+1)\; \mathcal{P}_s(m) + (-1)^{s}\;\textstyle\frac{2s+1}{2} \displaystyle \; 4^m \frac{(m!)^2}{(2m)!} - \textstyle\binom{m-1}{s},\] where $\mathcal{P}_s$ is a polynomial of degree~$s-1$ (in particular, $\mathcal{P}_0=0$), which can also be explicitly computed. Moreover, $\mathcal{P}_s$ can be deduced from the fact that $B(m,s)=0$ for $m=1,\dots,s+1$. In particular,
\[\mathcal{P}_s(m) = (-1)^{s+1}\;\textstyle\frac{2s+1}{2} \displaystyle \; 4^m \frac{(m!)^2}{(2m+1)!}\]
for  $m=1,\dots,s$. Anyway, we do not perform the computations here.
\end{rema}

\bibliographystyle{acm}

\end{document}